\documentclass[a4paper,reqno]{amsart}
\usepackage{amscd, amssymb, pb-diagram, amsmath}
\usepackage{amsthm}
\usepackage{slashed,hyperref}

\usepackage{tikz}
 \usepackage[foot]{amsaddr}

\usepackage{tikz-cd}
\usepackage[margin=3.7cm]{geometry}
\usepackage{color}
\usepackage{tensor}
\usetikzlibrary{arrows,matrix,backgrounds,decorations.markings}
\numberwithin{equation}{section}

\usepackage[all]{xy}
\usepackage[textsize=tiny,color=green!20!white]{todonotes}

\DeclareFontFamily{OT1}{pzc}{}
\DeclareFontShape{OT1}{pzc}{m}{it}{<-> s * [1.2] pzcmi7t}{}
\DeclareMathAlphabet{\mathpzc}{OT1}{pzc}{m}{it}

\def\End{\operatorname{End}}
\def\Hom{\operatorname{Hom}}
\def\Dom{\operatorname{Dom}}

\def\ker{\operatorname{Ker}}

\def\dim{\operatorname{dim}}

\def\rd{\operatorname{d}\!}

\def\C{\mathbb{C}}

\def\R{\mathbb{R}}
\def\N{\mathbb{N}}
\def\Z{\mathbb{Z}}
\def\T{\mathbb{T}}

\newtheorem{thm*}{Theorem}
\newtheorem{thm}{Theorem}[section]

\newtheorem{cor}[thm]{Corollary}

\newtheorem{lemma}[thm]{Lemma}
\newtheorem{prop}[thm]{Proposition}

\newtheorem{thm1}{Theorem}
\newtheorem{cor1}[thm1]{Corollary}

\theoremstyle{definition}
\newtheorem{definition}[thm]{Definition}

\theoremstyle{remark}
\newtheorem{remark}[thm]{Remark}

\newtheorem{example}[thm]{Example}

\begin{document}

\date{\today}

\title[Poisson transforms and the Baum-Connes conjecture]{Sharp mapping properties of Poisson transforms and the Baum-Connes conjecture}

\author{Heiko Gimperlein$^*$, Magnus Goffeng$^\dagger$}
\address[$^*$]{Engineering Mathematics, Universit\"{a}t Innsbruck, Technikerstraße 13, 6020 Innsbruck, Austria}
\address[$^\dagger$]{Centre for Mathematical Sciences, Lund University, Box 118, SE-221 00 Lund, Sweden}

\email{heiko.gimperlein@uibk.ac.at, magnus.goffeng@math.lth.se}

\begin{abstract}
We prove a sharp, quantitative analogue of Helgason's conjecture at the level of distributions: For a semisimple Lie group $G$ of real rank one, Poisson transforms map a Sobolev space on $P\backslash G$ boundedly with closed range to an $L^2$-space on $K\backslash G$. The result is obtained for the Poisson transform studied by Knapp-Wallach under the name Szegö map, and the appropriate Sobolev spaces are defined using van Erp-Yuncken's Heisenberg calculus. The proof generalizes to show that commutators of this Poisson transform with smooth functions on the Furstenberg compactification are compact. This proves the remaining open conjecture in Julg's seminal program to establish the Baum-Connes conjecture for closed subgroups of semisimple Lie groups of real rank one.
\end{abstract}

\maketitle

\section{Introduction}

Let $G$ be a connected, semisimple Lie group with maximal compact subgroup $K$ and minimal parabolic subgroup $P$. The Helgason conjecture \cite{helgason70} states that the Poisson transform defines an isomorphism between the space of joint eigenfunctions for the algebra of invariant differential operators on the symmetric space $K\backslash G$ and their boundary values, which are hyperfunction sections of a line bundle over the Furstenberg boundary $P\backslash G$. The Helgason conjecture was proven by Kashiwara et al.~\cite{kashi78}, and we refer to \cite{gkks,ks94,schmidbvp} for alternative approaches and independent proofs. In this paper we consider a quantitative analogue of Helgason's conjecture at the level of distributions. Motivated by Julg's program to prove the Baum-Connes conjecture for closed subgroups of Lie groups of real rank one \cite{julghow}, we consider sharp Sobolev mapping properties for a particular Poisson transform studied by Knapp-Wallach \cite{knappwallach}, called the Szegö map. 

More precisely, as a main result we prove that in real rank one the Szegö map produces an isomorphism of Hilbert spaces from a closed subspace of a Heisenberg-Sobolev space on $P\backslash G$ to a discrete series subrepresentation on $K\backslash G$. Our method is general for real rank one, and we discuss how technical developments of the Heisenberg calculus would extend our techniques to higher rank. 

We deduce implications for the Baum-Connes conjecture \cite{bch}, a conjecture in operator $K$-theory known to hold for large classes of groups, such as a-T-menable groups \cite{higsonkasp}, almost connected groups \cite{chabert03} and hyperbolic groups \cite{laffa,laffahyp}. A modern overview of the conjecture and its status can be found in \cite{julg19}. Notably, the Baum-Connes conjecture is to date not resolved for closed subgroups of semi-simple Lie groups, where already such an elementary case as $SL(3,\Z)$ is widely open.  The essentially only known technique of using the $\gamma$-element to study the Baum-Connes conjecture is obstructed in real rank $>1$ by Lafforgue's strong property (T) \cite{laffastrong}. Julg proposed a program \cite{julghow}  to prove the Baum-Connes conjecture with coefficients for all closed subgroups of Lie groups of real rank one. It would follow from two open conjectures, one of which was recently shown by Julg and Nishikawa \cite{julgnishi}.  The works  \cite{julghow,julgnishi} provide a general approach, with proofs given for $Sp(n,1)$. In this paper we prove the remaining conjecture in Julg's program for general Lie groups of real rank one: the commutator of the Poisson transform with a smooth function on the Furstenberg compactification is a compact operator. This follows by adapting the proof of the sharp mapping properties. 

\subsection{Analytic properties of the Szegö map}
Let us describe the setup and results in more detail. In the bulk of the paper, we restrict to connected, semisimple Lie  groups $G$ of real rank one. We tacitly assume that the rank of $G$ coincides with that of a  maximal compact subgroup $K$.  Fix a standard, minimal parabolic subgroup $P\subseteq G$. We can define $P$ from its Langlands decomposition $P=MAN$ where $A$ is defined from a maximal abelian subalgebra of the noncompact part in the Cartan decomposition $\mathfrak{g}=\mathfrak{k}+\mathfrak{p}$, $N$ is defined from the positive root spaces for a choice of positive restricted roots and $M$ is the centralizer of $A$ in $K$. It holds that $M=P\cap K$. For an irreducible $K$-representation $(\mathcal{V},\tau)$ with regular character and highest weight $\lambda$, we will write $(\mathcal{H},\sigma)\subseteq (\mathcal{V},\tau|_M)$ for the irreducible $M$-representation generated by the highest weight vector in $\mathcal{V}$. We assume that $A$ is defined from the non-compact roots from maximally compact Cartan subalgebra (as in \cite{knappwallach}), and define the character $\nu_{\mathcal{V}}$ on $A$ by 
$$\nu_{\mathcal{V}}(E_{\alpha_j}+E_{-\alpha_j}):=\frac{2\langle \lambda +n_j\alpha_j,\alpha_j\rangle}{\langle \alpha_j,\alpha_j\rangle},$$
as in \cite{knappwallach} where $n_j$ is defined in \cite[Equation (6.5.b)]{knappwallach}. We write $\mathbb{V}:=\mathcal{V}\times_K G\to K\backslash G$ and $\mathbb{H}_\nu:=\mathcal{H} \times_PG\to P\backslash G$ for the associated homogeneous vector bundles, here we extended the $M$-action on $\mathcal{H}$ to an $MA$-action via $\nu$ and to a $P$-action by letting the $N$-factor act trivially. The $G$-equivariant Poisson transform in this context is the Knapp-Wallach Szegö map defined as
\begin{equation}
\label{kwszeg}
S_{\mathcal{V}}:\mathcal{D}'(P\backslash G,\mathbb{H}_{\nu_{\mathcal{V}}})\to C^\infty(K\backslash G,\mathbb{V}), \quad S_{\mathcal{V}} f(g):=\int_K \tau(k)^{-1} f(kg)\rd k.
\end{equation}
The Poisson transform induces an infinitesimal intertwiner from the nonunitary principal series $L^2(P\backslash G,\mathbb{H}_{\nu_{\mathcal{V}}})$ onto the discrete series subrepresentation of $L^2(K\backslash G,\mathbb{V})$ with lowest $K$-type $\lambda$.

In this paper we study finer analytic mapping properties of the Szegö map $S_{\mathcal{V}}$ \eqref{kwszeg}, using techniques from microlocal geometric analysis. The relevance of such techniques in related problems on $K \backslash G$ was shown, for example, in \cite{gimppois,hansenhilg,mazzeovasy}. Our work focuses on the boundary $P\backslash G$. To phrase the results, we consider $P\backslash G$ as a filtered manifold and use the Heisenberg calculus. This calculus is well studied in the literature \cite{Dave_Haller1,ewertfixed,goffkuz,vanerpyunck}. In Sections \ref{lnlnadna} and \ref{subsecparatangb} below, we give an overview adapted to $P\backslash G$. 

To briskly describe the setup, the reader could think of $P\backslash G$ as modeled on Bruhat charts $V=\theta(N)$, for the Cartan involution $\theta$, and the Heisenberg calculus as consisting of operators that up to arbitrary regularity have Schwartz kernels given by sums of terms $k(x,xy^{-1})$ smooth in the first coordinate and homogeneous in the second for the $A$-action on $V$. By choosing an orthonormal basis for the simple root space in $V$, we obtain a sub-Laplacian type operator $\Delta_H$ whose complex powers belong to the Heisenberg calculus and the associated scale of Heisenberg-Sobolev spaces $W^s_H(P\backslash G):=(1+\Delta_H)^{-s/2}L^2(P\backslash G)$ plays a similar role in the Heisenberg calculus as the ordinary Sobolev spaces play in classical pseudodifferential theory. Similarly, one defines Heisenberg-Sobolev spaces $W^s_H(P\backslash G,\mathbb{E})$ for a vector bundle $\mathbb{E}\to P\backslash G$ and if $\mathbb{E}$ is homogeneous, $G$-translation induces a continuous $G$-action on $W^s_H(P\backslash G,\mathbb{E})$. 

Knapp-Wallach \cite{knappwallach} proved that the Szegö map $S_{\mathcal{V}}$ from Equation \eqref{kwszeg} is an infinitesimal intertwiner. The Casselman-Wallach globalization theorem \cite{krobern,cassglob,wallachasy} implies that there is an $s\in \R$ such that $S_{\mathcal{V}}$ extends to a continuous operator $S_{\mathcal{V}}: W^{s}_{H}(P\backslash G, \mathbb{H}_{\nu_{\mathcal{V}}})\to L^2(K \backslash G,\mathbb{V})$. The main result of this paper gives the optimal Sobolev exponent and shows that at this exponent the range is closed, i.e. $S_{\mathcal{V}}$ induces a $G$-equivariant isomorphism of Hilbert subspaces.

\begin{thm1}
\label{conja}
Assume that $G$ is a connected, semisimple Lie group of real rank one. Let $(\mathcal{V},\tau)$ be an irreducible $K$-representation with regular character and highest weight $\lambda$. We take $(\mathcal{H},\sigma)$, $A$ and $\nu_{\mathcal{V}}$ as above, and write $\nu_\mathcal{V}=s_\mathcal{V}\alpha$,
where $\alpha$ is the associated simple restricted root. The Szegö map \eqref{kwszeg} induces a continuous mapping with closed range
$$S_{\mathcal{V}}: W^{-s_{\mathcal{V}}}_{H}(P\backslash G, \mathbb{H}_{\nu_{\mathcal{V}}})\to L^2(K \backslash G,\mathbb{V}).$$
Moreover, the source projection $P_{\mathcal{V}}$ of $S_{\mathcal{V}}$ is a zeroth order operator in the Heisenberg calculus on $P\backslash G$ and $P_{\mathcal{V}}W^{-s_\mathcal{V}}_{H}(P\backslash G, \mathbb{H}_{\nu_{\mathcal{V}}})\cong W^{-s_\mathcal{V}}_{H}(P\backslash G, \mathbb{H}_{\nu_{\mathcal{V}}})/\ker(S_{\mathcal{V}})$ is $G$-equivariantly isomorphic as Hilbert spaces to the discrete series representation with lowest $K$-type $\lambda$.
\end{thm1}

Our proof of Theorem \ref{conja} is found in Section \ref{aldknadln}. The idea in the proof is to show that the formal adjoint $S_{\mathcal{V}}^\dagger$ in the $L^2$-pairing, expected to be a mapping
$$S_{\mathcal{V}}^\dagger:L^2(K \backslash G,\mathbb{V})\to [W^{-s_\mathcal{V}}_{H}(P\backslash G, \mathbb{H}_{\nu_{\mathcal{V}}})]^*=W^{s_\mathcal{V}}_{H}(P\backslash G, \mathbb{H}_{\nu_{\mathcal{V}}}^*),$$ 
can be composed with $S_{\mathcal{V}}$, after which a computation involving ideas of Ewert \cite{ewertfixed} implies that $S_{\mathcal{V}}^\dagger S_{\mathcal{V}}$ is the Knapp-Stein intertwiner which is of order $-2s_\mathcal{V}$ in the Heisenberg calculus on $P\backslash G$. The continuity follows, and closed range is then a matter of manipulations with Knapp-Stein's theory for intertwiners \cite{knappstein} proved in Lemma \ref{sourcedesc}. 

\begin{remark}
For simplicity, we can assume that $K$ acts unitarily on $W^{s}_{H}(P\backslash G, \mathbb{H}_{\nu_{\mathcal{V}}})$ for all $s$ in Theorem \ref{conja}. Then by decomposing over $K$-types we see that the source projection $P_{\mathcal{V}}$ of $S_{\mathcal{V}}$ is orthogonal in all of the Hilbert spaces $W^{s}_{H}(P\backslash G, \mathbb{H}_{\nu_{\mathcal{V}}})$. 
\end{remark}

\begin{remark}
A weaker version of Theorem \ref{conja}, in special cases without the closed range statement, was proven by Julg-Kasparov \cite{julgkasp} for $G=SU(n,1)$ and by Julg \cite{julghow} for $G=Sp(n,1)$, both results motivated by the study of the Baum-Connes conjecture. The aspect of relevance to this paper in Julg-Kasparov's work \cite{julgkasp} was refined by Ørsted-Zhang \cite{orschang05} who further studied the Szegö map for $G=SU(n,1)$ and by direct computation proved Theorem \ref{conja} in the case of $K$-subrepresentations $\mathcal{V}$ of $\wedge^n\mathfrak{p}$. Our approach to Theorem \ref{conja} does not rely on direct computation but rather on relating the quasi character on the $A$-factor to scaling properties pertinent to the Heisenberg calculus.
\end{remark}

\begin{remark}
\label{alknalndarank}
We remark that we are heavily relying on $G$ having real rank one in our usage of the Heisenberg calculus when proving Theorem \ref{conja}. In the case that $G$ is of higher rank, $P\backslash G$ is still a filtered manifold with a Heisenberg calculus. However, the operators of relevance arising from representation theory are homogeneous for the higher rank torus $A$ so they have the wrong regularity properties to fit into the Heisenberg calculus. It would be more suitable to model $P\backslash G$ as a \emph{multifiltered manifold} -- a notion studied in \cite[Section 3.4]{yunckenhab}, see also \cite{yuncken2013}. The space $P\backslash G$ is multifiltered by all its restricted roots. 

Unfortunately, no Heisenberg calculus has yet been developed for multifiltered manifolds even if the appropriate scale of Sobolev spaces has been discussed from a representation theory perspective by Bernstein \cite{bernsteinmulti}. It is conceivable that an analogue of Theorem \ref{conja} with a similar proof to ours holds in higher rank, when replacing the Heisenberg calculus and the Heisenberg-Sobolev scale by an appropriate analogue for multifiltered manifolds. But given Puschnigg's obstructions to finite summability in higher rank lattices \cite{puschnigg} and Lafforgue's strong property (T) \cite{laffastrong}, we can expect the higher rank case to be more complicated.
\end{remark}

For an $M$-subrepresentation $(\mathcal{H},\sigma)\subseteq (\mathcal{V},\tau|_M)$ and a quasi-character $\nu\in \mathfrak{a}_\C^*$, we can more generally consider the $G$-equivariant Poisson transform 
\begin{equation}
\label{kwszeggeneric}
P_\nu:\mathcal{D}'(P\backslash G,\mathbb{H}_{\nu})\to C^\infty(K\backslash G,\mathbb{V}), \quad P_\nu f(g):=\int_K \tau(k)^{-1} f(kg)\rd k.
\end{equation}
From the asymptotics of admissible representations \cite[Chapter VIII]{knappbook} or \cite[Chapter 4]{wallachbook}, we see that the Poisson transform induces an infinitesimal intertwiner from the nonunitary principal series $L^2(P\backslash G,\mathbb{H}_{\nu})$ onto a subrepresentation of the weighted $L^2$-space $L^2(K\backslash G,\mathbb{V}; \mathsf{a}^t)$ for large enough $t>0$. Here $\mathsf{a}:G\to (0,1]$ denotes the $A^+$-part in the $KAK$-decomposition $G=KA^+K$ when identifying $A^+\cong (0,1]$. The argument we give below proving Theorem \ref{conja} implies the following quantitative analogue of Helgason's conjecture.

\begin{cor1}
\label{remarkkwszeggeneric}
The Poisson transform \eqref{kwszeggeneric}, for large enough $t>0$ and for $\nu=s\alpha$, induces a continuous $G$-equivariant mapping 
$$P_\nu: W^{r}_{H}(P\backslash G, \mathbb{H}_{\nu})\to L^2(K \backslash G,\mathbb{V},\mathsf{a}^t),$$
if and only if $r\geq -\mathrm{Re}(s)-t/2$.
\end{cor1}

\subsection{Applications to the Baum-Connes conjecture in real rank one} 
We are interested in the commutator of the Szegö map with smooth functions on the Furstenberg compactification. Recall that if $G$ is of real rank one, $K\backslash G$ is diffeomorphic to an open unit ball, $P\backslash G$ is diffeomorphic to a sphere identifying with the boundary of $K\backslash G$ in such a way that $Z=K\backslash G\cup P\backslash G$ carries a smooth structure diffeomorphic to a closed unit ball. The compact $G$-space $Z$ is called the Furstenberg compactification. The reader should be wary that in rank $>1$ the Furstenberg compactification is more complicated. On the representation theory side, such aspects are best understood from the asymptotics of matrix coefficients, see \cite{knappbook,vcbschlich,wallachbook}.

\begin{thm1}
\label{conjb}
Under the assumptions and  notation of Theorem \ref{conja}, with $Z$ denoting the Furstenberg compactification of $K\backslash G$, for any $b\in C^\infty(Z)$ the following commutator is compact
\begin{equation} 
\label{jnljnad}
[S_{\mathcal{V}},b]: W^{-s_\mathcal{V}}_{H}(P\backslash G, \mathbb{H}_{\nu_{\mathcal{V}}})\to L^2(K \backslash G,\mathbb{V}).
\end{equation}
In fact, the operator \eqref{jnljnad} extends by continuity to $W^{-s_{\mathcal{V}}-1/2}_{H}(P\backslash G, \mathbb{H}_{\nu_{\mathcal{V}}})$.
\end{thm1}

We prove Theorem \ref{conjb} in Section \ref{commsec} by showing that the composition 
$$W^{-s_\mathcal{V}}_{H}(P\backslash G, \mathbb{H}_{\nu_{\mathcal{V}}})\xrightarrow{[S_{\mathcal{V}},b]} L^2(K \backslash G,\mathbb{V})\xrightarrow{[S_{\mathcal{V}},b]^\dagger}W^{s_\mathcal{V}}_{H}(P\backslash G, \mathbb{H}_{\nu_{\mathcal{V}}}^*),$$
is a Heisenberg pseudodifferential operator of order $-2s_\mathcal{V}$ with vanishing principal symbol. That is, $[S_{\mathcal{V}},b]^\dagger [S_{\mathcal{V}},b]$ is a Heisenberg pseudodifferential operator of order $-2s_\mathcal{V}-1$. Here $[S_{\mathcal{V}},b]^\dagger=[b^*,S_{\mathcal{V}}^\dagger]$ denotes the formal adjoint of $[S_{\mathcal{V}},b]$ in the $L^2$-pairings. The main technical tool we apply is the work of Ewert \cite{ewertthesis, ewertfixed}, realizing Heisenberg pseudodifferential operators as averages over the dilation, see also \cite{debordskandav}. The Schwartz kernel of $[S_{\mathcal{V}},b]^\dagger [S_{\mathcal{V}},b]$ can be computed as an average over $A$, and the work of Ewert \cite{ewertthesis, ewertfixed} does in essence reduce the proof to showing smoothness of the integrand on the parabolic tangent groupoid. To clarify the proof, we use a groupoid that we call the extended parabolic tangent groupoid of $P\backslash G$, discussed in Subsection \ref{subsecparatang}. The extended parabolic tangent groupoid is also used to show that $S_{\mathcal{V}}^\dagger S_{\mathcal{V}}$ is a Heisenberg pseudodifferential operator of order $-2s_\mathcal{V}$ thereby proving the continuity part of Theorem \ref{conja}.

\begin{remark}
We note the following $K$-homological connection between the source projection $P_{\mathcal{V}}$ of the Szegö map \eqref{kwszeg} and the discrete series representations. We write $\slashed{D}_{\mathcal{V}}$ for the Dirac operator on an appropriate homogeneous bundle on $K\backslash G$ whose kernel is the discrete series representation with lowest $K$-type $\mathcal{V}$, see \cite{laffaicm} for details. Under the boundary mapping in $G$-equivariant $K$-homology $\partial: K_0^G(K\backslash G)\to K_1^G(P\backslash G)$, Theorem \ref{conjb} and \cite{fries} implies that
$$[P_{\mathcal{V}}]=\partial[\ker \slashed{D}_{\mathcal{V}}]\in K_1^G(P\backslash G).$$
We note that $[P_{\mathcal{V}}]$ is a non-zero class in $K_1^G(P\backslash G)$ since the restriction map to $K_1(P\backslash G)$ produces a multiple of the generator (as can be computed with Bott periodicity). Nevertheless, the class of $[P_{\mathcal{V}}]$ under the restriction map  $K_1^G(P\backslash G)\to K_1^K(P\backslash G)$ pairs trivially with $K$-equivariant $K$-theory $K_1^K(P\backslash G)$ as can be seen from Bott's index theorem on homogeneous spaces \cite{botthomo}. The reader can find a more detailed discussion in Section \ref{szkhom}.
\end{remark}

Julg's seminal program for proving the Baum-Connes conjecture with coefficients for semisimple Lie groups of real rank one has been a motivation for this work. The case $G=SO(n,1)$ was proven by Kasparov \cite{kaspson}, and studied further by Chen for $n$ even in \cite{chenson}. The case $G=SU(n,1)$ was later proven by Julg-Kasparov \cite{julgkasp}. A general program was presented by Julg in \cite{julghow}. Julg's program reduced the Baum-Connes conjecture for semisimple Lie groups of real rank one to proving two conjectures: one representation theoretical and one analytical. The representation theoretical  conjecture, somewhat simplified, claims that spherical representations in a closed strip have slow exponential growth. This conjecture was proven by Julg-Nishikawa \cite[Theorem C]{julgnishi}, with detailed arguments given for $Sp(n,1)$. In \cite{julgnishi} the authors state that the Baum-Connes conjecture in real rank one follows from the analytical conjecture, which we establish in Theorem \ref{conjb}. \\

\noindent {\bf Acknowledgements.} 
The authors are very grateful to Pierre Julg and Robert Yuncken for comments on an earlier version of this work. We thank Bernhard Kr\"{o}tz, Erfan Rezaei, and Genkai Zhang for fruitful discussions. The second listed author was supported by the Swedish Research Council Grant VR 2025-03923.

\section{Background and notations}
\label{secback}

We review some well known notions from geometric analysis and representation theory in this section. Let $G$ denote a connected, semisimple Lie group. To simplify notation and to fix it throughout the paper, we assume that $G$ has real rank one, even when results in this section also hold in higher rank. We write $K$ for a maximal compact subgroup with Lie algebra $\mathfrak{k}$. We also assume that $G$ and $K$ have the same rank to ensure the existence of discrete series representations.

Take a Cartan decomposition $\mathfrak{g}=\mathfrak{k}+\mathfrak{p}$ of the Lie algebra of $G$ and a maximal abelian subalgebra $\mathfrak{a}\subseteq \mathfrak{p}$. We write $\mathfrak{a}_\C:=\mathfrak{a}\otimes_\R\C$. An element $\nu\in \mathfrak{a}_\C$ will be identified with the quasicharacter on $A$ defined by $a^\nu:=\mathrm{e}^{\nu(\log(a))}$ for $a\in A$. For $\lambda\in \mathfrak{a}^*$ we define the root space $\mathfrak{g}_\lambda:=\{X\in \mathfrak{g}: [H,X]=\lambda(H)X\; \forall H\in \mathfrak{a}\}$. Since $G$ has real rank one, $\mathfrak{a}$ is one-dimensional and there is a restricted root $\alpha\in \mathfrak{a}^*$ such that 
$$\mathfrak{g}=\mathfrak{g}_{2\alpha}+\mathfrak{g}_\alpha+\mathfrak{g}_0+\mathfrak{g}_{-\alpha}+\mathfrak{g}_{-2\alpha}.$$
If $\mathfrak{g}=\mathfrak{so}(n,1)$ then $\mathfrak{g}_{2\alpha}=\mathfrak{g}_{-2\alpha}=0$, but for any other real rank one Lie algebra the spaces $\mathfrak{g}_{2\alpha}$ and $\mathfrak{g}_{-2\alpha}$ are non-zero. We write $A=\mathrm{exp}(\mathfrak{a})$ and $P=MAN$ for the associated minimal parabolic subgroup of $G$, where $M$ is the centralizer of $A$ in $K$ and $N$ is the simply connected, nilpotent Lie group with Lie algebra $\mathfrak{n}=\mathfrak{g}_{2\alpha}+\mathfrak{g}_\alpha$. If we let $\theta$ denote the Cartan involution, and $V:=\theta(N)$ with Lie algebra $\mathfrak{v}$ then $\mathfrak{v}=\mathfrak{g}_{-\alpha}+\mathfrak{g}_{-2\alpha}$. We write $\rho$ for the half-sum of restricted roots with multiplicity, i.e. 
$$\rho=\frac{p+2q}{2}\alpha,.$$
where $p=\dim(\mathfrak{g}_{-\alpha})$ and $q=\dim(\mathfrak{g}_{-2\alpha})$. When discussing Knapp-Wallach's Szegö map, in Subsection \ref{subseclknlknad}, we shall also fix a noncompact root $\alpha_1$ for a maximal torus in $K$ that in turn will determine a particular choice of $A$ and simple restricted root $\alpha$ on $A$. 

We have the Bruhat decomposition 
$$G=MANV\dot{\cup} MANw,$$
where $w$ is an element centralizing $M$ and generating the restricted Weyl group of $\mathfrak{a}$. For $g\in MANV$ we write $m(g)\in M$ and $v(g)\in V$ for the unique elements with 
$$g\in m(g)ANv(g).$$
Since $MANV\subseteq G$ has full measure, $m(g)$ and $v(g)$ is defined for almost every $g\in G$. For $g\in G$, we write 
$$g=\mathsf{a}(g)n(g)\kappa(g)\in ANK,$$
for the Iwasawa decomposition. 

We will normalize the Haar measures so that $K$ and $M$ have unit volume and 
\begin{equation}
\label{normalz}
\int_V \mathsf{a}(v)^{2\rho}\rd v=1.
\end{equation}

\subsection{Representation theory}

Let us recall the basic elements of representation theory needed for this work. We refer the details to the textbooks, see for instance \cite{knappbook, wallachbook,wallachbookII}. The ultimate problem in representation theory is to given a group $G$ describe the space of irreducible, unitary representations of $G$. For the general class of semisimple Lie groups, this problem is hard and remains open. It turns out that for a semisimple Lie group $G$, it is easier to describe the space of \emph{admissible representations}. The irreducible admissible representations are classified by Langlands parameters by forming the uniquely determined irreducible quotient of the representation parabolically induced from the Langlands data. We refer the reader to \cite{afgomackey,clareetalI,clareetalII} for a holistic approach to such questions using operator algebraic methods. The reader can also find a computation of the space of irreducible, unitary representations in various instances of real rank one groups in \cite{foxhaskell,prudhon} and references therein. For semi-simple Lie groups of real rank one, the Langlands parameters arise from the discrete series representations of $G$ and the non-unitary principal series representations. The main result of this paper is concerned with the connection between the two.

\subsubsection{The non-unitary principal series representations} 
\label{subsubnonut}
are constructed as follows. If we take a quasicharacter $\nu\in \mathfrak{a}^*_\C$ and a finite-dimensional unitary representation $(\mathcal{H},\sigma)$ of $M$, there is an associated unitary principal series representation $\pi_{\sigma,\nu}$ of $G$ defined as follows in the induced/compact picture and the noncompact picture, respectively. By an abuse of notation, we write $\pi_{\sigma,\nu}$ for the representation in both pictures. In the induced/compact picture, we consider the homogeneous vector bundle $\mathbb{H}_\nu:=\mathcal{H}_\nu\times_P G\to P\backslash G$ where $\mathcal{H}_\nu:=\mathcal{H}$ with the action of $P$ defined from $man\mapsto \sigma(m)a^{\nu+\rho}$. The representation $\pi_{\sigma,\nu}$ of $G$ is on the Hilbert space $L^2(P\backslash G,\mathbb{H}_\nu)$ where it acts by right translation. We norm $L^2(P\backslash G,\mathbb{H}_\nu)$ by noting that $P\backslash G=M\backslash K$ so 
$$L^2(P\backslash G,\mathbb{H}_\nu)=L^2(M\backslash K,\mathcal{H}\times_MK)=L^2(K,\mathcal{H})^M.$$
The noncompact picture of the representation $\pi_{\sigma,\nu}$ is the representation space $L^2(V,\mathcal{H})$ with the action
$$\pi_{\sigma,\nu}(g)f(x)=\mathsf{a}(xg)^{\rho +\nu}\sigma(m(xg))f(v(xg)).$$
The unitary $G$-isomorphism $W_{\sigma,\nu}:L^2(V,\mathcal{H})\to L^2(K,\mathcal{H})^M$ is given by 
$$W_{\sigma,\nu}f(k):=\mathsf{a}(k)^{\rho +\nu}\sigma(m(k))f(v(k)),$$
with inverse
$$W_{\sigma,\nu}^{-1}\tilde{f}(x):=\mathsf{a}(x)^{\rho +\nu}f(\kappa(x)),$$

The representation $\pi_{\sigma,\nu}$ is also said to be in the non-unitary principal series. When $\nu\in i\mathfrak{a}$, then $\pi_{\sigma,\nu}$ is unitary and if moreover $\nu\neq 0$ then $\pi_{\sigma,\nu}$ irreducible, see \cite[Theorem 5]{knappstein} for more details. It was characterized in \cite[Section 14]{knappstein} when $\pi_{\sigma,\nu}$ is unitarizable. 

\subsubsection{The discrete series representations} 
\label{subsubecdisc} 
are constructed under our standing assumption that $K$ has the same rank as $G$. We take a maximal torus $T\subseteq K$. The discrete series representations are constructed from certain characters $\lambda$ on $T$ as solution spaces in $L^2(K\backslash G,\mathbb{V}_\lambda)$ for a first order differential operator. We will only briefly recall the salient feature in the construction, and refer the details to the literature \cite{atiischmid,knappbook,laffaicm,partha,schmiddie}.

Since $K$ has the same rank as $G$, we have that $\mathfrak{t}_\C\subseteq \mathfrak{g}_\C$ is a Cartan subalgebra. We write $\Delta=\Delta(\mathfrak{g}_\C,\mathfrak{t}_\C)$ for the associated roots and $\Delta_K=\Delta(\mathfrak{k}_\C,\mathfrak{t}_\C)$ for the set of compact roots. We also write $\Delta_n:=\Delta\setminus \Delta_K$ for the set of non-compact roots. A character $\lambda \in \mathfrak{t}_\C^*$ is said to be integral if it lifts to $T$ and non-singular if $\langle \lambda,\beta\rangle\neq 0$ for all $\beta\in \Delta$. If $\Lambda\in \mathfrak{t}_\C^*$ is a non-singular character we can define a positive root system by $\Delta^+:=\{\beta\in \Delta: \langle \Lambda,\beta\rangle> 0$ and with $\Lambda$ associate the character $\lambda\in \mathfrak{t}_\C^*$ by $\lambda:=\Lambda+\delta_n-\delta_k$, where $\delta_n$ is the half-sum of positive non-compact roots and $\delta_k$ is the half-sum of positive compact roots. If $\Lambda$ is non-singular and $\lambda$ integral, there is an associated discrete series representation $\pi_\lambda$ constructed inside $L^2(K\backslash G,\mathbb{V}_\lambda)$, for $\mathcal{V}_\lambda$ the $K$-representation with highest weight $\lambda$, as the kernel of a certain $G$-equivariant differential operator $\mathcal{D}_\Lambda$ \cite{schmiddie}, as used in \cite[Section 2]{knappwallach}. Alternatively, the discrete series representation $\pi_\lambda$ can be constructed from Dirac operators on $K\backslash G$ \cite{atiischmid,laffaicm,partha}.

\subsection{Poisson transforms}

There is a formalism for Poisson transforms dating back to Furstenberg \cite{furstann} and Helgason, explored at length in his books \cite{HelgasonDGLGSS,helgasongga,helgasongass} and the main player in the Helgason conjecture \cite{helgason70,helgason76}. We follow the more recent treatment due to Cap-Harrach-Julg \cite{capjulg} and connect it to Knapp-Wallach's Szegö map from \cite{knappwallach}. The formalism characterizes all $G$-equivariant maps from functions spaces on $P\backslash G$ to those on $K\backslash G$. The term Poisson transform here derives from the view on $P\backslash G$ as a (Furstenberg) boundary of $K\backslash G$ and the Poisson transform extends function to homogeneous solutions to a PDE, much in line with Knapp-Wallach's Szegö map \cite{knappwallach} that forms the main object of study in this paper. We summarize the formalism of Cap-Harrach-Julg \cite{capjulg} in the following theorem. 

\begin{thm}
\label{chjthm}
Assume that $\mathbb{H}\to P\backslash G$ and $\mathbb{V}\to K\backslash G$ are $G$-equivariant vector bundles with fiber $\mathcal{H}$ at $Pe$ and $\mathcal{V}$ at $Ke$. Then there is an isomorphism 
$$\mathrm{Hom}_G(\mathcal{D}'(P\backslash G,\mathbb{H}),C^\infty(K\backslash G,\mathbb{V}))\cong \mathrm{Hom}_M(\mathcal{H},\mathcal{V}),$$
obtained from composing the Schwartz kernel isomorphism 
$$\mathrm{Hom}_G(\mathcal{D}'(P\backslash G,\mathbb{H}),C^\infty(K\backslash G,\mathbb{V}))\cong C^\infty(K\backslash G\times P\backslash G; \mathbb{V}\times \mathbb{H}^*)^G,$$
with the isomorphism of the latter to $\mathrm{Hom}_M(\mathcal{H},\mathcal{V})$ defined from the $G$-equivariant diffeomorphism 
$$p_M:M\backslash G\to K\backslash G\times P\backslash G, \quad Mg\mapsto (Kg,Pg).$$
\end{thm}

\begin{remark}
Tracing through the isomorphism, we see that a $\phi\in  \mathrm{Hom}_M(\mathcal{H},\mathcal{V})$ under Theorem \ref{chjthm} corresponds to the Poisson transform
$$P_\phi:\mathcal{D}'(P\backslash G,\mathbb{H})\to C^\infty(K\backslash G,\mathbb{V}), \quad P_\phi u(g):=\int_{K} \tau(k)^{-1}\phi [u(kg)]\mathrm{d} k,$$
where $\tau$ denotes the isotropy representation of $K$ on the fiber $\mathcal{V}$ of $\mathbb{V}$ at $Ke$. In this integral representation of $P_\phi$, we identify $\mathcal{D}'(P\backslash G,\mathbb{H})=\mathcal{D}'(G,\mathcal{H})^P$ and $C^\infty(K\backslash G,\mathbb{V})=C^\infty( G,\mathcal{V})^K$.
\end{remark}

\begin{example}[Poisson transforms and the Furstenberg boundary]
A simple example of a Poisson transform arises when $\mathcal{H}=\C$ has the trivial $M$-action and $\mathcal{V}=\C$ has the trivial $K$-action. In this case, when fixing quasicharacter $\nu\in \mathfrak{a}_\C^*$ we obtain a line bundle $L_\nu\to P\backslash G$ from the action of $P$ on $\C$ defined from $man\mapsto a^{\nu+\rho}$. We write $P_\nu:\mathcal{D}'(P\backslash G,L_\nu)\to C^\infty(K\backslash G)$ for the associated Poisson transform. In \cite{vcbschlich}, the range of  $P_\nu:C^\infty(P\backslash G,L_\nu)\to C^\infty(K\backslash G)$ was computed and the boundary asymptotics of $P_\nu\varphi$, for $\varphi\in C^\infty(P\backslash G,L_\nu)$, was computed and related to the inverse of $P_\nu$ on its image. In the case $\nu=0$, we arrive at a Poisson transform
$$P:C^\infty(P\backslash G)\to C^\infty(K\backslash G), \quad P\varphi(Kg):=\int_K \varphi(kg)\rd k,$$
with range in the bounded harmonic functions on $K\backslash G$. When identifying $K\backslash G$ with a unit ball and $P\backslash G$ with its sphere at infinity \cite{jiborel}, the closed unit ball $Z=K\backslash G\dot{\cup}P\backslash G$ is called the Furstenberg compactification \cite{furstann} of $K\backslash G$. The asymptotic behavior of $P\varphi$ at $P\backslash G$ is precisely $\varphi$ up to lower order terms, by \cite{vcbschlich}. Writing $C_0^\infty(K\backslash G):=C_0(K\backslash G)\cap C^\infty(Z)$ we arrive at a short exact sequence of Frech\'{e}t algebras
$$0\to C_0^\infty(K\backslash G)\to C^\infty(Z)\to C^\infty(P\backslash G)\to 0,$$
which is linearly split by the positive $G$-equivariant Poisson transform $P$. This fact is refined by the next proposition that follows from \cite{vcbschlich}.
\end{example}

To describe smooth functions on the Furstenberg compactification $Z$,  we will use the $KA^+K$-description of $G$. Here $A^+\subseteq A$ denotes the closed, positive Weyl chamber for our restricted root $\alpha$. We have that $KA^+K= G$, see \cite[Theorem 8.6, page 323]{HelgasonDGLGSS} and restriction along this inclusion induces an embedding $C^\infty(Z)\hookrightarrow C^\infty_b(A^+K)$. We note that $A^+K\cong (0,1]\times K$ via the map $ak\mapsto (a^{-\alpha},k)$ and we write $C^\infty_0(A^+K)$ for the space of smoothly bounded functions $b$ on $A^+K$ with 
$$\lim_{a^{-\alpha}\to 0} b(ak)=\lim_{H\to +\infty} b(\mathrm{e}^Hk)=0,$$
where $H\to +\infty$ in $\mathfrak{a}$ if $\alpha(H)\to +\infty$.

\begin{prop}
\label{dedcom}
For any $b\in C^\infty(Z)$, there is a $b_0\in C^\infty(M\backslash K)$, a $b_1\in C^\infty_0(A^+K)^M$ and a pre-compact subset $U\subseteq K\backslash G$ such that outside of $U$, 
$$b(ak)=b_0(k)+b_1(ak).$$
\end{prop}

\subsection{Poisson transforms and BGG-complexes} 
\label{subsecpoisson}
The connection between Poisson transforms, BGG-complexes and harmonic function was described in \cite{capjulg}. This connection plays an important role for the application of our main results to Julg's program for the Baum-Connes conjecture, so we recall this connection briefly. The BGG-complex is a graded hypoelliptic complex isomorphic to a subquotient of the de Rham complex. As \cite{capjulg} emphasizes, the BGG-complex on $P\backslash G$ is already packaged into principal series components contrasting the more complicated representation theoretical content of the de Rham complex. Let us make this somewhat more precise, leaving the details to the literature \cite{cap,morecap,Dave_Haller1,goffbgg,julghow}. 

Using the Cayley transform $V\to P\backslash G=M\backslash K$, $v\mapsto M\kappa(v)$, we can identify $T(P\backslash G)=\mathfrak{v}\times_M K$. The Killing form allows us to identify $\mathfrak{v}^*=\mathfrak{n}$, so $T^*(P\backslash G)=\mathfrak{n}\times_M K$ and more generally $\wedge^j T^*(P\backslash G)=(\wedge^j\mathfrak{n})\times_M K$. A homogeneous bundle $\mathbb{E}\to P\backslash G$ is called a tractor bundle if  there is a finite-dimensional $G$-module $E$ such that $\mathbb{E}=E\times_P G$; note $\wedge^j T^*(P\backslash G)\otimes\mathbb{E} =(\wedge^j\mathfrak{n}\otimes E)\times_P G$. The Kostant differential $\partial_E$ is a vector bundle morphism  $\wedge^* T^*(P\backslash G)\to \wedge^{*-1} T^*(P\backslash G)$ defined from the Lie algebra differential on $\wedge^*\mathfrak{n}\otimes E$. The homology of the Kostant differential form vector bundles $H_*(\mathbb{E})\to P\backslash G$. We have that $H_*(\mathbb{E})=H_*(\wedge^j\mathfrak{n}\otimes E)\times_P G$ where $H_*(\wedge^j\mathfrak{n}\otimes E)$ is a module for $MA$ whose structure is explicitly determined from the highest weights in $E$ from \cite[Theorem 3.3.5 and Proposition 3.3.6]{cap}. In particular, $L^2(P\backslash G,H_*(\mathbb{E}))$ decomposes as a direct sum of non-unitary principal series representations with parameters explicitly determined from the highest weights in $E$. The de Rham differential on $C^\infty(P\backslash G,\wedge^*T^*(P\backslash G)\otimes\mathbb{E})$ compresses to the subquotient $C^\infty(P\backslash G,H_*(\mathbb{E}))$ via a procedure involving compressing to forms harmonic for the Kostant differential, for details see \cite{Dave_Haller1} or \cite{morecap,goffbgg}. The resulting BGG-complex is a $G$-equivariant complex whose constituents are non-unitary principal series representations. We now summarize \cite[Proposition 2.2]{capjulg} connecting Poisson transforms and BGG-complexes to harmonic forms. Recall that $\mathfrak{p}=\mathfrak{g}/\mathfrak{k}$ is the fiber of the tangent bundle of $K\backslash G$.

\begin{thm}
Let $E_1$ and $E_2$ be finite-dimensional $G$-representations inducing homogeneous bundles $\mathbb{E}_1\to P\backslash G$ and $\mathbb{E}_2\to K\backslash G$. Take a $\phi\in  \mathrm{Hom}_M(\wedge^j\mathfrak{n}\otimes E_1,\wedge^l \mathfrak{p}^*\otimes E_2)$. Then the following are equivalent:
\begin{enumerate}
\item $P_\phi$ maps $\mathcal{D}'(P\backslash G,\wedge^jT^*(P\backslash G)\otimes \mathbb{E}_1)$ into harmonic $l$-forms on $K\backslash G$.
\item $P_\phi\circ \partial=P_\phi\circ \rd \circ \partial=0$ for the exterior differential $\rd$ and the Kostant differential $\partial$.
\item $\phi\circ \partial=0$ in $\mathrm{Hom}_M(\wedge^{j+1}\mathfrak{n}\otimes E_1,\wedge^l \mathfrak{p}^*\otimes E_2)$ and $\phi\circ \delta\circ \partial=0$ in $\mathrm{Hom}_M(\wedge^{j}\mathfrak{n}\otimes E_1,\wedge^l \mathfrak{p}^*\otimes E_2)$ for $\delta$ the dual of the Lie algebra differential on $C^j(\mathfrak{n})$.
\end{enumerate}
If either of these conditions hold, $P_\phi$ descends to a map from distributions on the BGG-complex $\mathcal{D}'(P\backslash G,H_j(\mathbb{E}_1))$ to harmonic $l$-forms on $\mathbb{E}_2$. 
\end{thm}

\subsection{Knapp-Wallach's Szegö map}
\label{subseclknlknad}

The Poisson transforms we will primarily study in this paper arise from the particular situation studied by Knapp-Wallach \cite{knappwallach}, relating certain non-unitary principal series representations to discrete series representations. For simplicity of notation we continue to work under the assumption that $G$ has real rank one. We follow the notations for discrete series representations from Subsection \ref{subsubecdisc}. Consider an irreducible $K$-representation $(\mathcal{V},\tau)$ with highest weight $\lambda$ such that $\Lambda=\lambda-\delta_n+\delta_k$ is non-singular. 

We assume that $A$ is defined from a noncompact root defined relative the compact Cartan subalgebra $\mathfrak{t}_\C$ as in \cite[Section 2]{knappwallach}. To be precise, we fix a noncompact simple root $\alpha_1$ for the maximal torus in $K$. Then $E_{\alpha_1}+E_{-\alpha_1}\in \mathfrak{p}$ defines $A$ and a simple restricted root $\alpha$ for $A$. We write $\rho$ for the associated half sum of positive restricted roots. We also form the character $\nu_{\mathcal{V}}$ on $A$ by 
\begin{equation}
\label{nuvforea}
\nu_{\mathcal{V}}(E_{\alpha_1}+E_{-\alpha_1}):=\frac{2\langle \lambda +n_1\alpha_1,\alpha_1\rangle}{\langle \alpha_1,\alpha_1\rangle}=\rho(E_{\alpha_1}+E_{-\alpha_1})-\frac{2}{|\alpha|^2}\langle \Lambda,\alpha_1\rangle,
\end{equation}
where $n_1$ is defined in \cite[Equation (6.5.b)]{knappwallach} and the second identity is found in \cite[Equation (1.4)]{blankrankone}. We also form the irreducible $M$-representation $(\mathcal{H},\sigma)\subseteq (\mathcal{V},\tau|_M)$ generated by $\lambda$.

The $G$-equivariant Knapp-Wallach Szegö map is then defined as
\begin{align*}
&S_{\mathcal{V}}:\mathcal{D}'(P\backslash G,\mathbb{H}_{\nu_{\mathcal{V}}})\to C^\infty(K\backslash G,\mathbb{V}), \\ 
&S_{\mathcal{V}} f(g):=\int_K \tau(k)^{-1} f(kg)\rd k=\int_K\mathsf{a}(lg^{-1})^{\nu_{\mathcal{V}}}\tau(\kappa(lg^{-1}))f(l)\rd l.
\end{align*}
It was proven by Knapp-Wallach \cite{knappwallach} that $S_{\mathcal{V}}$ induces an infinitesimal intertwiner from the nonunitary principal series $L^2(P\backslash G,\mathbb{H}_{\nu_{\mathcal{V}}})$ to the discrete series representation $\pi_\lambda$ with lowest $K$-type $\lambda$. In particular, $S_{\mathcal{V}}$ maps the $K$-finite vectors in $C^\infty(P\backslash G,\mathbb{H}_{\nu_{\mathcal{V}}})$ into the $K$-finite vectors of $\pi_\lambda$, in particular $S_{\mathcal{V}}f\in L^2(K\backslash G,\mathbb{V})$ for $f$ $K$-finite.

\begin{remark}
\label{globaadad}
We note that $S_{\mathcal{V}}$ is only an infinitesimal intertwiner, and there are to the authors' knowledge no results characterizing Sobolev type spaces on $P\backslash G$ for which $S_{\mathcal{V}}$ is continuous into $L^2(K\backslash G,\mathbb{V})$. However, by Casselman-Wallach's globalization theorem \cite{cassglob,wallachasy} or \cite[Chapter 11]{wallachbookII}, and in particular Bernstein-Krötz's interpretation  \cite[Theorem 1.1]{krobern}, it holds for any decreasing scale of Sobolev spaces $(\mathpzc{H}^s(P\backslash G,\mathbb{H}_{\nu_{\mathcal{V}}}))_{s\in \R}$ (satisfying $\mathpzc{H}^0(P\backslash G,\mathbb{H}_{\nu_{\mathcal{V}}})=L^2(P\backslash G,\mathbb{H}_{\nu_{\mathcal{V}}})$, $\cap_s \mathpzc{H}^s(P\backslash G,\mathbb{H}_{\nu_{\mathcal{V}}})=C^\infty(P\backslash G,\mathbb{H}_{\nu_{\mathcal{V}}})$ and $[\mathpzc{H}^s(P\backslash G,\mathbb{H}_{\nu_{\mathcal{V}}})]^*=\mathpzc{H}^{-s}(P\backslash G,\mathbb{H}_{\nu_{\mathcal{V}}})$ in the $L^2$-pairing), that there exists an $s_0\in \R$ for which 
$$S_{\mathcal{V}}: \mathpzc{H}^{s_0}(P\backslash G,\mathbb{H}_{\nu_{\mathcal{V}}})\to L^2(K\backslash G,\mathbb{V}),$$
is continuous. To summarize, abstract nonsense ensures continuity of $S_{\mathcal{V}}$ in any Sobolev scale into the discrete series representation acting unitarily on $L^2$, but we have no control of the Sobolev regularity nor of whether or not if there is an optimal value or if $S_{\mathcal{V}}$ has closed range for the optimal Sobolev regularity.
\end{remark}

\begin{remark}
Knapp-Wallach's Szegö map ties together with the BGG-complexes discussed in Subsection \ref{subsecpoisson}. This connection is of importance to Julg's program for the Baum-Connes conjecture \cite{julgkasp,julghow}. It is well known that the harmonic $L^2$-forms on $K\backslash G$ concentrate in middle degree and forms a finite sum of discrete series representations. 

For instance, if $G=SU(n,1)$, the harmonic forms were studied by Pedon \cite{pedonsun}. The space of harmonic $k$-forms $\mathfrak{H}^k(K\backslash G)\subseteq L^2(K\backslash G,\wedge^k T^*( K\backslash G))$ is trivial unless $k=n$. Moreover, $\mathfrak{H}^{n}(K\backslash G)$ splits as a direct sum of $n+1$ discrete series representations with multiplicity one, see \cite[Theorem 3.3]{pedonsun} for more details including a description of the lowest $K$-types. In this case, it follows from \cite{julgkasp} that Knapp-Wallach's Szegö map is an infinitesimal intertwiner mapping $K$-finite vectors of $L^2(P\backslash G,H_{n}(\C\times_P G))$ to  $K$-finite vectors of $\mathfrak{H}^{n}(K\backslash G)$.

For $G=Sp(n,1)$, the story is similar and the harmonic forms were studied by Pedon \cite{pedonharm}. The space of harmonic $k$-forms $\mathfrak{H}^k(K\backslash G)\subseteq L^2(K\backslash G,\wedge^k T^*( K\backslash G))$ is trivial unless $k=2n$. Moreover, $\mathfrak{H}^{2n}(K\backslash G)$ splits as a direct sum of $n+1$ discrete series representations with multiplicity one, see \cite[Theorem 4.13]{pedonharm} for more details including a description of the lowest $K$-types. In this case, it follows from \cite{julghow} that Knapp-Wallach's Szegö map is an infinitesimal intertwiner mapping $K$-finite vectors of $L^2(P\backslash G,H_{n}(\C\times_P G))$ to  $K$-finite vectors of $\mathfrak{H}^{2n}(K\backslash G)$.
\end{remark}

\subsection{Knapp-Stein intertwiners}
\label{subsecks}

An important tool for studying equivalences and reducibility in parabolically induced representations is Knapp-Stein's theory for intertwiners \cite{knappstein}. We will use Knapp-Stein intertwiners in the non-unitary principal series together with the Heisenberg calculus to prove the closed range statement in Theorem \ref{conja} below in Section \ref{aldknadln}. We follow the notations and conventions of Subsubsection \ref{subsubnonut}. Recall that $w$ denotes an element centralizing $M$ and generating the restricted Weyl group of $\mathfrak{a}$, the Knapp-Stein intertwiner depends on $w$ but since we only consider groups of real rank one we drop $w$ from the notation used in \cite{knappstein}.

We first define the Knapp-Stein intertwiner in the compact picture $L^2(P\backslash G,\mathbb{H}_\nu)=L^2(K,\mathcal{H})^M$. We say that $\mathrm{Re}(\nu)> 0$ if $\nu=z\rho$ for a $z\in \C$ with $\mathrm{Re}(z)> 0$. The Knapp-Stein intertwiner 
$$A_{\sigma,\nu}:L^2 (P\backslash G,\mathbb{H}_\nu)\to L^2(P\backslash G,\mathbb{H}_{-\nu}),$$
is for $\mathrm{Re}(\nu)> 0$ defined by the integral operator
$$A_{\sigma,\nu}f(k):=\int_K\mathsf{a}(lw)^{\rho-\nu}\sigma(m(lw))f(lk)\rd l.$$
A priori, $A_{\sigma,\nu}:C^\infty(P\backslash G,\mathbb{H}_\nu)\to C^\infty(P\backslash G,\mathbb{H}_{-\nu})$ but since the kernel of $A_{\sigma,\nu}$ is integrable for $\mathrm{Re}(\nu)> 0$, $A_{\sigma,\nu}$ is a well defined bounded operator on $L^2$. Moreover, 
$$A_{\sigma,\nu}\pi_{\sigma,\nu}(g)=\pi_{\sigma,-\nu}(g)A_{\sigma,\nu},$$
and so $A_{\sigma,\nu}$ is $G$-equivariant. As proven in \cite{knappstein}, the intertwiners $A_{\sigma,\nu}$ are essentially the only possible intertwiners within the nonunitary principal series apart from some exceptional cases. The next result is found as \cite[Lemma 21]{knappstein}.

\begin{lemma}
Assume that $(\mathcal{H},\sigma)$ and $(\mathcal{H}',\sigma')$ are two $M$-representations and that $\nu,\nu'\in \mathfrak{a}_\C$ are two characters with $\mathrm{Re}(\nu)> 0$. If 
$$\mathcal{L}:C^\infty(P\backslash G,\mathbb{H}_\nu)\to C^\infty(P\backslash G,\mathbb{H}'_{-\nu}),$$
is a $G$-equivariant linear operator, then either $\sigma=\sigma'$ and $\nu=\nu'$ and $\mathcal{L}$ is a scalar, or $\sigma=\sigma'$ and $\nu=-\nu'$ and for a constant $c\in \C$ we have $\mathcal{L}=cA_{\sigma,\nu}$. 
\end{lemma}

To extend to the case $\mathrm{Re}(\nu)\leq 0$, one proceeds by a meromorphic extension in $\nu$. We shall use the notation $C^\infty(P\backslash G,\mathbb{H}):=C^\infty(P\backslash G,\mathbb{H}_\nu)$ which is a well defined Frech\'{e}t space, that when dropping the $G$-action does not depend on $\nu$.  We recall \cite[Theorem 3]{knappstein}

\begin{lemma}
The holomorphic family of operators 
$$(A_{\sigma,\nu}:C^\infty(P\backslash G,\mathbb{H})\to C^\infty(P\backslash G,\mathbb{H}))_{\mathrm{Re}(\nu)> 0},$$
extends meromorphically to $\nu \in \C$. The poles of $\nu\mapsto A_{\sigma,\nu}$ are situated in the points $\nu\in -\alpha \N$ where $\alpha$ denotes the simple, positive, restricted root, and the poles are at most simple. 
\end{lemma}

\begin{remark}
\label{lknkanda}
We note in particular, that combining the two lemmas above, we have that if $(\mathcal{L}_\nu)_{\mathrm{Re}(\nu)>> 0}$ is a holomorphic family of operators $C^\infty(P\backslash G,\mathbb{H})\to C^\infty(P\backslash G,\mathbb{H})$ such that for $\mathrm{Re}(\nu)>> 0$
$$\mathcal{L}_\nu:C^\infty(P\backslash G,\mathbb{H}_\nu)\to C^\infty(P\backslash G,\mathbb{H}_{-\nu}),$$
is a $G$-equivariant linear operator, then up to multiplying by a holomorphic function in $\mathrm{Re}(\nu)>> 0$, $(L_\nu)_{\mathrm{Re}(\nu)>> 0}$ extends meromorphically to $\nu \in \C$ with poles only in $\nu\in -\alpha \N$ and $\mathcal{L}_\nu$ coincides with $A_{\sigma,\nu}$.
\end{remark}

To understand the placement of the Knapp-Stein intertwiners in the Heisenberg calculus, we also describe them in the nilpotent picture. We write 
$$A_{\sigma,\nu}^{(0)}:=W_{\sigma,\nu}^{-1}A_{\sigma,\nu}W_{\sigma,\nu}.$$
A short computation shows that $A_{\sigma,\nu}^{(0)}:C^\infty_c(V,\mathcal{H})\to C^\infty(V,\mathcal{H})$ takes the form
$$A_{\sigma,\nu}^{(0)}f(x)=\int_V\mathsf{a}(yw)^{\rho-\nu}\sigma(m(yw))f(yx)\rd y.$$
We can from \cite[Lemma 19]{knappstein} conclude the following. We use the notation $|\pmb{\lambda}|$ for the bundle of $1$-densities.

\begin{prop}
\label{lknlajndljadn}
Consider the convolution kernel density
$$\mathfrak{a}_{\sigma,\nu}(y):=\mathsf{a}(yw)^{\rho-\nu}\sigma(m(yw))\rd y, \quad y\in V\setminus \{e\},$$
of $A_{\sigma,\nu}^{(0)}$, where $\rd y$ denotes the constant density normalized so that \eqref{normalz} holds. Then $\mathfrak{a}_{\sigma,\nu}\in C^\infty(V\setminus \{e\},\End(\mathcal{H})\otimes |\pmb{\lambda}|)$ and $\mathfrak{a}_{\sigma,\nu}$ is homogeneous of degree $-2\nu$ under pushforward of densities in the $A$-action $a:v\mapsto ava^{-1}$. In particular, $A_{\sigma,\nu}^{(0)}$ is for $\nu\in \mathfrak{a}_\C\setminus (-\alpha\N)$ uniquely determined from the unique homogeneous extension of $\mathfrak{a}_{\sigma,\nu}$ to a distribution on $V$.
\end{prop}

\begin{remark}
\label{langlandsrem}
To facilitate our understanding of the projection $P_{\mathcal{V}}$ in the statement of Theorem \ref{conja}, we note that \cite[Theorem 7.24]{knappbook} implies the following interpretation of the range of $A_{\sigma,\nu}$. For  $\mathrm{Re}(\nu)> 0$, \cite[Theorem 7.24]{knappbook} implies that, as Harish-Chandra modules, $A_{\sigma,\nu}$ implements an infinitesimal equivalence between the image of $A_{\sigma,\nu}$ and the unique irreducible quotient of $\pi_{\sigma,\nu}$ -- aka the Langlands quotient $J(P,\sigma,\nu)$ of the Langlands parameter $(P,\sigma,\nu)$. In particular, the image of $A_{\sigma,\nu}$ is always an irreducible Harish-Chandra module.
\end{remark}

\section{The parabolic tangent groupoid}
\label{lnlnadna}

The main technical tool in this work is the Heisenberg calculus on filtered manifolds, also known as Carnot manifolds. We will in this section describe how the parabolic space 
$$X:=P\backslash G,$$ 
forms a filtered manifold (Subsection \ref{subsecfiltered}) and to describe its parabolic tangent groupoid (Subsection \ref{subsecparatang}). A construction of importance to our proofs of  Theorems \ref{conja} and \ref{conjb}, that we present in Subsection \ref{subsecparatang}, is that of an extended parabolic tangent groupoid on $X=P\backslash G$ that with Ewert's techniques \cite{ewertthesis, ewertfixed} for the Heisenberg calculus will allow for a clearer identification of how the operators relevant for Theorem \ref{conja} and \ref{conjb} belong to the Heisenberg calculus. We restrict our attention to $G$ having real rank one. The discussion throughout the section extends to higher rank, albeit it seems unlikely to the authors to be directly applicable to problems in representation theory of higher rank groups, cf. Remark \ref{alknalndarank}. We will assume the reader to be familiar with the theory of Lie groupoids and their associated convolution algebras, see \cite{debordlescure,debordskandext,debordskandpseudo} for more material in this area.

\subsection{The filtered geometry of $P\backslash G$}
\label{subsecfiltered}
We consider the manifold $X:=P\backslash G$. This manifold carries a Carnot structure defined from the restricted weight spaces and there is an associated Heisenberg calculus of operators. We will follow the notations introduced at the start of Section \ref{secback}. The weight spaces for the simple, positive, restricted root $\alpha$ decomposes the Lie algebra as a graded Lie algebra
$$\mathfrak{g}=\mathfrak{g}_{-2\alpha}\oplus \mathfrak{g}_{-\alpha}\oplus \mathfrak{g}_0\oplus \mathfrak{g}_\alpha\oplus \mathfrak{g}_{2\alpha}.$$
Here $\mathfrak{g}_0=\mathfrak{a}\oplus \mathfrak{m}$, and $\mathfrak{g}_\geq=\mathfrak{g}_0\oplus \mathfrak{g}_1\oplus \mathfrak{g}_2$ is the Lie algebra of $P$, with $\mathfrak{n}=\mathfrak{g}_1\oplus \mathfrak{g}_2$ and $\mathfrak{v}=\mathfrak{g}_{-2}\oplus \mathfrak{g}_{-1}$. We can identify $\mathfrak{g}/\mathfrak{g}_\geq =\mathfrak{v}$ as $P$-modules if we equip $\mathfrak{v}$ with the quotient action of $P$, so we can identify 
$$TX=\mathfrak{v}\times_P G\to P\backslash G.$$
The decomposition $\mathfrak{v}=\mathfrak{g}_{-2}\oplus \mathfrak{g}_{-1}$ makes $X$ into a filtered manifold of depth $2$, i.e. $TX$ has a filtering 
$$0\subsetneq T^{-1}X\subsetneq T^{-2}X=TX.$$
The reader can find the general definition of a filtered manifold in for instance \cite{Dave_Haller1,ewertfixed,goffkuz,vanerpyunck}, where in depth $>2$ one further imposes the assumption that $[T^jX,T^k]\subseteq T^{j+k}X$ for all $j,k$. Here we have $T^{-1}X:= \mathfrak{g}_{-1}\times_P G$ and $T^{-2}X:=TX$. Indeed, $T^{-1}X$ is well defined since the Lie bracket respects the grading, and so $\mathfrak{g}_{-1}\subseteq \mathfrak{v}$ is $P$-invariant. 

We write $\mathfrak{t}_HX$ for the graded bundle associated with the filtered bundle $TX$. There is an induced $G$-action on $\mathfrak{t}_HX$ lifting the $G$-action on $X$. The decomposition $\mathfrak{v}=\mathfrak{g}_{-2}\oplus \mathfrak{g}_{-1}$ induces isomorphisms 
$$\mathfrak{t}_HX\cong \mathfrak{g}_{-1}\times_P G\oplus \mathfrak{g}_{-2}\times_P G\cong TX,$$
where the first but not the second isomorphism is equivariant. Indeed, $TX=\mathfrak{v}\times_P G$ with the filtered action of $P$ on $\mathfrak{v}$ but $\mathfrak{t}_HX=\mathfrak{v}\times_P G$ with the graded action of $P$ on $\mathfrak{v}$ defined from the quotient map $P\to MA$. The bundle  $\mathfrak{t}_HX$ is a locally trivial bundle of Lie algebras, with fiber $\mathfrak{v}$, and the fiberwise Lie bracket coincides with that induced from the Lie bracket of vector fields. The bundle of Lie algebras integrate to a bundle of nilpotent Lie groups $T_HX$. The locally trivial bundle $T_HX\to X$ will be viewed as a Lie groupoid, it is called the osculating Lie groupoid. There is a $G$-equivariant isomorphism of bundles over $X$
$$T_HX\cong V\times_PG,$$
where the $P$-action on $V$ is as graded automorphisms defined from the quotient map $P\to MA$. As $K$-equivariant bundles, $T_HX\cong V\times_MK$ when identifying $M\backslash K\cong P\backslash G$. \\

{\bf Dilations and identifying $A=(0,\infty)$:} Conjugation by elements from $A$ defines a dilation action on $V$, we write this as 
\begin{equation}
\label{lknlknadad}
\delta_a(v):=ava^{-1}.
\end{equation}
Occasionally, we identify $A=(0,\infty)$ via the mapping $a\mapsto a^{-\alpha}$. The positive Weyl chamber $A^+$ with respect to $\alpha$ corresponds to $(0,1]$. Under this identification, we write the elements of $A$ as $s>0$ which acts on $V$ in the action \eqref{lknlknadad} as 
$$\delta_s(\mathrm{e}^{X_1+X_2})=\mathrm{e}^{sX_1+s^2X_2},\quad\mbox{for $X_1\in \mathfrak{g}_{-\alpha}$ and $X_2\in \mathfrak{g}_{-2\alpha}$.}$$

\begin{remark}
In the case that $G$ is of higher rank, and $P\subseteq G$ is a more general (standard) parabolic subgroup then $P\backslash G$ is still a filtered manifold. In this case, one filters $T(P\backslash G)$ using the total degree with respect to all restricted roots. For more details, see \cite{cap}, Chapter 3.2.1 for the complex case and Chapter 3.2.9 for the real case. The higher rank case calls for a tool capturing more subtle changes in gradings, for instance through the multifiltered manifolds of Yuncken \cite[Section 3.4]{yunckenhab}.
\end{remark}

\subsection{The parabolic tangent groupoid on $P\backslash G$}
\label{subsecparatang}
The Heisenberg calculus we use was constructed by van Erp-Yuncken \cite{vanerpyunck} from almost homogeneous distributions on an adiabatic-like deformation for the dilation action near the diagonal. The appropriate adiabatic deformation comes from the parabolic tangent groupoid of $X$. We shall give a more detailed description of the parabolic tangent groupoid of $X=P\backslash G$ using that $G$ has real rank one.

For a general filtered manifold $X$, the parabolic tangent groupoid $\mathbb{T}_HX \rightrightarrows X\times [0,\infty)$ is defined as 
$$\mathbb{T}_HX:=(T_HX\times \{0\})\dot{\cup}(X\times X\times (0,\infty)),$$
as a groupoid over $X\times [0,\infty)$. We write $t$ for the variable parametrizing $[0,\infty)$. The fiber at $t=0$ is $T_HX$ which as a bundle of nilpotent Lie groups carries a fiberwise group structure over $X$. The fibers at $t>0$, $X\times X\times \{s\}$ are viewed as pair groupoids, i.e. with source map $\mathsf{s}(x,y)=y$, range map $\mathsf{r}(x,y)=x$ and composition rule $(x,y)(y,z)=(x,z)$. The parabolic tangent groupoid becomes a Lie groupoid when $X\times X\times (0,\infty)$ is given its natural smooth structure and declaring a particular map $\psi_\nabla:T_HX\times [0,\infty)\to \mathbb{T}_HX$ to be smooth. The map $\psi_\nabla$ is defined as 
\begin{equation}
\label{psinabladef}
\psi_\nabla(x,v,t):=
\begin{cases}
(x,v,0), \quad &t=0,\\
(\exp_x^\nabla(-\delta_tv),x,t), \quad &t>0.\end{cases}
\end{equation}
Here $\nabla$ is a graded connection. The smooth structure is independent of the choice of this connection. The dilation action on the fiber extends to the zoom action on the parabolic tangent groupoid defined for $\lambda>0$ as
$$\delta_\lambda(x,v,0):=(x,\delta_\lambda(v),0),\quad\mbox{and}\quad \delta_\lambda(x,y,t):=(x,y,\lambda^{-1}t).$$
\newline 

Let us describe the construction of the parabolic tangent groupoid in more detail in the specific case $X=P\backslash G$ relevant to this article. Note that $a\mapsto a^{-\alpha}$ defines a diffeomorphism $A\cong (0,\infty)$. We therefore have an identification 
$$P\backslash G\times P\backslash G\times (0,\infty)\cong (MN\backslash G)\times_A (MN\backslash G),$$
where we identify $MN\backslash G\cong A\times M\backslash K$ via the Iwasawa decomposition. Here, all identifications are as manifolds, not as groups, but the zoom action of $A\cong (0,\infty)$ on $(MN\backslash G)\times_A (MN\backslash G)$ goes via the diagonal action induced from that on  $MN\backslash G\cong A\times M\backslash K$. 

Let us introduce the partial compactifications
$$\overline{A}:=\{0\}\dot{\cup} A,\quad\mbox{and}\quad \overline{\mathfrak{Y}}=\overline{A}\times K,$$
which are endowed with a natural topology by declaring $a\to 0$ if $a^{-\alpha}\to 0$, where $\alpha$ is any simple, positive restricted root. The smooth map $a\mapsto a^{-\alpha}$ extends by continuity to a diffeomorphism $\overline{A}\cong [0,\infty)$. We often use the substitution 
$$t=a^{-\alpha},$$
and think of $[0,\infty)$ as parametrized by $\overline{A}$. The reader should beware that we view the partial compactification $\overline{A}$ as a topological space with an $A$-action, in which $A\subseteq \overline{A}$ is an $A$-invariant subset, and not a group. 

\begin{prop}
\label{lknadlnad}
For $X=P\backslash G=M\backslash K$, the parabolic tangent groupoid $\mathbb{T}_HX$ is set theoretically given by
$$\mathbb{T}_HX=V\times_M K\times \{0\}\dot{\bigcup} M\backslash K\times M\backslash K\times A\rightrightarrows M\backslash K \times \overline{A}.$$
It is endowed with a smooth structure, defined by declaring $M\backslash K\times M\backslash K\times A\subseteq \mathbb{T}_HX$ to be an open submanifold and declaring the following map a smooth embedding 
$$\psi:V\times_M K\times \overline{A}\to \mathbb{T}_HX, 
\quad 
(M(k,v),a)\mapsto\begin{cases}
(Mk\kappa(ava^{-1})^{-1},Mk,a)\; &a\in A,\\
{}\\
(M(v,k),0),\; &a=0.
\end{cases}$$
\end{prop}

\begin{proof}
To prove the proposition it suffices to prove that for some graded connection $\nabla$, the mapping
$$\psi^{-1}\circ \psi_\nabla:V\times_M K\times \overline{A}\to V\times_M K\times \overline{A},$$ 
is smooth. Here $\psi_\nabla$ is as in Equation \eqref{psinabladef}, and we identify $V\times_M K=T_H(M\backslash K)$. Recall that $ava^{-1}=\delta_av$. 

We take $\nabla$ such that for $X\in \mathfrak{v}$
$$\exp_{Mk}^\nabla(-\delta_a\mathrm{e}^{X})= Mk\mathrm{e}^{-D\kappa(aXa^{-1})}.$$
We will write $c:V\to P\backslash G$, $v\mapsto Pv$ for the Cayley transform and we view $c^{-1}$ as partially defined, with domain $(P\backslash G)\setminus \{Pw\}$. Since $\psi=\psi_\nabla$ on $V\times_M K\times\{0\}$, smoothness of $\psi^{-1}\circ \psi_\nabla$ follows if the function 
$$\hat{\kappa}:\mathfrak{v}\times A\to V, \quad \hat{\kappa}(v,a):= a^{-1}c^{-1}(\mathrm{e}^{-D\kappa(aXa^{-1})})^{-1}a,$$
extends to a smooth function $\mathfrak{v}\times \overline{A}\to V$ with $\hat{\kappa}(X,0)=\mathrm{e}^{X}$. Using the Iwasawa decomposition we see that it suffices to prove that the function 
$$\hat{\hat{\kappa}}:\mathfrak{v}\times A\to G, \quad \hat{\hat{\kappa}}(v,a):= a^{-1}\mathrm{e}^{-D\kappa(aXa^{-1})}a,$$
extends to a smooth function $\mathfrak{v}\times \overline{A}\to G$ with $\hat{\hat{\kappa}}(X,0)=\mathrm{e}^{X}$. Since the statement we wish to prove is local, we can assume that $G$ admits a faithful finite-dimensional representation. Because $D\kappa(aXa^{-1})=\frac{1}{2}aXa^{-1}-\frac{1}{2}\theta(aXa^{-1})$, the Baker-Campbell-Hausdorff formula in the faithful finite-dimensional representation of $G$ implies that $a^{-1}\mathrm{e}^{-D\kappa(aXa^{-1})}a=\mathrm{e}^X+a^{-\alpha}f(a^{-\alpha},X)$ for a real analytic function $f$. The proposition follows. 
\end{proof}

We will next define a groupoid that refines the parabolic tangent groupoid construction to include the $M$-part. That groupoid will encode the operators between the homogeneous vector bundles on $M\backslash K$. The building blocks for this groupoid are  two groupoids on $M\backslash K$ built as follows:
\begin{enumerate}
\item We write $MV\subseteq G$ for the group generated by $M$ and $V$, which is isomorphic to the semidirect product $V\rtimes M$. The quotient map $(MV)\times_M K\to M\backslash K$, induced from 
$$(MV)\times K\ni (mv,k) \mapsto Mk\in M\backslash K,$$
defines a principal $MV$-bundle on $M\backslash K$. The left action of $M$ on $MV$ produces a fiberwise action on the principal $MV$-bundle $(MV)\times_M K\to M\backslash K$. We will consider this principal bundle as a Lie groupoid
$$(MV)\times_M K\rightrightarrows M\backslash K.$$
This groupoid carries a smooth $M$-action, trivial on the base $M\backslash K$, which is not by groupoid automorphisms. In particular, there is an $A$-equivariant $M$-principal bundle
$$\mathfrak{q}_0:(MV)\times_M K\to V\times_MK\cong T_H(M\backslash K).$$
\item We also consider the Lie groupoid $K\times_M K\rightrightarrows M\backslash K$ with source map $M(k,l)\mapsto Ml$, range map $M(k,l)\mapsto Mk$ and product $M(k,l)\cdot M(l,l'):=M(k,l')$. We can view $K\times_MK$ as the quotient of $K\times K$ by a diagonal $M$-action, so there is an $M$-principal bundle

$$\mathfrak{q}_1:K\times_MK\to M\backslash K\times M\backslash K.$$

\end{enumerate}

For the proofs of our main results, we repeatedly  use the following construction.

\begin{lemma}
\label{lnljnljknkjlnad}
For $X=P\backslash G=M\backslash K$, we form a Lie groupoid 
$$\widehat{\mathbb{T}}_HX\rightrightarrows M\backslash K \times \overline{A},$$
by set theoretically defining a groupoid as
$$\widehat{\mathbb{T}}_HX:=(MV)\times_M K\times \{0\}\dot{\bigcup} K\times_M K\times A\rightrightarrows M\backslash K \times \overline{A}.$$
It is endowed with a manifold structure using the property that $\varphi:\widehat{\mathbb{T}}_HX\to \C$ is smooth if and only if $\varphi$ is smooth on $K\times_M K\times A$ and there is a $\varphi_0\in C^\infty(V\times K\times \overline{A})$ such that 
$$\begin{cases}
\varphi(\kappa(ava^{-1}),k,a)&=\varphi_0(v,k,a), \; a\in A,\\
{}\\
\varphi(M(v,k),0)&=\varphi_0(v,k,0).
\end{cases}$$
Moreover, the $M$-principal bundles $\mathfrak{q}_0:(MV)\times_M K\to T_H(M\backslash K)$ and $\mathfrak{q}_1:K\times_MK\to M\backslash K\times M\backslash K$ induce a smooth $A$-equivariant $M$-principal bundle
$$\mathfrak{q}:\widehat{\mathbb{T}}_HX \to \mathbb{T}_HX.$$
\end{lemma}

\begin{proof}
An equivalent way to define the smooth structure on $\widehat{\mathbb{T}}_HX$ is to declare $K\times_MK\times A\subseteq \widehat{\mathbb{T}}_HX$ to be an open submanifold and to declare the following map a smooth embedding 
\begin{equation}
\label{psihatdef}
\hat{\psi}:V\times K\times \overline{A}\to \widehat{\mathbb{T}}_HX, 
\quad
(k,v,a)\mapsto\begin{cases}
(M(k\kappa(ava^{-1})^{-1},k),a)\; &a\in A,\\
{}\\
(M(v,k),0),\; &a=0.
\end{cases}
\end{equation}
Using the map $\psi$ from Proposition \ref{lknadlnad}, we see that $\mathfrak{q}$ is smooth and then it is clear that it is a smooth principal bundle. 

To show that $\widehat{\mathbb{T}}_HX$ is a Lie groupoid, we write it as a deformation to the normal cone similarly to Mohsen's Carnot groupoid construction \cite[Section 2 and 3]{mohsendef}. 

Consider the isotropy groupoid of $K\times_MK$
$$\mathfrak{I}:=\{\gamma\in K\times_MK: \mathsf{r}(\gamma)=\mathsf{s}(\gamma)\}= \{M(k,l): Mk=Ml\}.$$
We have that $\mathfrak{I}=\ker(\mu)$ where $\mu:K\times_MK\to K$ is the groupoid homomorphism $M(k,l)\mapsto k^{-1}l$, and moreover the groupoid $\mathfrak{I}\rightrightarrows M\backslash K$ has $\mathsf{r}=\mathsf{s}$ and can be viewed as the principal $M$-bundle $M\times_MK\to M\backslash K$. We form the Lie groupoid $\mathfrak{T}$ obtained from deformation to the normal cone along the closed groupoid inclusion $\mathfrak{I}\subseteq K\times_MK$
$$\mathfrak{T}:=\pmb{DNC}(K\times_MK,\mathfrak{I})\rightrightarrows M\backslash K\times [0,\infty).$$
The normal bundle $N_\mathfrak{I}$ to the inclusion $\mathfrak{I}\subseteq K\times_MK$ is $(\mathfrak{k}/\mathfrak{m})\times_M\mathfrak{I}=(\mathfrak{k}/\mathfrak{m})\times K$, where we use left translation to trivialize the tangent bundles. As a groupoid on $M\backslash K$ we have that 
$$N_\mathfrak{I}=((\mathfrak{k}/\mathfrak{m})\rtimes M)\times_M K.$$
We note that 
$$D\kappa:\mathfrak{v}\xrightarrow{\sim} \mathfrak{k}/\mathfrak{m},$$
is an $M$-equivariant isomorphism so we can identify 
$$N_\mathfrak{I}=(\mathfrak{v}\rtimes M)\times _M K,$$ 
as groupoids on $M\backslash K$. In particular,  $\mathfrak{T}$ is as a set theoretic groupoid given by  
$$\mathfrak{T}=((\mathfrak{v}\rtimes M)\times _M K)\times\{0\}\dot{\bigcup} K\times_M K\times (0,\infty).$$
The smooth structure on $\mathfrak{T}$ is obtained as in Proposition \ref{lknadlnad}, and can be defined by declaring $K\times_M K\times (0,\infty)\subseteq \mathfrak{T}$ to be an open and smooth subgroupoid, and the following map to be a smooth embedding 
$$\psi_\mathfrak{T}:(\mathfrak{v}\rtimes M)\times_M K\times [0,\infty)\to \mathfrak{T}, 
\quad 
(M(mX,k),t)\mapsto\begin{cases}
(Mk\mathrm{e}^{-tD\kappa(X)}m^{-1},Mk,t)\; &t>0,\\
{}\\
(M(mX,k),0),\; &t=0.
\end{cases}$$

The final step in the alternative construction of $\widehat{\mathbb{T}}_HX$ as a Lie groupoid is to note that we have a closed subgroupoid 
$$(\mathfrak{g}_{-1}\rtimes M)\times _M K\subseteq N_\mathfrak{I}=(\mathfrak{v}\rtimes M)\times _M K.$$
Following \cite[Section 2 and 3]{mohsendef}, we define the Lie groupoid 
$$\hat{\mathfrak{T}}:=\pmb{DNC}(\mathfrak{T},(\mathfrak{g}_{-1}\rtimes M)\times _M K\times \{0\})\rightrightarrows M\backslash K\times [0,\infty)^2.$$
By the arguments of \cite[Section 2]{mohsendef}, we have that 
$$\hat{\T}_H(M\backslash K)=\hat{\mathfrak{T}}|_{M\backslash K\times \{1\}\times  [0,\infty)},$$
and an argument as in Proposition \ref{lknadlnad} shows that the identity holds as smooth manifolds.
\end{proof}

\begin{definition}
The Lie groupoid $\widehat{\mathbb{T}}_HX\rightrightarrows X \times \overline{A}$ from Lemma \ref{lnljnljknkjlnad} will be called the extended parabolic tangent groupoid of $X=P\backslash G$.
\end{definition}

We note the following straightforward corollary to Lemma \ref{lnljnljknkjlnad}. We will use the convention of suppressing notation indicating two finite-dimensional unitary representations  $(\mathcal{H}_1,\sigma_1)$ and $(\mathcal{H}_2,\sigma_2)$ of $M$ into $\sigma=(\sigma_1,\sigma_2)$. 

\begin{cor}
\label{kjnkjnkjnad}
Assume that $(\mathcal{H}_1,\sigma_1)$ and $(\mathcal{H}_2,\sigma_2)$ are finite-dimensional unitary representations of $M$. Then the linear map 
\begin{align*}
\mathfrak{q}_\sigma&:C^\infty(\widehat{\mathbb{T}}_HX,\Hom(\mathcal{H}_1,\mathcal{H}_2))\to C^\infty(\mathbb{T}_HX,\Hom(r^*\mathbb{H}_1,s^*\mathbb{H}_2)),\\
\mathfrak{q}_\sigma&(\varphi):=\int_M \sigma_2(m)(m^*\varphi)\sigma_1(m)^*\rd m,
\end{align*}
is a well defined, continuous, $A$-equivariant, split surjection such that 
$$\mathrm{supp}(\mathfrak{q}_\sigma(\varphi))\subseteq \mathfrak{q}(\mathrm{supp}(\varphi)),$$
where $\mathfrak{q}:\widehat{\mathbb{T}}_HX\to \mathbb{T}_HX$ denotes the quotient map.
\end{cor}

We note here that 
$$\begin{cases}
\mathfrak{q}_\sigma(\varphi)(Mk,Mk',a)&=\int_M \sigma(m)\varphi(k,mk',a)\sigma(m)^*\rd m, \; a\in A,\\
{}\\
\mathfrak{q}_\sigma(\varphi)(v,Mk,0)&=\int_M \sigma(m)\varphi(v,mk,0)\sigma(m)^*\rd m.
\end{cases}$$
In particular, Corollary \ref{kjnkjnkjnad} is clear from Lemma \ref{lnljnljknkjlnad}.

\section{van Erp-Yuncken's Heisenberg calculus}
\label{subsecparatangb}

We now turn to describing the Heisenberg calculus following van Erp-Yuncken \cite{vanerpyunck}, Dave-Haller \cite{Dave_Haller1} and we will make heavy use of Ewert's approach \cite{ewertthesis, ewertfixed}, which in turn was inspired by \cite{debordskandav}. We follow the presentation in the literature \cite{Dave_Haller1,ewertfixed,goffkuz,goffbgg,vanerpyunck}. We will refer many of the details to the references where the techniques are presented in full. We note that there is a plethora of Heisenberg calculi \cite{andromoyu,bealsgreiner,christgelleretal,fischruzh,melinoldpreprint,pongemonograph,streetbook, taylorncom}. The reason we use van Erp-Yuncken's calculus is that it is well adapted to the adiabatic techniques we hinted at in Subsection \ref{subsecparatang} and that we will use to study the Szegö map. In the case at hand of filtered manifolds of depth 2, the van Erp-Yuncken calculus coincides with the Beals-Greiner calculus \cite{bealsgreiner} by \cite{yunckencouchet}, and for general filtered manifolds the van Erp-Yuncken calculus coincides by construction with the polyhomogeneous part of the Melin calculus \cite{melinoldpreprint}. 

\subsection{$\mathsf{r}$-fibered distributions}
We recall some basic notions, for more details see \cite{vanerpyunck}. For a Lie groupoid $\mathcal{G}\rightrightarrows \mathcal{G}^{(0)}$, an $\mathsf{r}$-fibered distribution $u$ from a vector bundle $E_1\to \mathcal{G}^{(0)}$ to $E_2\to \mathcal{G}^{(0)}$  is a left $C^\infty( \mathcal{G}^{(0)})$ linear map $u:C^\infty(\mathcal{G},\mathsf{s}^*E_1)\to C^\infty(\mathcal{G}^{(0)},E_2)$. We use the suggestive notation 
$$u(\varphi)(x)=\int_{\gamma\in \mathcal{G}_x}u(\gamma)\phi(\gamma),$$ 
for $x\in \mathcal{G}^{(0)}$, where $\mathcal{G}_x:=\mathsf{r}^{-1}(x)$. We write $\mathcal{E}_\mathsf{r}'(\mathcal{G};E_1,E_2)$ for the space of $\mathsf{r}$-fibered distributions. Likewise, we define the space of $\mathsf{s}$-fibered distributions. We write $\mathcal{E}_{\mathsf{r},\mathsf{s}}'(\mathcal{G};E_1,E_2)$ for the space of distributions that are both $\mathsf{r}$-fibered and $\mathsf{s}$-fibered. If $\mathcal{G}^{(0)}$ is compact, the Schwartz kernel theorem embeds $\mathcal{E}_{\mathsf{r},\mathsf{s}}'(\mathcal{G};E_1,E_2)\subseteq \mathcal{E}'(\mathcal{G};\mathsf{r}^*E_2\otimes \mathsf{s}^*E_1^*\otimes |\pmb{\lambda}_\mathsf{r}|)$ and without assuming that $\mathcal{G}^{(0)}$ is compact, $\mathcal{E}_{\mathsf{r},\mathsf{s}}(\mathcal{G};E_1,E_2)$ embeds into the properly supported elements of $\mathcal{D}'(\mathcal{G};\mathsf{r}^*E_2\otimes \mathsf{s}^*E_1^*\otimes |\pmb{\lambda}_\mathsf{r}|)$. Here $|\pmb{\lambda}_\mathsf{r}|$ denotes the space of densities on the $\mathsf{r}$-fibers, i.e. the density bundle of $\ker(D\mathsf{r})$. We also write 
$$C^\infty_p(\mathcal{G};r^*E_2\otimes s^*E_1^*\otimes |\pmb{\lambda}_r|):=\mathcal{E}_{\mathsf{r},\mathsf{s}}'(\mathcal{G};E_1,E_2)\cap C^\infty(\mathcal{G};r^*E_2\otimes s^*E_1^*\otimes |\pmb{\lambda}_r|),$$
for the space of smooth, properly supported functions. The groupoid operation extends to a convolution product $\mathcal{E}_{\mathsf{r},\mathsf{s}}'(\mathcal{G};E_2,E_3)\times \mathcal{E}_{\mathsf{r},\mathsf{s}}'(\mathcal{G};E_1,E_2)\to \mathcal{E}_{\mathsf{r},\mathsf{s}}(\mathcal{G};E_1,E_3)$ in which $C^\infty_p(\mathcal{G};\mathsf{r}^*E_2\otimes \mathsf{s}^*E_1^*\otimes |\pmb{\lambda}_\mathsf{r}|)$ is an ideal. 

Let us describe two important special cases, constituting the building blocks in the parabolic tangent groupoid. If $\mathcal{G}=X\times X$ is the pair groupoid on a compact manifold $X$, 
$$\mathcal{E}_\mathsf{r}'(X\times X)=C^\infty(X,\mathcal{D}'(X;|\pmb{\lambda}|)),$$
with the product on $\mathcal{E}_{\mathsf{r},\mathsf{s}}'(X\times X)$ being $u_1*u_2(x,z)=\int_X u_1(x,y)u_2(y,z)$. Similarly, for $\mathcal{G}=X\times X\times (0,\infty)$, then $\mathcal{E}_\mathsf{r}'(X\times X\times (0,\infty))=C^\infty(X\times (0,\infty),\mathcal{D}'(X;|\pmb{\lambda}|))$ with the product on $\mathcal{E}_{\mathsf{r},\mathsf{s}}'(X\times X\times (0,\infty))$ being $u_1*u_2(x,z,t)=\int_X u_1(x,y,t)u_2(y,z,t)$. If $\mathcal{G}=T_HX$ is the osculating Lie groupoid on a filtered manifold $X$, $\mathcal{E}_{\mathsf{r},\mathsf{s}}'(T_HX)$ consists of smooth families $u=u(x)$ of compactly supported distributions on the fibers $u(x)\in \mathcal{E}'(T_HX_x)$ with the product $[u_1*u_2](x)=u_1(x)*u_2(x)$ defined from fiberwise convolution on $T_HX_x$.

\subsection{Heisenberg pseudodifferential operators}

For a $\lambda>0$ acting on $\mathbb{T}_HX$ via the zoom action, we write $\lambda_*: \mathcal{E}_{\mathsf{r},\mathsf{s}}'(\mathbb{T}_HX;E_1,E_2)\to  \mathcal{E}_{\mathsf{r},\mathsf{s}}'(\mathbb{T}_HX;E_1,E_2)$ for the push forward along $\lambda$. We here warn the reader about our convention of the preceding subsection to consider density valued kernels. As above, we write $t$ for the coordinate parametrizing $[0,\infty)$ in $X\times [0,\infty)$.

\begin{definition}
For $m\in \C$, a filtered manifold $X$ and vector bundles $E_1,E_2\to X$, we define the space $\pmb{\Psi}^m_H(X;E_1,E_2)\subseteq \mathcal{E}_{\mathsf{r},\mathsf{s}}'(\mathbb{T}_HX;E_1,E_2)$ to consist of those $\mathsf{r}$-fibered distributions $\pmb{k}$ such that for any $\lambda>0$ acting via the zoom action on $\mathbb{T}_HX$, 
$$\lambda_*a-\lambda^ma\in C^\infty_p(\mathcal{G};\mathsf{r}^*E_2\otimes \mathsf{s}^*E_1^*\otimes |\pmb{\lambda}_\mathsf{r}|),$$
and whose wave front set is conormal to the space of units.

We define the space $\Psi^m_H(X;E_1,E_2)\subseteq \mathcal{E}_{\mathsf{r},\mathsf{s}}'(X\times X;E_1,E_2)$ of Heisenberg pseudodifferential operators of order $m$ as the image of the evaluation map 
$$\mathrm{ev}_{t=1}:\pmb{\Psi}^m_H(X;E_1,E_2)\to \mathcal{E}_{\mathsf{r},\mathsf{s}}'(X\times X;E_1,E_2).$$
In other words, an operator $T:C^\infty(X,E_1)\to C^\infty(X,E_2)$ is a Heisenberg pseudodifferential operator of order $m$ if there is a $\pmb{k}_T\in \pmb{\Psi}^m_H(X;E_1,E_2)$ such that $\mathrm{ev}_{t=1}\pmb{k}_T$ is the Schwartz kernel of $T$. The reader can note that $t^k\pmb{\Psi}^{m-k}_H(X;E_1,E_2)\subseteq \pmb{\Psi}^m_H(X;E_1,E_2)$ so $\Psi^{m-k}_H(X;E_1,E_2)\subseteq \Psi^m_H(X;E_1,E_2)$ for any $k\in \N$.

We define $\tilde{\Sigma}_H^m(X;E_1,E_2)\subseteq \mathcal{E}_{\mathsf{r},\mathsf{s}}'(T_HX;E_1,E_2)$ as the image of the evaluation map 
$$\mathrm{ev}_{t=0}:\pmb{\Psi}^m_H(X;E_1,E_2)\to \mathcal{E}_{\mathsf{r},\mathsf{s}}'(T_H X;E_1,E_2),$$
and the space of Heisenberg symbols of order $m$ as the space 
$$\Sigma_H^m(X;E_1,E_2):=\tilde{\Sigma}_H^m(X;E_1,E_2)/C^\infty_c(T_H X;\mathsf{r}^*E_2\otimes \mathsf{s}^*E_1^*\otimes |\pmb{\lambda}_\mathsf{r}|).$$
\end{definition}

Let us summarize the salient features of the Heisenberg calculus in a proposition. We refer its proof to the literature \cite{Dave_Haller1,ewertfixed,goffkuz,vanerpyunck}.

\begin{prop}
The convolution product on $\mathsf{r},\mathsf{s}$-fibered distributions induces product operations on all the spaces $\pmb{\Psi}^m_H(X;E_1,E_2)$, $\Psi^m_H(X;E_1,E_2)$, $\tilde{\Sigma}_H^m(X;E_1,E_2)$ and $\Sigma_H^m(X;E_1,E_2)$. Moreover, these spaces fit into a commuting diagram of multiplicative maps and exact rows and columns
\small
\[
\begin{CD}
@.0@.0 @.0@.@.\\
@.@VVV@VVV@VVV  \\
0 @>>>tC^\infty_p(\T_H X)@>>>C^\infty_p(\T_H X)@>{\mathrm{ev}_{t=0}}>>C^\infty_c(T_H X) @>>>0 \\
@.@VVV@VVV@VVV  \\
0 @>>>t\pmb{\Psi}_H^{m-1}(X)@>>>\pmb{\Psi}_H^{m}(X)@>{\mathrm{ev}_{t=0}}>>\tilde{\Sigma}_H^m(X) @>>>0 \\
@.@VVV@VVV@VVV  \\
0 @>>>\Psi_H^{m-1}(X)@>>>\Psi_H^{m}(X)@>{\sigma_H^m}>>\Sigma_H^m(X)  @>>>0 \\
@.@VVV@VVV@VVV  \\
@.0@.0 @.0@.@.\\
\end{CD}\]
\normalsize
where we for  clarity of notation dropped the vector bundles and density from the notations. Here the map 
$$\sigma_H^m:\Psi_H^{m}(X)\to \Sigma_H^m(X),$$
is defined by $\sigma_H^m(T):=\mathrm{ev}_{t=0}\pmb{k}_T+C^\infty_p(T_H X)$, for any $\pmb{k}_T\in \pmb{\Psi}^m_H(X)$ such that $\mathrm{ev}_{t=1}\pmb{k}_T$ is the Schwartz kernel of $T$.
\end{prop}

An element of $\Psi^{-\infty}(X;E_1,E_2):=\cap_{m} \Psi^m_H(X;E_1,E_2)$ is called a smoothing operator. It is clear from the definition that $\Psi^{-\infty}(X;E_1,E_2)=C^\infty(X\times X;\mathsf{r}^*E_2\otimes \mathsf{s}^*E_1^*\otimes |\pmb{\lambda}_\mathsf{r}|)$, justifying the term smoothing. In the same way as for the classical Hörmander calculus, the product in the complete symbol algebras $\Psi^{m}(X;E_1,E_2)/\Psi^{-\infty}(X;E_1,E_2)$ is local and $\Psi^m_H/\Psi^{-\infty}$ forms a soft sheaf of algebras on $X$.

\subsection{The symbol algebras on $P\backslash G$}
\label{symbollknlknad}
We can make the symbol algebras somewhat more concrete, and we do so in the case at hand $X=P\backslash G$. To do so, we first describe the fiberwise case -- the almost homogeneous, $\mathsf{r},\mathsf{s}$-fibered distribution on the nilpotent Lie group $V$ viewed as a groupoid over a point, so $\mathcal{E}_{\mathsf{r},\mathsf{s}}'(V)=\mathcal{E}'(V)$. The grading on $V$ is what is needed for the dilation action. Following the notation above, we write $C^\infty_c(V,|\pmb{\lambda}|)$ for the space of smooth, compactly supported densities on $V$. We use the notation 
$$\tilde{\Sigma}_H^mV:=\left\{k\in \mathcal{E}'(V): \begin{matrix}\lambda_*a-\lambda^ma\in C^\infty_c(V,|\pmb{\lambda}|), \; \forall \lambda>0,\\ \vspace{-2mm}&\\ 
 \; a|_{V\setminus \{0\}}\in C^\infty(V\setminus \{0\},|\pmb{\lambda}|)\end{matrix}\right\},$$
and $\Sigma_H^mV:=\tilde{\Sigma}_H^mV/C^\infty_c(V,|\pmb{\lambda}|)$. We use the notation $0\in V$ for the identity element. 

We write $P^mV\subseteq C^\infty(V,|\pmb{\lambda}|)$ for the space of density valued polynomials that are homogeneous of degree $m$. We note that if we have chosen a constant density $\rd v$ on $V$, we can write $p_m=q_m\rd v$ where $q_m$ is a polynomial on $V$. Since $p_m$ has order $m$ then $q_m$ as a polynomial must have order $-m-d$ where $d\in \N$ is the homogeneous dimension of $V$. That is, $d$ is the homogeneity degree of $\rd v$ and it is given by 
$$d=p+2q, \quad\mbox{for $p:=\dim\mathfrak{g}_\alpha$ and $q:=\dim\mathfrak{g}_{2\alpha}$.}$$ 
In particular, $p_m=0$ unless $m\in -d-\N$. We write $\mathfrak{S}^mV\subseteq \mathcal{D}'(V,|\pmb{\lambda}|)$ for the space of density valued distributions on $V$ that are smooth on $V\setminus\{0\}$ and are homogeneous of degree $m$.

\begin{lemma}
\label{adanlkn}
Let $m\in \C$ and assume that $V$ is a graded, simply connected, nilpotent Lie group equipped with a homogeneous length function $|\cdot|$. Write $d$ for the homogeneous dimension of $V$. Any $a\in \tilde{\Sigma}_H^mV$ can be written as 
$$a=\chi(a_m+p_m\log|\cdot|)+a_\infty,$$
where $\chi,a_\infty\in C^\infty_c(V)$ and $\chi=1$ near the identity element $0\in V$, $a_m\in \mathfrak{S}^mV$ and $p_m\in P^mV$. If $m\notin -d-\N$ then $p_m=0$ and $a_m$ is uniquely determined. If $m\in -d-\N$ then $p_m$ is uniquely determined but $a_m$ is uniquely defined only up to elements of $P^mV$.
 \end{lemma}

This lemma can be found as \cite[Lemma 3.8]{Dave_Haller1}. An important consequence of Lemma \ref{adanlkn} is a description of $\Sigma_H^mV$. Given an $a\in \Sigma_H^mV$ it is determined as 
$$a=a_m+p_m\log|\cdot|+C^\infty(V)$$
for an $a_m\in \mathfrak{S}^mV$ and a $p_m\in P^mV$. Since $\mathfrak{S}^mV\cap C^\infty(V)=P^mV$ we conclude the following proposition.

\begin{lemma}
Let $m\in \C$ and assume that $V$ is a graded, simply connected, nilpotent Lie group equipped with a homogeneous length function $|\cdot|$. Then the decomposition of Lemma \ref{adanlkn} induces an isomorphism
$$\Sigma_H^mV\cong (\mathfrak{S}^mV+P^mV\log|\cdot|)/P^mV.$$
\end{lemma}

In order to understand the algebraic properties of the symbol algebra $\Sigma_H^mV$, we consider its action on Schwartz functions on $V$. However, due to that the symbols are determined only up to $P^mV$, we need to consider space $\mathcal{S}_0(V)$ of Schwartz functions on $V$ that vanish to infinite order in the trivial representation. We write $\mathcal{S}(V)$ for the Frech\'{e}t space of Schwartz class functions on $V$, and more precisely, we define 
$$\mathcal{S}_0(V):=\left\{h\in \mathcal{S}(V): \int_V q(v)h(v)\rd v=0\; \mbox{for all polynomials}\; q\right\}.$$
If $p=q\rd v\in P^mV$, we have that $p*h(v)=\int_V q(vw^{-1})h(w)\rd w=0$, so we conclude the following. We use the notation $\mathcal{L}_{\mathcal{S}_0(V)}(\mathcal{S}_0(V))$ for the space of continuous linear maps on $\mathcal{S}_0(V)$ that commute with the right action of $\mathcal{S}_0(V)$ on itself via convolution.

\begin{prop}
\label{blablaprec}
Let $m\in \C$ and assume that $V$ is a graded, simply connected, nilpotent Lie group. Then the convolution product on $V$ produces a well defined injective $*$-homomorphism
$$\Sigma_H^mV\to \mathcal{L}_{\mathcal{S}_0(V)}(\mathcal{S}_0(V)).$$
\end{prop}

To further study the symbol algebra, we use noncommutative Fourier analysis. That is, for any unitary representation $\pi:V\to U(\mathfrak{h}_\pi)$ we can form the dense subspace
$$\mathcal{S}_0(\pi):=\pi(\mathcal{S}_0(V))\mathfrak{h}_\pi\subseteq \mathfrak{h}_\pi.$$
We can identify $\mathcal{S}_0(\pi)=\mathcal{S}_0(V,\mathfrak{h}_\pi)^V$ and topologize $\mathcal{S}_0(\pi)$ as such. Writing $\mathcal{L}(\mathcal{S}_0(\pi))$ for the algebra of linear continuous operators on $\mathcal{S}_0(\pi)$, we can extend $\pi$ to a $*$-homomorphism
$$\pi:\mathcal{L}_{\mathcal{S}_0(V)}(\mathcal{S}_0(V))\to \mathcal{L}(\mathcal{S}_0(\pi)), \quad \pi(T)(\pi(f)v):=\pi(Tf)v.$$
We note that if $\pi$ is irreducible and non-trivial, then $\mathcal{S}_0(\pi)=\mathcal{S}(\pi)$. 

\begin{lemma}
\label{christlemma}
For any unitary, irreducible, nontrivial representation $\pi:V\to U(\mathfrak{h}_\pi)$ we can represent 
$$\pi:\Sigma_H^mV\to \mathcal{L}(\mathcal{S}(\pi)).$$
Moreover, an element $a\in \Sigma_H^mV$ admits a two sided inverse $a^{-1}\in \Sigma_H^{-m}V$ if and only if for all unitary, irreducible, nontrivial representation $\pi:V\to U(\mathfrak{h}_\pi)$, $\pi(a):\mathcal{S}(\pi)\to \mathcal{S}(\pi)$ is injective with dense range. 
\end{lemma}

The first part of the lemma is clear from Proposition \ref{blablaprec} and the discussion following it. The second part is proven as \cite[Theorem 6.2]{christgelleretal}, see also \cite[Lemma 3.9]{Dave_Haller1}. We now return to the global situation on the compact Carnot manifold $X=P\backslash G$.  

\begin{prop}
\label{gequicad}
Let $m\in \C$. Consider the compact Carnot manifold $X=P\backslash G$ and, for $j=1,2$, homogeneous vector bundles $\mathbb{H}_j=\mathcal{H}_j\times_P G\to X$ for representations $\mathcal{H}_j$ of $P$ factoring over $MA$. It then holds that 
$$\tilde{\Sigma}_H^m(X;\mathbb{H}_1,\mathbb{H}_2)=C^\infty(P\backslash G; (\tilde{\Sigma}_H^mV \otimes \Hom(\mathcal{H}_1,\mathcal{H}_2))\times_PG),$$
and 
$$\Sigma_H^m(X;\mathbb{H}_1,\mathbb{H}_2)=C^\infty(P\backslash G; (\Sigma_H^mV \otimes \Hom(\mathcal{H}_1,\mathcal{H}_2))\times_PG).$$
In particular, we can identify $\Sigma_H^m(X;\mathbb{H}_1,\mathbb{H}_2)^G=(\Sigma_H^mV \otimes \Hom(\mathcal{H}_1,\mathcal{H}_2))^{MA}$.  
\end{prop}

We refer the proof of the proposition to the discussion in \cite[Section 4.7]{goffbgg}.

\subsection{Kernel structure}

The structural results for the symbol algebras in Subsection \ref{symbollknlknad} can asymptotically be used to describe also the Heisenberg pseudodifferential operators up to smoothing operators. We have the following description of the integral kernels of elements of $\Psi^m_H$ on $X=P\backslash G$.

\begin{lemma}
\label{lemkernel}
Let $m\in \C$. Consider the compact Carnot manifold $X=P\backslash G$ and, for $j=1,2$, homogeneous vector bundles $\mathbb{H}_j=\mathcal{H}_j\times_P G\to X$ for representations $\mathcal{H}_j$ of $P$ factoring over $MA$. Consider an operator 
$$T: C^\infty(X, \mathbb{H}_1)\to C^\infty(X, \mathbb{H}_2).$$
Then $T$ is a Heisenberg pseudodifferential operator of order $m$ if and only if the Schwartz kernel $k_T\in \mathcal{D}'(X\times X;\mathsf{r}^*\mathbb{H}_2\otimes \mathsf{s}^*\mathbb{H}_1^*\otimes |\pmb{\lambda}_\mathsf{r}|)$ of $T$ has the following kernel description. The distribution $k_T$ is smooth off the diagonal, and for any $x_0\in X$, and any precompact open set $U\subseteq V\cdot x_0$ in the nilpotent chart $V\cdot x_0$, we have on $U$ that 
$$k_T(x,y)=k(x,xy^{-1}) ,$$
where $k_U\in C^\infty(U,\mathcal{E}'(V;|\pmb{\lambda}|)\otimes \Hom(H_1,H_2))$ satisfies that for any $N\in \N$, there is an $l$ such that for $j=0,1,2,\ldots, l$, there are $k_{U,j}\in C^\infty(U,\tilde{\Sigma}_H^{m-j}V\otimes \Hom(\mathcal{H}_1,\mathcal{H}_2))$ such that 
$$k_U-\sum_{j=0}^l k_{U,j}\in C^\infty(U,C^N(V,\Hom(\mathcal{H}_1,\mathcal{H}_2))).$$
 In this case, we have in $U$ that  
 $$\sigma_H^{m}(T)=k_{U,0}+C^\infty(U,C^\infty_c(V;|\pmb{\lambda}|\otimes \Hom(\mathcal{H}_1,\mathcal{H}_2))).$$
 \end{lemma}

Lemma \ref{lemkernel} is well known, and we here simply present it in a format appropriate for our purposes. See more in \cite[Chapter 3.1.4]{pongemonograph} or \cite[Theorem 15.39]{bealsgreiner}.

\begin{proof}
The statement that $T$ is a Heisenberg pseudodifferential operator of order $m$ is local and the claimed equivalence is also local, so we can assume that that bundles are the trivial line bundle and that $T$ is compactly supported in $U$. By the definition of the Heisenberg calculus, $T$ will be an order $m$ operator in the calculus if the function $(0,\infty)\ni t\mapsto t^{j-\nu-s}k(x,\mathrm{exp}(t.X))\in C^\infty(U,\mathcal{E}'(\mathfrak{v}))$ extends to a $C^\infty$ function $[0,\infty)\ni t\mapsto t^{j-\nu-s} k(x,\mathrm{exp}(t.X))\in C^\infty(U,\mathcal{E}'(\mathfrak{v}))$. By Taylor expanding at $t=0$, we deduce the asymptotic expansion. 
 \end{proof}
 
 \subsection{Ewert's approach}
 \label{lknlknaddowow}
 
We will draw heavily on the ideas in Ewert's approach to the Heisenberg calculus \cite{ewertthesis, ewertfixed}. The approach stemmed from work of Debord-Skandalis \cite{debordskandav}. We note that Ewert \cite{ewertthesis, ewertfixed} uses a fixed point algebra construction for zeroth order operators; and for averaging smooth compactly supported functions to order zero one has to restrict the averaging procedure to elements vanishing to infinite order in the trivial representation. We approach the construction using a meromorphic extension in the order parameter, where no restriction is needed in the trivial representation. 

The idea can be explained in the simpler Euclidean case as follows: if $\varphi\in C^\infty_c(\R^n,|\pmb{\lambda}|)$ and $s\in \C$, then we can in a distributional sense form 
$$\pmb{k}_{\varphi}:s\mapsto \int_0^\infty \lambda^{-s}\lambda_*\varphi \frac{\rd \lambda}{\lambda}.$$
That is, for a test function $\phi\in C^\infty_c(\R^n)$, we have the meromorphic function
$$s\mapsto \pmb{k}_{\varphi}(s).\phi=\int_{\R^n}\int_0^\infty \lambda^{-s-n}\varphi(\lambda^{-1} x)\phi(x) \frac{\rd \lambda}{\lambda}.$$
Of course, $\pmb{k}_{\varphi}(s)$ is a priori only defined for $\mathrm{Re}(s)<<0$ but by standard arguments \cite[Chapter 3]{horI} extends by meromorphicity to $\C$ with potential poles only situated at $s\in n+\N$. A short computation gives that in fact $\pmb{k}_{\varphi}(s)$ is a homogeneous distribution if $s$ is a regular point. 

We now paraphrase Ewert's work \cite{ewertthesis, ewertfixed} that extends the above idea to the Heisenberg calculus. Following \cite[Chapter 9]{ewertthesis}, we write 
$$\mathcal{S}'(\T_HX; \Hom(\mathsf{r}^*E_1,\mathsf{s}^*E_2)\otimes |\pmb{\lambda}_\mathsf{r}|)\subseteq \mathcal{D}'(\T_HX; \Hom(\mathsf{r}^*E_1,\mathsf{s}^*E_2)\otimes |\pmb{\lambda}_\mathsf{r}|),$$ 
for the subspace of fiberwise tempered distributions.

\begin{thm}
\label{lknlknadthm}
Consider a compact Carnot manifold $X$ of homogeneous dimension $d$ and for $j=1,2$, two vector bundles $E_j\to X$. For any $\varphi\in C^\infty_c(\T_HX; \Hom(\mathsf{r}^*E_1,\mathsf{s}^*E_2)\otimes |\pmb{\lambda}_\mathsf{r}|)$, the function 
\begin{equation}
\label{kphiforgen}
\pmb{k}_{\varphi}:s\mapsto \int_0^\infty \lambda^{-s}\lambda_*\varphi \frac{\rd \lambda}{\lambda}\in \mathcal{S}'(\T_HX; \Hom(\mathsf{r}^*E_1,\mathsf{s}^*E_2)\otimes |\pmb{\lambda}_\mathsf{r}|),
\end{equation}
is well defined and holomorphic for $\mathrm{Re}(s)<<0$. Moreover, for any homogeneous $\chi\in C^\infty_p(\T_HX)$ with $\chi=1$ on a neighborhood of $T_HX\times \{0\}\dot{\cup}\Delta_M\times (0,\infty)$ we have $\chi\pmb{k}_{\varphi}(s)\in \pmb{\Psi}^s_H(X;E_1,E_2)$ for $\mathrm{Re}(s)<<0$. In fact, $\chi\pmb{k}_{\varphi}$ extends meromorphically to $s\in \C$ with potential poles for $s\in d+\N$ and the property that $\mathrm{FP}_{s=m}\chi\pmb{k}_{\varphi}(s)\in \pmb{\Psi}^m_H(X;E_1,E_2)$ for any $m\in \C$. 
\end{thm}

\begin{proof}
For  convenience of notation, we drop the vector bundles and the density bundle from the notations. Recall the construction of $\psi_\nabla$ in Equation \eqref{psinabladef}. We can write $\varphi=\psi^*_\nabla\varphi_0+\tilde{\varphi}$ where $\varphi_0\in C^\infty_c(T_HX\times[0,\infty))$ and $\tilde{\varphi}\in C^\infty_c(\T_HX)$ vanishes to infinite order at $T_HX\times \{0\}\subseteq \T_HX$. The fact that $\tilde{\varphi}$ vanishes to infinite order at $T_HX\times \{0\}$ implies that $\pmb{k}_{\tilde{\varphi}}$ depends holomorphically on $s$ and $\pmb{k}_{\tilde{\varphi}}(s)\in C^\infty_p$ for any $s$. Moreover, $\pmb{k}_{\psi^*_\nabla\varphi_0}=\psi^*_\nabla\pmb{k}_{\varphi_0}$, where $\pmb{k}_{\varphi_0}(s)\in \mathcal{D}'(T_HX\times [0,\infty))$ is defined ad verbatim $\pmb{k}_{\varphi}$. That is, writing $\varphi_0=\varphi_{00}\rd v$ for a scalar $\varphi_{00}$ we have that
\begin{equation}
\label{homoext}
[\pmb{k}_{\varphi_0}(s)](x,v,t)=\left[\int_0^\infty \lambda^{-s-d}\varphi_{00}(x,\lambda^{-1} v,\lambda t) \frac{\rd \lambda}{\lambda}\right] \rd v.
\end{equation}
If $\mathrm{Re}(s)<<0$ then in a distributional sense it is clear from a change of variables that $\lambda^s\pmb{k}_{\varphi_0}=\lambda_*\pmb{k}_{\varphi_0}$. To extend meromorphically in $s$, we note that for $u_s(\lambda):= \lambda^{-s-d}$, the distribution valued function $s\mapsto u_s$ extends meromorphically to a homogeneous distribution with simple poles in $s\in d+\N$, see for instance \cite{horI}.

Since $\varphi_0\in C^\infty_c(T_HX\times[0,\infty))$, the integral \eqref{homoext} defines a smooth function for $v\neq 0$ and the singular support of $\pmb{k}_{\varphi_0}(s)$ is a subset of $X\times [0,\infty)\subseteq T_HX\times[0,\infty)$. From almost homogeneity, we conclude that the wave front set of $\pmb{k}_{\varphi_0}(s)$ is conormal to the unit space $X\times [0,\infty)\subseteq T_HX\times[0,\infty)$.

We conclude that $\pmb{k}_{\varphi}(s)$, in the terminology of \cite[Chapter 9]{ewertthesis}, is an almost homogeneous fiberwise tempered distribution whose wave front set  is conormal to the unit space $X\times [0,\infty)\subseteq \T_HX$. By a similar argument as in the Euclidean setting, for any homogeneous $\chi\in C^\infty_p(\T_HX)$ with $\chi=1$ on a neighborhood of $T_HX\times \{0\}\dot{\cup}\Delta_M\times (0,\infty)$, $\chi\pmb{k}_{\varphi}(s)$ is an almost homogeneous $\mathsf{r}$-fibered distribution. Using the inversion mapping on $\T_HX$, we can argue mutatis mutandis to see that $\chi\pmb{k}_{\varphi}(s)$ is an almost homogeneous $\mathsf{s}$-fibered distribution. In conclusion, $\chi\pmb{k}_{\varphi}(s)\in \pmb{\Psi}^s_H$ for $s\notin d+\mathbb{N}$ and for $s\in d+\N$ we have that $\chi\pmb{k}_{\varphi}(s)\in \pmb{\Psi}^s_H$ by meromorphicity of the family.  
\end{proof}

\begin{remark}
\label{ljknkjlnad}
We can from $\varphi$ and $\chi$ form the operator $T_{\varphi,m}:=\mathrm{ev}_{t=1}\mathrm{FP}_{s=m}\chi\pmb{k}_{\varphi}(s)\in \Psi^m_H(X;E_1,E_2)$. Up to smoothing operators $T_{\varphi,m}$ is independent of $\chi$. We have that 
$$\sigma_H^m(T_{\varphi,m})=\mathrm{FP}_{s=m}\int_0^\infty \lambda^{-s}\lambda_*[\mathrm{ev}_{t=0}\varphi] \frac{\rd \lambda}{\lambda}+C^\infty_c.$$
In explicit coordinates, it means that 
$$\sigma_H^m(T_{\varphi,m})(x,v)=\mathrm{FP}_{s=m}\int_0^\infty \lambda^{-s-d}\varphi(x,\lambda ^{-1}v,0) \frac{\rd \lambda}{\lambda}+C^\infty_c.$$
\end{remark}

Let us specialize to the case $X=P\backslash G$ and give a technical corollary to Theorem \ref{lknlknadthm} that will be of use later in the paper. Recall the smooth mapping $\hat{\psi}:V\times K\times \overline{A}\to \widehat{\T}_HX$ from Equation \eqref{psihatdef} and that $\varphi\in C^\infty_c(\widehat{\T}_HX; \Hom(\mathcal{H}_1,\mathcal{H}_2)\otimes |\pmb{\lambda}_\mathsf{r}|)$ if and only if $\hat{\psi}^*\varphi$ and $\varphi|_{K\times_M K\times A}$ are smooth. We will also use the map $\mathfrak{q}_\sigma$ defined in Corollary \ref{kjnkjnkjnad} and in particular the extension by duality of $\mathfrak{q}_\sigma$ to $\mathsf{r},\mathsf{s}$-fibered distributions $\mathfrak{q}_\sigma:\mathcal{E}'_{\mathsf{r},\mathsf{s}}(\widehat{\mathbb{T}}_HX,\Hom(\mathcal{H}_1,\mathcal{H}_2))\to \mathcal{E}'_{\mathsf{r},\mathsf{s}}(\mathbb{T}_HX,\Hom(\mathsf{r}^*\mathbb{H}_1,\mathsf{s}^*\mathbb{H}_2))$. We write $\rd a$ for the Haar measure on $A$.

\begin{cor}
\label{lknlknad}
Let $m\in \C$ and $G$ be a connected semi-simple Lie group of real rank on. Write $A^+\subseteq A$ for its positive Weyl chamber with respect to the fixed simple, positive restricted root used to construct the parabolic tangent groupoid in Subsection \ref{subsecparatang}. Consider the compact Carnot manifold $X=P\backslash G$ and, for $j=1,2$, homogeneous vector bundles $\mathbb{H}_j=\mathcal{H}_j\times_P G\to X$ for representations $\mathcal{H}_j$ of $P$ factoring over $MA$. Consider an element $\rho\in C^\infty_c(\widehat{\mathbb{T}}_HX; \Hom(\mathcal{H}_1,\mathcal{H}_2)\otimes |\pmb{\lambda}_r|)$ and form the tempered distribution
$$
\hat{\pmb{k}}_{\rho}:s\mapsto \int_{A^+} a^{s\alpha }a_*\rho \rd a\in \mathcal{S}'(\widehat{\mathbb{T}}_HX; \Hom(\mathcal{H}_1,\mathcal{H}_2)\otimes |\pmb{\lambda}_r|).
$$
Then $\hat{\pmb{k}}_{\rho}$ is well defined and holomorphic for $\mathrm{Re}(s)<<0$ and extends meromorphically to $s\in \C$ with potential poles for $s\in \rho+\N\alpha$. Moreover, for any $\chi\in C^\infty_p(\T_HX)$ with $\chi=1$ on a neighborhood of $T_HX\times \{0\}\dot{\cup}\Delta_M\times (0,\infty)$,
$$\hat{\pmb{k}}_{\rho,\sigma}:=\chi \mathfrak{q}_\sigma(\hat{\pmb{k}}_{\rho})$$
defines a meromorphic function on $\C$ with $\hat{\pmb{k}}_{\rho,\sigma}(s)\in \pmb{\Psi}^s_H(X;\mathbb{H}_1,\mathbb{H}_2)$ that differs from $\chi \pmb{k}_{\mathfrak{q}_\sigma(\rho)}$ by an entire function with values in $C^\infty_p(\T_H X;\mathbb{H}_1,\mathbb{H}_2)$. 
\end{cor}

\begin{proof}
Using $\hat{\psi}$ (see Equation \eqref{psihatdef}) as in the proof of Theorem \ref{lknlknadthm}, we can reduce to showing that for $\rho_0\in C^\infty_c(V\times K\times \overline{A}; \Hom(\mathcal{H}_1,\mathcal{H}_2)\otimes |\pmb{\lambda}_r|)$ the tempered distribution
\begin{equation}
\label{kphiforwithahattie}
\hat{\pmb{k}}_{\rho_0}:s\mapsto \int_{A^+} a^{s\alpha }a_*\rho_0 \rd a\in \mathcal{S}'(V\times K\times \overline{A}; \Hom(\mathcal{H}_1,\mathcal{H}_2)\otimes |\pmb{\lambda}_r|),
\end{equation}
has the sought after properties upon pulling back along $\hat{\psi}^{-1}$. The integral \eqref{kphiforwithahattie} can be written over $(0,1]$ instead via the change of variables $a^{-\alpha}=\lambda$ as
$$\hat{\pmb{k}}_{\rho_0}:s\mapsto \int_{0}^1 \lambda^{-s}\lambda_*\rho_0 \frac{\rd \lambda}{\lambda}\in \mathcal{S}'(V\times K\times \overline{A}; \Hom(\mathcal{H}_1,\mathcal{H}_2)\otimes |\pmb{\lambda}_r|).$$
The proof therefore proceeds similarly to that of Theorem \ref{lknlknadthm}, but we highlight the differences, in particular arising from the integral being over $A^+\cong (0,1]$ instead of over $(0,\infty)$.

If $\mathrm{Re}(s)<<0$ then in a distributional sense the same change of variables as in Theorem \ref{lknlknadthm} implies that $\lambda^s\hat{\pmb{k}}_{\rho_0}-\lambda_*\hat{\pmb{k}}_{\rho_0}\in C^\infty_p$, where the integration over $(0,1]$ produces the defect from homogeneity. In particular, also in this situation $\hat{\pmb{k}}_{\rho_0}$ forms an almost homogeneous fiberwise tempered distribution and $\hat{\pmb{k}}_{(\hat{\psi}^{-1})^*\rho_0,\sigma}$ is an almost homogeneous $\mathsf{r},\mathsf{s}$-fibered distribution with wavefront set conormal to the unit space. The meromorphic extension in $s$ proceeds similarly, noting that for $u_s^+(\lambda):= \lambda^{-s-d}\chi_[0,1](\lambda)$ extends meromorphically in $s$ with simple poles in $s\in d+\N$, see for instance \cite{horI}.

\end{proof}

\begin{remark}
\label{nochi}
The reader should note that there is a large liberty in choosing $\chi$ in Corollary \ref{lknlknad}. In fact, we can take $\chi\in C^\infty_p(\T_HX)$ to satisfy $\chi=1$ at $t=1$.
\end{remark}

\begin{remark}
\label{ljknkjlnadbd}
Just as in Remark \ref{ljknkjlnad}, we can from $\rho$ and $\chi$ form the operator $\hat{T}_{\rho,\sigma,m}:=\mathrm{ev}_{t=1}\mathrm{FP}_{s=m}\chi \hat{\pmb{k}}_{\rho,\sigma}(s)\in \Psi^m_H(X;\mathbb{H}_1,\mathbb{H}_2)$. Up to smoothing operators, $\hat{T}_{\rho,\sigma,m}$ is independent of $\chi$ and in coordinates trivializing $T_H(P\backslash G)$ it holds that 
$$\sigma_H^m(\hat{T}_{\rho,\sigma,m})(x,v)=\mathrm{FP}_{s=m}\int_{A^+M} \lambda^{-s-d}\sigma_2(m)\rho(xm,m\lambda^{-1}v,0)\sigma_1(m)^* \frac{\rd \lambda}{\lambda}\rd m+C^\infty_c.$$
\end{remark}

\subsection{Analytic features of the Heisenberg calculus}

The operators from the Heisenberg calculus act nicely in a scale of Hilbert spaces that we call Heisenberg-Sobolev spaces. We recall their construction here and especially discuss the relation between $H$-ellipticity and Fredholm properties. These abstract considerations run in parallel to those in the classical situation of elliptic pseudodifferential operators from the Hörmander calculus; for an overview, see \cite[Chapter 1]{shubinbook}.

We construct the Heisenberg-Sobolev spaces for a homogeneous vector bundle $\mathbb{H}=\mathcal{H}\times_P B$ on $X=P\backslash G$. The reader can find the construction for general filtered manifolds in \cite{Dave_Haller1,goffkuz}. As in Section \ref{secback}, we decompose $\mathfrak{v}=\mathfrak{g}_{-\alpha}\oplus \mathfrak{g}_{-2\alpha}$. The subspace $\mathfrak{g}_{-\alpha}\subseteq \mathfrak{v}$ generates $\mathfrak{v}$ as a Lie algebra. We can equip $\mathfrak{g}_{-\alpha}$ with an inner product structure coming from $(v,v')=B(v,\theta(v'))$ where $B$ denotes the Killing form on $\mathfrak{g}$. Choose an ON-basis $X_1,\ldots, X_p$ for $\mathfrak{g}_{-\alpha}$ and consider 
$$\delta_H:=-\sum_{j=1}^p X_j^2\in \mathcal{U}(\mathfrak{v})\subseteq \mathcal{U}(\mathfrak{g}).$$
Since $M$ acts isometrically on $\mathfrak{g}_{-\alpha}$, $\delta_H$ is $M$-invariant. For a homogeneous vector bundle $\mathbb{H}=\mathcal{H}\times_P B$, we define the second order differential operator $\Delta_{\mathbb{H}}:C^\infty(P\backslash G,\mathbb{H})\to C^\infty(P\backslash G,\mathbb{H})$ by restricting $\delta_H$ acting on $C^\infty(K,\mathcal{H})=C^\infty(K)\otimes \mathcal{H}$ via the left regular action on $C^\infty(K)$ to its $M$-invariant subspace $C^\infty(K,\mathcal{H})^M=C^\infty(P\backslash G,\mathbb{H})$. Clearly, $\Delta_{\mathbb{H}}$ is $K$-invariant.

It follows from Hörmander's sum of squares theorem that $\Delta_{\mathbb{H}}$ is elliptic. It is not hard by a form argument to construct the Friedrichs extension of $\Delta_{\mathbb{H}}$ to a self-adjoint operator on $L^2(P\backslash G,\mathbb{H})$, by an abuse of notation we write also $\Delta_{\mathbb{H}}$ for its Friedrichs extension. By \cite{Dave_Haller2}, the complex powers $(1+\Delta_{\mathbb{H}})^m$, $m\in \C$, form a holomorphic family of invertible operators in the Heisenberg calculus, with $(1+\Delta_{\mathbb{H}})^m\in \Psi^{2m}_H(P\backslash G,\mathbb{H})$. We define the Heisenberg-Sobolev space of order $s\in \R$ as 
$$W^s_H(P\backslash G,\mathbb{H}):=(1+\Delta_{\mathbb{H}})^{-s/2}L^2(P\backslash G,\mathbb{H})\subseteq \mathcal{D}'(P\backslash G,\mathbb{H}).$$
We view $W^s_H(P\backslash G,\mathbb{H})$ as Hilbert spaces in the inner product making the operator $(1+\Delta_{\mathbb{H}})^{s/2}:W^s_H(P\backslash G,\mathbb{H})\to L^2(P\backslash G,\mathbb{H})$ unitary. Since $\Delta_{\mathbb{H}}$ is $K$-invariant, the restriction of the $G$-action on $W^s_H(P\backslash G,\mathbb{H})$ to $K$ is unitary. By sub-elliptic regularity for $\Delta_{\mathbb{H}}$, we have that $C^\infty(P\backslash G,\mathbb{H})\subseteq W^s_H(P\backslash G,\mathbb{H})$ is dense for any $s$. The following result follows from \cite[Proposition 3.9]{Dave_Haller1}.

\begin{thm}
\label{bddsds}
Let $s,t\in \R$. Consider the compact Carnot manifold $X=P\backslash G$ and, for $j=1,2$, homogeneous vector bundles $\mathbb{H}_j=\mathcal{H}_j\times_P G\to X$ for representations $\mathcal{H}_j$ of $P$ factoring over $MA$. Then any $T\in \Psi_H^m(X;\mathbb{H}_1,\mathbb{H}_2)$ extends by continuity to a continuous operator
$$T:W^s_H(P\backslash G,\mathbb{H})\to W^t_H(P\backslash G,\mathbb{H}),$$
as soon as $s+\mathrm{Re}(m)\leq t$ and if $s+\mathrm{Re}(m)< t$ then $T:W^s_H\to W^t_H$ is compact. 
\end{thm}

To study the mapping properties of Heisenberg pseudodifferential operators in the Heisenberg-Sobolev spaces, we will use the symbol calculus. We use the notation $\hat{V}$ for the unitary dual of $V$ and $\pmb{1}_V:V\to U(1)$ for the trivial representation. We introduce the notation 
$$\delta_{\mathbb{H}}:=\sigma_H^2(\Delta_\mathbb{H})\in \Sigma_H^2(X;\mathbb{H}).$$
We have that $1+\delta_{\mathbb{H}}$ admits complex powers within the symbol algebra. Indeed, we can take 
$$(1+\delta_{\mathbb{H}_2})^m:=\sigma_H^{2m}((1+\Delta_\mathbb{H})^m)\in \Sigma_H^{2m}(X;\mathbb{H}).$$
For $s\in \R$, $x\in X$ and $\pi\in \hat{V}\setminus \{\pmb{1}_V\}$, we define the Hilbert space 
$$\mathfrak{h}^s_\pi :=\pi((1+\delta_{\mathbb{H}_2})^{-s/2})\mathfrak{h}_\pi\subseteq \mathcal{S}'(\pi).$$
We in fact drop the $x$ since we can fit these Hilbert spaces into a homogeneous bundle over $P\backslash G$. We review this fact below. For $a\in \Sigma^m_HV$ and $\pi\in \hat{V}\setminus \{\pmb{1}_V\}$ the operator $\pi(a):\mathcal{S}(\pi)\to \mathcal{S}(\pi)$ extends by continuity to a bounded operator 
$$\pi(a):\mathfrak{h}^s_\pi \to \mathfrak{h}^{s-\mathrm{Re}(m)}_\pi.$$

\begin{lemma}
\label{alkdaldknad}
For $a\in \Sigma^m_H(X;\mathbb{H}_1,\mathbb{H}_2)$ the following statements are equivalent:
\begin{enumerate}
\item for any $x\in X$ and $\pi\in \hat{V}\setminus \{\pmb{1}_V\}$, 
$$\pi(a(x)):\mathcal{S}(\pi)\otimes \mathcal{H}_1\to \mathcal{S}(\pi)\otimes \mathcal{H}_2,$$ 
is injective with dense range;
\item for any $x\in X$ and $\pi\in \hat{V}\setminus \{\pmb{1}_V\}$, 
$$\pi(a(x)):\mathfrak{h}^s_\pi \otimes \mathcal{H}_1\to \mathfrak{h}^{s-\mathrm{Re}(m)}_\pi\otimes \mathcal{H}_2,$$ 
is invertible for some $s\in \R$;
\item for any $x\in X$ and $\pi\in \hat{V}\setminus \{\pmb{1}_V\}$, 
$$\pi(a(x)):\mathfrak{h}^s_\pi \otimes \mathcal{H}_1\to \mathfrak{h}^{s-\mathrm{Re}(m)}_\pi\otimes \mathcal{H}_2,$$ 
is invertible for all $s\in \R$;
\item $a$ admits a two-sided multiplicative inverse $a^{-1}\in \Sigma^{-m}_H(X;\mathbb{H}_2,\mathbb{H}_1)$.
\end{enumerate}
\end{lemma}

\begin{proof}
It is clear that 4) $\Rightarrow$ 3) $\Rightarrow$ 2) $\Rightarrow$ 1). Lemma \ref{christlemma} implies that 1) $\Rightarrow$ 4).
\end{proof}

\begin{definition}
An operator $T\in \Psi_H^m(X;\mathbb{H}_1,\mathbb{H}_2)$ is said to be $H$-elliptic if $\sigma_H^m(T)\in \Sigma^m_H(X;\mathbb{H}_1,\mathbb{H}_2)$ satisfies any of the equivalent conditions of Lemma \ref{alkdaldknad}.
\end{definition}

\begin{thm}
\label{adnajdn}
Assume that $T\in \Psi_H^m(X;\mathbb{H}_1,\mathbb{H}_2)$ for some $m\in \C$. Then the following are equivalent:
\begin{enumerate}
\item $T:W^s_H(P\backslash G,\mathbb{H}_1)\to W^{s-\mathrm{Re}(m)}_H(P\backslash G,\mathbb{H}_2)$ is Fredholm for some $s\in \R$;
\item $T:W^s_H(P\backslash G,\mathbb{H}_1)\to W^{s-\mathrm{Re}(m)}_H(P\backslash G,\mathbb{H}_2)$ is Fredholm for all $s\in \R$;
\item $T$ is $H$-elliptic.
\end{enumerate}
\end{thm}

\begin{proof}
It is clear that 2) $\Rightarrow$ 1). That 1) $\Rightarrow$ 3) follows from \cite{andromoyu}. Finally,  3) $\Rightarrow$ 2) since if $T$ is $H$-elliptic, we can by Lemma \ref{alkdaldknad} take an $R\in \Psi_H^{-m}(X;\mathbb{H}_2,\mathbb{H}_1)$ with $\sigma^{-m}_H(R)=\sigma_H^m(T)^{-1}\in \Sigma_H^{-m}(X;\mathbb{H}_2,\mathbb{H}_1)$, so $1-RT$ and $1-TR$ are order $-1$ and therefore compact by Theorem \ref{bddsds}.
\end{proof}

\subsection{Source projections and closed range}
\label{subsecclosed}
Our next goal is to find a characterization at the symbolic level for operators from the Heisenberg calculus to have closed range in the scale of Heisenberg-Sobolev spaces. We describe the symbol of its source projection. 

To do so, we first discuss how the fiberwise represented Heisenberg-Sobolev spaces $\mathfrak{h}^s_\pi $ fit together into a homogeneous bundle over $P\backslash G$. This is clear if $\mathfrak{g}_{-2\alpha}=0$, i.e. when $V$ is abelian. For the sake of this discussion, we assume that $\mathfrak{g}_{-2\alpha}\neq 0$.  By Kirillov's orbit method \cite{Corwin_Greenleaf,kirillovbook}, we have that 
\begin{equation}
\label{kiridkda}
\hat{V}=(\mathfrak{g}_{2\alpha}\setminus \{0\})\dot{\cup}\mathfrak{g}_{\alpha}.
\end{equation}
Here we identify $\mathfrak{g}_{\alpha}=\mathfrak{g}_{-\alpha}^*$ and $\mathfrak{g}_{2\alpha}=\mathfrak{g}_{-2\alpha}^*$ via the Killing form. The only fact we shall need about the topology of $V$ is that $\mathfrak{g}_{2\alpha}\setminus \{0\}$ is a dense, Zariski open subset. For a non-zero $\xi\in\mathfrak{g}_{2\alpha}$, the Kirillov form $\omega_\xi$ has annihilator $\mathfrak{g}_{-2\alpha}$ and is therefore non-degenerate on $\mathfrak{g}_{-\alpha}$. The representations in $(\mathfrak{g}_{2\alpha}\setminus \{0\})\subseteq \hat{V}$ come from the orbit of some $\xi\in \mathfrak{g}_{2\alpha}\setminus\{0\}$ and are flat, so we call representations from $(\mathfrak{g}_{2\alpha}\setminus \{0\})\subseteq \hat{V}$ flat orbit representations.

Using the Killing form, we can identify forms with matrices and on $\mathfrak{g}_{-\alpha}$ construct the smooth family $J_\xi:=\omega_\xi|\omega_\xi|^{-1}$ of complex structures adapted to the metric $g_\xi:=|\omega_\xi|$, both parametrized by $\xi\in \mathfrak{g}_{2\alpha}\setminus \{0\}$. As such, the real vector bundle
$$\mathcal{E}:=\mathfrak{g}_{-\alpha}\times  \mathfrak{g}_{2\alpha}\setminus \{0\}\to  \mathfrak{g}_{2\alpha}\setminus \{0\},$$
forms a complex vector bundle with a hermitean structure. By the standard construction, we can form the Fock bundle which is the locally trivial bundle of Hilbert spaces 
$$\mathcal{F}\to \mathfrak{g}_{2\alpha}\setminus \{0\},\quad\mbox{by}\quad \mathcal{F}:=\bigoplus_{n=0}^\infty \mathcal{E}^{\otimes_{\rm sym} n}.$$
Using the Bargman representation, $\mathcal{F}$ the property that the fiber over $\xi$ is the representation space corresponding to $\xi$ under the Kirillov map \eqref{kiridkda}. The bundle $\mathcal{F}$ is $MA$-equivariant. We can form the fiber bundle 
$$q:\Gamma:=(\mathfrak{g}_{2\alpha}\setminus \{0\})\times_{P}G\to P\backslash G,$$ 
with fiber $\mathfrak{g}_{2\alpha}\setminus \{0\}$. We can also form the locally trivial bundle of Hilbert spaces over the total space $\Gamma$
$$\mathcal{F}_{P\backslash G}:=\mathcal{F}\times_P G\to \Gamma.$$
The bundle $\mathcal{F}_{P\backslash G}\to \Gamma$ parametrizes all flat orbit representations of the osculating Lie groupoid $T_X(P\backslash G)$.

By the same token, we have for any $s\in \R$ a locally trivial bundle of Hilbert spaces over the total space $\Gamma$
$$\mathcal{F}^s_{P\backslash G}:=(1+\delta_{\mathbb{H}_2})^{-s/2}\mathcal{F}_{P\backslash G}\to \Gamma.$$
The discussion above implies that we can embed 
$$\Sigma_H^m(X;\mathbb{H}_1,\mathbb{H}_2)\hookrightarrow C^\infty(\Gamma,\Hom(\mathcal{F}^s_{P\backslash G}\otimes q^*\mathbb{H}_1,\mathcal{F}^{s-\mathrm{Re}(m)}_{P\backslash G}\otimes q^*\mathbb{H}_2))$$
for any $s\in \R$. Here we by $\Hom(\mathcal{F}^s_{P\backslash G}\otimes q^*\mathbb{H}_1,\mathcal{F}^{s-\mathrm{Re}(m)}_{P\backslash G}\otimes q^*\mathbb{H}_2)$ mean the locally trivial bundle of bounded linear operators between the fibers.

\begin{lemma}
\label{alknalkdn}
Consider an element $a\in \Sigma^m_H(X;\mathbb{H}_1,\mathbb{H}_2)$. The following conditions are equivalent:
\begin{enumerate}
\item For any $s\in \R$, the bundle map 
$$a:\mathcal{F}_{P\backslash G}^s\otimes q^*\mathbb{H}_1\to \mathcal{F}_{P\backslash G}^{s-\mathrm{Re}(m)}\otimes q^*\mathbb{H}_2,$$ 
has closed range.
\item There exists an operator $T\in \Psi^m_H(P\backslash G;\mathbb{H}_1,\mathbb{H}_2)$ with 
$$\sigma^m_H(T)=a,$$
that has closed range as an operator $T:W^s_H(P\backslash G,\mathbb{H}_1)\to W^{s-\mathrm{Re}(m)}_H(P\backslash G,\mathbb{H}_2)$ for some $s\in \R$,'
\end{enumerate}
If any of these conditions hold, for any $T\in (\sigma_H^m)^{-1}(a)$, the source projection $P_T$ of $T$ in $L^2$ has $P_T\in \Psi^0_H(X;\mathbb{H}_1)$ and its principal symbol $\sigma_H^0(P_T)$ is determined by the fact that for any $x\in X$ and for any flat orbit representation $\pi$ we have that $\pi(\sigma_H^0(P_T)(x))$ is the source projection of $\pi(a(x))$.
\end{lemma}

\begin{proof}
By order reduction, we can reduce to the case that $s_0=m=0$. 

We start by proving that 1)$\Rightarrow$ 2). If the bundle map $a:\mathcal{F}_{P\backslash G}\otimes q^*\mathbb{H}_1\to \mathcal{F}_{P\backslash G}\otimes q^*\mathbb{H}_2$ has closed range, compactness of $S(\mathfrak{g}_{2\alpha})$ ensures that there is an $\epsilon>0$ such that $\lambda+a^*a$ is invertible in $\Sigma^0_H(X;\mathbb{H}_1)$ for all $0<|\lambda|<\epsilon$. Since the algebra of zeroth order operators in the Heisenberg calculus is closed under holomorphic functional calculus \cite{pongecrelle, vanerppol}, we can form the Riesz projector
$$p:=\frac{1}{2\pi i}\int_{|\lambda|=\epsilon/2} (\lambda+a^*a)^{-1}\mathrm{d}\lambda\in \Sigma^0_H(X;\mathbb{H}_1).$$ 
By construction $p:\mathcal{F}_{P\backslash G}\otimes q^*\mathbb{H}_1\to \mathcal{F}_{P\backslash G}\otimes q^*\mathbb{H}_1$ projects onto the range of $a^*$. Therefore, 
$$(1-p)\oplus a:\mathcal{F}_{P\backslash G}\otimes q^*\mathbb{H}_1\to \mathcal{F}_{P\backslash G}\otimes q^*(\mathbb{H}_1\oplus \mathbb{H}_2),$$
is injective and by \cite[Lemma 3.9]{Dave_Haller1} (cf. Lemma \ref{alkdaldknad}), there exists a $b=(b_0\oplus b_1)^T\in \Sigma^0_H(X;\mathbb{H}_1\oplus \mathbb{H}_2,\mathbb{H}_1)$ such that $b_0(1-p)+b_1a=1$. We can conclude that $ab_1=p$ by multiplying with $p$ from the right. 

Assume now that $T_0\in \Psi^0_H(X;\mathbb{H}_1,\mathbb{H}_2)$ has $\sigma^0_H(T)=a$. We can lift $p$ to an operator $P\in \Psi^0_H(X;\mathbb{H}_1)$ that we by standard algebraic tricks can ensure is a projection modulo $\Psi^{-1}(X;\mathbb{H}_1)$. We have that $TP-T\in \Psi^{-1}(X;\mathbb{H}_1,\mathbb{H}_2)$. We lift $b_1$ to an operator $B\in \Psi^0_H(X;\mathbb{H}_2,\mathbb{H}_1)$ with $B-PB\in \Psi^{-1}_H(X;\mathbb{H}_2,\mathbb{H}_1)$. We then have that $TB-P\in \Psi^{-1}(X;\mathbb{H}_1)$. We can again by standard algebraic tricks modify $T_0$, $B$ and $P$ to operators with $P=P^*$, $P^2-P$ smoothing and $T_0B-P$ smoothing. Modifying $T_0$, $B$ and $P$ by appropriate smoothing operators we can ensure that $P\in \Psi^0_H(X;\mathbb{H}_1)$ is a projection projecting onto the range of $T$, and so $T$ has closed range. This proves 2).

We now outline the argument that 2)$\Rightarrow$ 1), it follows similar ideas as the implication 1)$\Rightarrow$ 2). If 3) holds, we can argue as above but at the operator theoretic level constructing the range projection of $T$ in $\Psi^0_H(X;\mathbb{H}_1)$. Now using the one-sided version of Theorem \ref{adnajdn} (see \cite[Section 3.2]{Dave_Haller1} or \cite{andromoyu}), we can reverse the argument above to show that the associated bundle map $a:\mathcal{F}_{P\backslash G}\otimes q^*\mathbb{H}_1\to \mathcal{F}_{P\backslash G}\otimes q^*\mathbb{H}_2$ has closed range and so 1) holds.
\end{proof}

\begin{remark}
It is not possible to strengthen Lemma \ref{alknalkdn} to characterizing closed range by properties of the symbol. Indeed, if $a=0$ the only elements $T\in (\sigma^m_H)^{-1}\{a\}$ with closed range as an operator $T:W^s_H(P\backslash G,\mathbb{H}_1)\to W^{s-\mathrm{Re}(m)}_H(P\backslash G,\mathbb{H}_2)$ are the finite rank smoothing operators $T$.
\end{remark}

\section{Mapping properties of the Szegö map}
\label{aldknadln}

We now turn to studying the Szegö map, and to the proof of Theorem \ref{conja}. The proof uses the extended parabolic tangent groupoid of Subsection \ref{subsecparatang}, in which we set up the techniques for proving Theorem \ref{conjb}. Finally, in Subsection \ref{lknlnad} we relate $S_{\mathcal{V}}^\dagger S_{\mathcal{V}}$ to Knapp-Stein intertwiners to show closed range of $S_\mathcal{V}$. The relationship to Knapp-Stein intertwiners and Lemma \ref{alknalkdn} is used in Corollary \ref{blablaks} to conclude symbolic information on the source projection $P_\mathcal{V}$ of $S_\mathcal{V}$.

\subsection{Preliminary considerations}
\label{boundviaks}

Let us fix an irreducible unitary representation $(\mathcal{V},\tau)$ of $K$ with highest weight $\lambda$ and we let $(\mathcal{H},\sigma)$ denote the unitary $M$-representation generated by $\lambda$. For any $\nu\in \mathfrak{a}_\C^*$, we obtain a $P$-action on $\mathcal{H}$ via $(man).v=a^\nu\sigma(m)v$ and the homogeneous vector bundle $\mathbb{H}_\nu:=\mathcal{H}\times_PG\to P\backslash G$. We think of the $G$-space $C^\infty(P\backslash G;\mathbb{H}_\nu)$ as
$$C^\infty(P\backslash G;\mathbb{H}_{\nu_{\mathcal{V}}}):=C^\infty( G;\mathcal{H})^P.$$
We consider the Szegö map 
$$S_{\mathcal{V}}:C^\infty(P\backslash G;\mathbb{H}_{\nu_{\mathcal{V}}})\to C^\infty(K\backslash G,\mathbb{V}),$$
which is the $G$–equivariant map
\begin{equation}
\label{popskdps}
S_{\mathcal{V}} f(g)=\int_K \tau(k)^{-1}f(kg)\rd k=\int_K\mathsf{a}(lg^{-1})^{\nu_{\mathcal{V}}}\tau(\kappa(lg^{-1}))f(l)\rd l
\end{equation}
The reader should not that $C^\infty(P\backslash G;\mathbb{H}_\nu)$ as a Frech\'{e}t space and $L^2(P\backslash G;\mathbb{H}_\nu)$ as a Hilbert space are independent of $\nu$. 

To study the analytic aspects of $S_{\mathcal{V}}$, we will use the $KA^+K$-description of $G$. Here $A^+\subseteq A$ denotes the closed, positive Weyl chamber for our restricted root $\alpha$. We have that $KA^+K= G$, see \cite[Theorem 8.6, page 323]{HelgasonDGLGSS}. It follows from \cite[Proposition 5.28, page 141]{knappbook} that the restriction along the inclusion $A^+K\subseteq G$ induces an isometry
$$\iota:L^2(K\backslash G;\mathbb{V})\to L^2(A^+K;\mathcal{V}, h),$$
for the weight 
$$h(ak)=h(a)=\sinh(\alpha(H(a))^p\sinh(2\alpha(H(a))^{q},$$
where $p:=\dim\mathfrak{g}_\alpha$ and $q:=\dim\mathfrak{g}_{2\alpha}$. We tacitly use the Haar measures on $A$ and $K$. We can write 
\begin{equation}
\label{iotasgnu} 
\iota S_{\mathcal{V}} f(ak)=\int_K\mathsf{a}(lk^{-1}a^{-1})^{\nu_{\mathcal{V}}}\tau(\kappa(lk^{-1}a^{-1}))f(l)\rd l=a^{-\nu_{\mathcal{V}}} \int_K \mathcal{G}_{\mathcal{V}}(lk^{-1},a)f(l)\rd l,
\end{equation}
where $\mathcal{G}_\mathcal{V}(l,a):=\mathfrak{G}_\mathcal{V}(ala^{-1})$ and $\mathfrak{G}_{\mathcal{V}}(g):=\mathsf{a}(g)^{\nu_{\mathcal{V}}}\tau(\kappa(g))$.

The discussion in Remark \ref{globaadad} shows that $S_{\mathcal{V}}:C^\infty(P\backslash G;\mathbb{H}_{\nu_{\mathcal{V}}})\to L^2(K\backslash G,\mathbb{V})$ allowing us to form the formal adjoint of $\iota S_{\mathcal{V}}$. The formal adjoint is the $G$–equivariant map that we by an abuse of notation write as
$$S_{\mathcal{V}}^\dagger :L^2(A^+K;\mathcal{V}, h)\to \mathcal{D}'(P\backslash G;\mathbb{H}_{-\nu_{\mathcal{V}}})=\mathcal{D}'(K;\mathcal{H})^M,$$
and that takes the form
\begin{align}
\label{popskdpsdual}
S_{\mathcal{V}}^\dagger  u(k)=&\int_{A^+K}a^{-\nu_{\mathcal{V}}}\mathcal{G}_{\mathcal{V}}^*(lk^{-1},a)u(al)h(a)\rd a\rd l=\\
\nonumber
=&\int_{(0,1)\times K}\lambda^{s_{\mathcal{V}}}\mathcal{G}_{\mathcal{V}}^*(lk^{-1},\lambda)u(\lambda l)h(t))\frac{\rd \lambda}{\lambda}\rd l,\quad \mbox{for}\quad \nu_{\mathcal{V}}=s_{\mathcal{V}}\alpha.
\end{align}
Here $\mathcal{G}_{\mathcal{V}}^*(l,a):=\overline{\mathcal{G}_{\mathcal{V}}^*(l^{-1},a)}$ and we in the second line use the variable $\lambda=a^{-\alpha}$. We can in particular form the operator
$$
S_{\mathcal{V}}^\dagger  S_{\mathcal{V}}:C^\infty(P\backslash G;\mathbb{H}_{\nu_{\mathcal{V}}})\to \mathcal{D}'(P\backslash G;\mathbb{H}_{-\nu_{\mathcal{V}}}).
$$
Since the $L^2$-pairing $W^{-s_\mathcal{V}}_H(P\backslash G,\mathbb{H}_{\nu_\mathcal{V}})\times W^{s_\mathcal{V}}_H(P\backslash G,\mathbb{H}_{-\nu_\mathcal{V}})\to \C$ is perfect, we conclude the following. 

\begin{prop}
\label{stararsta}
Write $\nu_{\mathcal{V}}=s_{\mathcal{V}}\alpha$. The $G$-equivariant mapping $S_{\mathcal{V}}:C^\infty(P\backslash G;\mathbb{H}_{\nu_{\mathcal{V}}})\to L^2(K\backslash G,\mathbb{V})$ extends to a continuous mapping $S_{\mathcal{V}}:W^{-s_\mathcal{V}}_H(P\backslash G,\mathbb{H}_{\nu_\mathcal{V}})\to L^2(K\backslash G,\mathbb{V})$ if and only if $S_{\mathcal{V}}^\dagger  S_{\mathcal{V}}:C^\infty(P\backslash G;\mathbb{H}_{\nu_{\mathcal{V}}})\to \mathcal{D}'(P\backslash G;\mathbb{H}_{-\nu_{\mathcal{V}}})$ extends by continuity to a continuous operator 
$$S_{\mathcal{V}}^\dagger  S_{\mathcal{V}}:W^{-s_\mathcal{V}}_H(P\backslash G,\mathbb{H}_{\nu_\mathcal{V}})\to W^{s_\mathcal{V}}_H(P\backslash G,\mathbb{H}_{-\nu_\mathcal{V}}).$$
\end{prop}

We also define the family of operators indexed by $\nu\in \mathfrak{a}_\C^*$ given by
\begin{align}
\label{lnueq}
L_{\nu}f(k)=&\int_{A^+}\int_{K}\int_K a^{-2\nu}  \mathcal{G}_{\mathcal{V}}^*(l_1k^{-1},a)\mathcal{G}_{\mathcal{V}}(l_2l_1^{-1},a)f(l_2)h(a)\rd l_1\rd l_2\rd a=\\
=&\int_{0}^1\int_{K}\int_K \lambda^{2s}  \mathcal{G}_{\mathcal{V}}^*(l_1k^{-1},\lambda)\mathcal{G}_{\mathcal{V}}(l_2l_1^{-1},\lambda)f(l_2)h(\lambda)\rd l_1\rd l_2\frac{\rd \lambda}{\lambda},
\end{align}
for $\nu=s\alpha$ and $\lambda=a^{-\alpha}$. A priori, $L_\nu$ is only defined for $\mathrm{Re}(\nu)>>0$ and for this range it produces a holomorphic family of operators 
$$
L_{\nu}:C^\infty(P\backslash G;\mathbb{H}_{\nu_{\mathcal{V}}})\to \mathcal{D}'(P\backslash G;\mathbb{H}_{-\nu_{\mathcal{V}}}),
$$
due to it being given by an absolutely convergent integral operator. To understand $L_\nu$ and to make the formal identity $S_{\mathcal{V}}^\dagger  S_{\mathcal{V}}=L_{\nu_{\mathcal{V}}}$ rigorous, we use the extended parabolic tangent groupoid.

\subsection{The Szegö map and the extended parabolic tangent groupoid}
\label{boundviaext}

We now describe an approach with the extended parabolic tangent groupoid to showing that $S_{\mathcal{V}}$ is bounded as stated in Theorem \ref{conja}. We start with an elementary observation connecting the Szegö map to the extended parabolic tangent groupoid. We use the notation 
$$\mathcal{G}_\nu(l,a):=\mathfrak{G}_\nu(ala^{-1}),$$ 
where $\mathfrak{G}_\nu(g):=\mathsf{a}(g)^{\nu}\tau(\kappa(g))$ so $\mathcal{G}_{\mathcal{V}}=\mathcal{G}_{\nu_{\mathcal{V}}}$. The reader should note that $\mathcal{G}_\nu$ relates to $\iota P_\nu$ (with $P_\nu$ as in Equation \eqref{kwszeggeneric}, see also Corollary \ref{remarkkwszeggeneric}) as $\mathcal{G}_{\mathcal{V}}$ relates to $\iota S_{\mathcal{V}}$ in Equation \eqref{iotasgnu}. We shall use the notation $\rd\,(Ml)$ for the $K$-invariant density on $M\backslash K$ normalized to unit volume.

\begin{lemma}
\label{lknlnadarrr}
For any $\nu\in \mathfrak{a}_\C^*$, the density valued function 
$$K\times_MK\times A\ni (k,l,a)\mapsto \mathcal{G}_\nu(lk^{-1},a)\rd\,(Ml),$$
can be written as
\begin{equation}
\label{gluingdondonda}
\mathcal{G}_\nu(lk^{-1},a^{-1})\rd\,(Ml)=a^{2\rho}[a_*\pmb{G}_\nu](k,l,1),
\end{equation}
for a smooth, density valued function $\pmb{G}_\nu\in C^\infty(\hat{\T}_H(P\backslash G),\Hom(\mathcal{H},\mathcal{V})\otimes |\pmb{\lambda}_r|)$ with 
$$\pmb{G}_\nu(k,v,0)=\tau(\kappa(v))\rd v.$$ 
\end{lemma}

\begin{proof}
In terms of $\lambda=a^{-\alpha}$, we can rewrite Equation \eqref{gluingdondonda} as
$$\lambda^{p+2q}\mathcal{G}_\nu(lk^{-1},\lambda^{-1})\rd\,(Ml)=[\lambda_*\pmb{G}_\nu](k,l,1).$$
Clearly, if such a function  $\pmb{G}_\nu\in C^\infty$ exists, then for $a\in A$, 
$$\pmb{G}_\nu(k,l,a)=\mathcal{G}_\nu(lk^{-1},a)\rd\,(Ml)=\mathfrak{G}_\nu(a^{-1} lk^{-1}a)\rd\,(Ml).$$ 
We shall show that $\pmb{G}_\nu$ defined in this way for $a\in A$ extends by continuity to a smooth function on $\hat{\T}_H(P\backslash G)$. For $a\in A$, we compute that 
$$\hat{\psi}^*\pmb{G}_\nu(x,v,a)=\mathfrak{G}_\nu(a^{-1}\kappa(ava^{-1})^{-1}a)\rd v.$$
We have that $\mathfrak{G}_\nu$ is a smooth function. The function 
$$\tilde{\kappa}:V\times A\to G, \quad \tilde{\kappa}(v,a):= a^{-1}\kappa(ava^{-1})^{-1}a,$$
extends to a smooth function $V\times \overline{A}\to G$ with $\tilde{\kappa}(v,0)=\kappa(v)$. We conclude that $\hat{\psi}^*\pmb{G}_\nu(x,v,a)=\mathfrak{G}_\nu(\tilde{\kappa}(v,a))$ extends to a smooth function on $V\times K\times \overline{A}$, so $\pmb{G}_\nu$ is smooth by Lemma \ref{lnljnljknkjlnad}.
\end{proof}

We note now that 
$$h(a)=\sinh(\alpha(H(a))^p\sinh(2\alpha(H(a))^{q}=a^{2\rho}\tilde{h}(a),$$
where $\tilde{h}\in C^\infty(\overline{A})$. Indeed, in $\lambda=a^{-\alpha}$, we can write 
$$h(\lambda)=\sinh(\log \lambda)^p\sinh(2\log \lambda)^{q}=\lambda^{-(p+2q)}\tilde{h}(\lambda),\quad\mbox{for}\quad \tilde{h}(\lambda)=\frac{(\lambda^2-1)^p(\lambda^4-1)^q}{2^{p+q}}.$$

If we take $\chi\in C^\infty_c(\overline{A})$ being identically $1$ on $A^+$, we can therefore by Lemma \ref{lknlnadarrr} form the element
$$\pmb{\mathcal{G}}_\nu:=\chi\tilde{h}\pmb{G}_\nu^**\pmb{G}_\nu \in C^\infty_c(\hat{\T}_H(P\backslash G),\End(\mathcal{H})\otimes |\pmb{\lambda}_r|),$$
where the convolution is in the parabolic tangent groupoid $\hat{\T}_H(P\backslash G)$. We can therefore write the operator $L_\nu$ from Equation \eqref{lnueq} as acting on $f\in C^\infty(P\backslash G,\mathbb{H}_\nu)=C^\infty(K,\mathcal{H})^M$ via
\begin{align}
\label{aldjnakdjna}
L_{\nu}f(k)=&\mathrm{ev}_{t=1}\int_{A^+K} a^{-2\nu} [a_*\pmb{\mathcal{G}}_{\mathcal{V}}](k,l,t)f(l)\rd a=\\
\nonumber 
=&\mathrm{ev}_{t=1}\int_0^1\int_K \lambda^{2s} [\lambda_*\pmb{\mathcal{G}}_{\mathcal{V}}](k,l,t)f(l)\frac{\rd \lambda}{\lambda},
\end{align}
for $\nu=s\alpha$ and $\lambda=a^{-\alpha}$. We here use the notation $\pmb{\mathcal{G}}_{\mathcal{V}}:=\pmb{\mathcal{G}}_{\nu_{\mathcal{V}}}$. Recall also that $\pmb{\mathcal{G}}_{\mathcal{V}}$ is density valued, so integration in $l\in K$ is also carried out in \eqref{aldjnakdjna}. Here we are using the fact that $f$ transforms under $M$ with $\sigma$. In particular, for $\nu=s\alpha$, up to smoothing errors we have from Corollary \ref{lknlknad} that 
\begin{equation}
\label{blablathun}
L_\nu=\mathrm{ev}_{t=1}\hat{\pmb{k}}_{\pmb{\mathcal{G}}_{\mathcal{V}},\sigma}(2s)\in \Psi^{-2s}_H(P\backslash G;\mathbb{H}_\nu,\mathbb{H}_{-\nu}),
\end{equation}
where we arguing as in Remark \ref{nochi} can assume that $\chi$ in Corollary \ref{lknlknad} has $\chi=1$ at $t=1$.

\begin{prop}
\label{lnudadnadlp}
The meromorphic family of operators $(L_\nu)_{\nu \in \mathfrak{a}^*_\C}$ forms a meromorphic family of Heisenberg pseudodifferential operators with $L_{\nu}\in \Psi^{-2s}_H(P\backslash G; \mathbb{H}_\nu,\mathbb{H}_{-\nu})$ for $\nu=s\alpha$. In particular, the operator $L_\beta$ extends, for $\nu$ outside a discrete set, by continuity to a continuous linear operator
$$W^{-s}_{H}(P\backslash G, \mathbb{H}_{\nu})\to W^{s}_{H}(P\backslash G, \mathbb{H}_{-\nu}).$$
\end{prop}

\begin{proof}
The proposition follows directly from Corollary \ref{lknlknad}, Theorem \ref{bddsds} and Equation \eqref{blablathun}.
\end{proof}

\begin{lemma}
\label{lnudadnadlplemma}
The point $\nu_\mathcal{V}$ is a regular value for the family $(L_\nu)_{\nu \in \mathfrak{a}^*_\C}$ and it holds that $L_{\nu_{\mathcal{V}}}=S_{\mathcal{V}}^\dagger S_{\mathcal{V}}$.
\end{lemma}

\begin{proof}
We can write $L_{\nu_{\mathcal{V}}-\epsilon\alpha}=S_{\mathcal{V}}^\dagger \mathsf{a}^{2\epsilon}S_{\mathcal{V}}$. Here $\mathsf{a}$ is, by an abuse of notation, the bounded operator on $L^2(A^+K;\mathcal{V}, h)$ given by $\mathsf{a}f(ak)=a^{-\alpha}f(ak)$. We have that $L_{\nu_{\mathcal{V}}-\epsilon\alpha}$ is a meromorphic family in $\epsilon$ by Proposition \ref{lnudadnadlp}. The lemma now follows from Proposition \ref{lnudadnadlp} and that $L_{\nu_{\mathcal{V}}-\epsilon}$ converges strongly on $K$-finite vectors to $S_{\mathcal{V}}^\dagger S_{\mathcal{V}}$ as $\epsilon \to 0+$.
\end{proof}

Proposition \ref{lnudadnadlp} and Lemma \ref{lnudadnadlplemma} implies the following result proving the boundedness part of Theorem \ref{conja}.

\begin{cor}
\label{proofofthema}
We write $\nu_{\mathcal{V}}=s_{\mathcal{V}}\alpha$. The Szegö map $S_{\mathcal{V}}$ from Equation \eqref{kwszeg} extends by continuity to a continuous linear operator
$$S_{\mathcal{V}}: W^{-s_\mathcal{V}}_{H}(P\backslash G, \mathbb{H}_{\nu_{\mathcal{V}}})\to L^2(K \backslash G,\mathbb{V}),$$
with $S_{\mathcal{V}}^\dagger S_{\mathcal{V}}\in \Psi^{-2s_\mathcal{V}}_H(P\backslash G; \mathbb{H}_{\nu_{\mathcal{V}}},\mathbb{H}_{-\nu_{\mathcal{V}}})$. 
\end{cor}

\begin{proof}
We have that $S_{\mathcal{V}}^\dagger S_{\mathcal{V}}=L_{\nu_{\mathcal{V}}}\in \Psi^{-2s_\mathcal{V}}_H$ by Proposition \ref{lnudadnadlp} and Lemma \ref{lnudadnadlplemma}. To prove that $S_{\mathcal{V}}$ is bounded, by Proposition \ref{stararsta} it suffices to show that the, a priori densely defined, composition 
$$W^{-s_\mathcal{V}}_{H}(P\backslash G, \mathbb{H}_{\nu_{\mathcal{V}}})\xrightarrow{S_{\mathcal{V}}} L^2(K \backslash G,\mathbb{V})\xrightarrow{S_{\mathcal{V}}^\dagger }W^{s_\mathcal{V}}_{H}(P\backslash G, \mathbb{H}_{-\nu_{\mathcal{V}}}),$$
is bounded. Indeed, we have that 
$$\|S_{\mathcal{V}}\|_{W^{-s_\mathcal{V}}_{H}(P\backslash G, \mathbb{H}_{\nu_{\mathcal{V}}})\to L^2(K \backslash G,\mathbb{V})}^2=\|S_{\mathcal{V}}^\dagger S_{\mathcal{V}}\|_{W^{-s_\mathcal{V}}_{H}(P\backslash G, \mathbb{H}_{\nu_{\mathcal{V}}})\to W^{s_\mathcal{V}}_{H}(P\backslash G, \mathbb{H}_{-\nu_{\mathcal{V}}})},$$ 
since $W^{s_\mathcal{V}}_{H}(P\backslash G, \mathbb{H}_{-\nu_{\mathcal{V}}})$ is dual to $W^{-s_\mathcal{V}}_{H}(P\backslash G, \mathbb{H}_{\nu_{\mathcal{V}}})$ in the $L^2$-pairing. The boundedness of $S_{\mathcal{V}}^\dagger S_{\mathcal{V}}$ follows from Theorem \ref{bddsds}. 
\end{proof}

\subsection{Closed range of the Szegö map}
\label{lknlnad}

We now turn to proving that the Szegö map $S_{\mathcal{V}}$ has closed range and to compute the principal symbol of its source projection. Let us first make some further remarks about Knapp-Stein intertwiners, recall their definition from Subsection \ref{subsecks}. We fix a generator $w\in M'$ of the Weyl group such that $w^2=1$ and $w\sigma=\sigma$. In this case, we have that the family of Knapp-Stein intertwiners $(A_{\sigma,\nu})_{\nu\in \mathfrak{a}_\C^*}$ satisfies that 
\begin{equation} 
\label{someidedadadn}
A_{\sigma,-\nu}A_{\sigma,\nu}=c_\sigma(\nu), \quad\mbox{and}\quad A_{\sigma,\nu}^*=A_{\sigma,\bar{\nu}},
\end{equation}
where $c_\sigma$ is a meromorphic function and $\overline{z\alpha}:=\bar{z}\alpha$, for the simple restricted root $\alpha$ we have already fixed.

\begin{lemma}
\label{nkjbnadap}
There is a $k\in \Z$ and a non-zero constant $c_{\mathcal{V}}\in \R$ such that the Szegö map $S_{\mathcal{V}}$ from Equation \eqref{kwszeg} satisfies 
$$S_{\mathcal{V}}^\dagger S_{\mathcal{V}}=c_{\mathcal{V}}\lim_{\nu\to \nu_{\mathcal{V}}}(\nu-\nu_{\mathcal{V}})^kA_{\sigma,\nu}.$$
\end{lemma}

\begin{proof}
As in Corollary \ref{remarkkwszeggeneric}, we can more generally consider the $G$-equivariant Poisson transform 
$$
P_\nu:\mathcal{D}'(P\backslash G,\mathbb{H})\to C^\infty(K\backslash G,\mathbb{V}), \quad P_\nu f(g):=\int_K \tau(k)^{-1} f(kg)\rd k,
$$
and arguing as above for $\mathcal{L}_{\sigma,\nu}:=P_\nu^\dagger P_\nu$ we arrive at a meromorphic family $(\mathcal{L}_{\sigma,\nu})_{\nu\in \mathfrak{a}_\C^*}$ of operators $\mathcal{L}_{\sigma,\nu}\in \Psi^{-2\nu}_H(P\backslash G; \mathbb{H})$ that are $G$-equivariant as operators 
$$\mathcal{L}_{\sigma,\nu}:C^\infty(P\backslash G,\mathbb{H}_\nu)\to C^\infty(P\backslash G,\mathbb{H}_{-\nu}).$$
Since $S_{\mathcal{V}}$ is bounded by Corollary \ref{proofofthema}, it follows that $\nu_\mathcal{V}$ is a regular value of $\mathcal{L}_{\sigma,\nu}$ and so by Remark \ref{lknkanda} we have that $S_{\mathcal{V}}^\dagger S_{\mathcal{V}}=c_{\mathcal{V}}A_{\sigma,\nu_{\mathcal{V}}}$ for some constant $c_{\mathcal{V}}\in \C$ up to renormalization as $\nu\to \nu_{\mathcal{V}}$. If $c_{\mathcal{V}}=0$ then $S_{\mathcal{V}}=0$ which is a contradiction. Since both $S_{\mathcal{V}}^\dagger S_{\mathcal{V}}$ and $A_{\sigma,\nu}$ are symmetric on $L^2$ for real $\nu$, $c_{\mathcal{V}}$ is real.
\end{proof}

\begin{remark}
Alternatively, we can prove Lemma \ref{nkjbnadap} using the proof of \cite[Theorem 14.92]{knappbook} and Schur's lemma. Indeed, in the case at hand, \cite[Theorem 14.92]{knappbook} and its consequences discussed in Remark \ref{langlandsrem} implies that both a renormalized value of $A_{\sigma,\nu_\mathcal{V}}$ and $S_{\mathcal{V}}^\dagger S_{\mathcal{V}}$ are infinitesimal intertwiners from the source of $S_{\mathcal{V}}$ to its range. Lemma \ref{nkjbnadap} then follows from Schur's lemma since the source of $S_{\mathcal{V}}$ is equivalent to a discrete series representation.
\end{remark}

We write $P_{\mathcal{V}}$ for the orthogonal projection onto the orthogonal complement of $\ker S_{\mathcal{V}}$. We will implicitly assume that the $K$-action on $W^s_H(P\backslash G,\mathbb{H})$ is unitary. With Peter-Weyl's theorem, we obtain the isomorphism of $K$-representations 
\begin{equation}
\label{kdecom}
\bigoplus_{\pi\in \hat{K}}^{L^2}\mathcal{V}_\pi^{n_\pi}= L^2(P\backslash G,\mathbb{H}),
\end{equation}
where $\mathcal{V}_\pi$ runs over the irreducible, unitary $K$-representations and $n_\pi=\dim \Hom_M(\mathcal{V}_\pi,\mathcal{H})$. Moreover, the natural embedding of each summand $\mathcal{V}_\pi^{n_\pi}$ into $W^s_H(P\backslash G,\mathbb{H})$ is up to a scaling factor isometric for any $s$. We see from $K$-equivariance that $P_{\mathcal{V}}$ is the orthogonal projection onto the orthogonal complement of $\ker S_{\mathcal{V}}$ in any of the Sobolev spaces. Indeed, since  $P_{\mathcal{V}}$ is $K$-equivariant, it takes the form $\bigoplus_{\pi\in \hat{K}}^{L^2}\mathrm{id}_{\mathcal{V}_\pi}^{n'_\pi}$ in the $K$-type decomposition \eqref{kdecom} for some $n_\pi'\leq n_\pi$.

\begin{lemma}
\label{sourcedesc}
For the $k\in \Z$ of Lemma \ref{nkjbnadap}, the limit 
$$T_\mathcal{V}:=\lim_{\nu\to \nu_{\mathcal{V}}} c_\mathcal{V}^{-1}c_\sigma(\nu)^{-1}(\nu-\nu_{\mathcal{V}})^{-k}A_{\sigma,-\nu}P_\mathcal{V}.$$ 
exists as a non-zero operator on $C^\infty(P\backslash G,\mathbb{H})$ and forms a $G$-equivariant operator $T_\mathcal{V}\in \Psi^{2s_\mathcal{V}}_H(P\backslash G,\mathbb{H}_{-\nu_{\mathcal{V}}},\mathbb{H}_{\nu_{\mathcal{V}}})$. The operator $T_\mathcal{V}$ satisfies
$$T_\mathcal{V}S_{\mathcal{V}}^\dagger S_{\mathcal{V}}=P_\mathcal{V}.$$
In particular, $S_{\mathcal{V}}$ has closed range and its partial left inverse as an operator from its source to range is the continuous $G$-equivariant operator 
$$T_\mathcal{V}S_{\mathcal{V}}^\dagger:L^2(K\backslash G;\mathbb{V})\to W^{-s_\mathcal{V}}_H(P\backslash G; \mathbb{H}_{\nu_{\mathcal{V}}}).$$
\end{lemma}

Lemma \ref{sourcedesc} together with Corollary \ref{proofofthema} allow us to conclude the proof of Theorem \ref{conja}.

\begin{proof}
We have by Lemma \ref{nkjbnadap} that $B_\mathcal{V}(\nu):=c_\mathcal{V}(\nu-\nu_{\mathcal{V}})^kA_{\sigma,\nu}P_\mathcal{V}$ is regular at $\nu=\nu_{\mathcal{V}}$, with value $S_{\mathcal{V}}^\dagger S_{\mathcal{V}}$. Set $T_\mathcal{V}(\nu):= c_\mathcal{V}^{-1}c_\sigma(\nu)^{-1}(\nu-\nu_{\mathcal{V}})^{-k}A_{\sigma,-\nu}$. Since $A_{\sigma,-\nu}A_{\sigma,\nu}=c_\sigma(\nu)$,  it holds that
\begin{equation}
\label{kjnkjnadjnad}
T_\mathcal{V}(\nu)B_\mathcal{V}(\nu)=P_\mathcal{V}.
\end{equation}
Note that for real $\nu$, $B_\mathcal{V}(\nu)$ is self-adjoint by Equation \eqref{someidedadadn}. We see that 
$$T_\mathcal{V}:=\lim_{\nu\to \nu_{\mathcal{V}}} c_\mathcal{V}^{-1}c_\sigma(\nu)^{-1}(\nu-\nu_{\mathcal{V}})^{-k}A_{\sigma,-\nu}P_\mathcal{V},$$
exists in the strong sense on $K$-finite vectors. The form of $c_\sigma(\nu)^{-1}(\nu-\nu_{\mathcal{V}})^{-k}A_{\sigma,-\nu}$ given by Proposition \ref{lknlajndljadn} implies that $T_\mathcal{V}\in \Psi^{2s_\mathcal{V}}_H(P\backslash G,\mathbb{H}_{-\nu_{\mathcal{V}}},\mathbb{H}_{\nu_{\mathcal{V}}})P_\mathcal{V}$. In particular, $T_\mathcal{V}$ defines a continuous operator 
$$T_\mathcal{V}:W^{s_\mathcal{V}}_H(P\backslash G; \mathbb{H}_{-\nu_{\mathcal{V}}})\to W^{-s_\mathcal{V}}_H(P\backslash G; \mathbb{H}_{\nu_{\mathcal{V}}}).$$

It follows from Equation \eqref{kjnkjnadjnad} that $T_\mathcal{V}$ satisfies $T_\mathcal{V}S_{\mathcal{V}}^\dagger S_{\mathcal{V}}=P_\mathcal{V}$. By construction, $P_\mathcal{V}$ is the range projection of the continuous operator $T_{\mathcal{V}}$, with partial right inverse $S_{\mathcal{V}}^\dagger S_{\mathcal{V}}$. We see that   $S_{\mathcal{V}}^\dagger S_{\mathcal{V}}$ has closed range and by Lemma \ref{alknalkdn} and Corollary \ref{proofofthema} the source projection $P_\mathcal{V}$ of $S_{\mathcal{V}}^\dagger S_{\mathcal{V}}$ satisfies $P_{\mathcal{V}}\in \Psi^0_H(X;\mathbb{H}_{\nu_{\mathcal{V}}})$. It follows that $T_\mathcal{V}\in \Psi^{2s_\mathcal{V}}_H(P\backslash G,\mathbb{H}_{-\nu_{\mathcal{V}}},\mathbb{H}_{\nu_{\mathcal{V}}})$. 
\end{proof}

\begin{prop}
\label{nljnkjlnadkbiybg}
The Knapp-Stein intertwiners $(A_{\sigma,\nu})_{\nu \in \mathfrak{a}^*_\C}$ forms a meromorphic family of $G$-equivariant Heisenberg pseudodifferential operators $A_{\sigma,\nu}\in \Psi^{-2s}_H(P\backslash G; \mathbb{H}_\nu,\mathbb{H}_{-\nu})$ for $\nu=s\alpha$. Moreover, the principal symbol of $A_{\sigma,\nu}$ is determined by $G$-equivariance via Proposition \ref{gequicad} and its value in the fiber over $Pe\in P\backslash G$ where it is given by 
$$\sigma_H^{2s}(A_{\sigma,\nu})(Pe,\cdot)=\mathfrak{a}_{\sigma,\nu}+P^{-2s}V\in \Sigma_H^{-2s}V,$$
where $\mathfrak{a}_{\sigma,\nu}$ is given in Proposition \ref{lknlajndljadn}.
\end{prop}

\begin{proof}
The Knapp-Stein intertwiner $A_{\sigma,\nu}$ is $G$-equivariant, so it is enough to prove the proposition in the nilpotent chart $V\hookrightarrow P\backslash G$. In this nilpotent chart, $A_{\sigma,\nu}$ is given by $A_{\sigma,\nu}^{(0)}$ and the proposition then follows from Proposition \ref{lknlajndljadn} and the kernel characterization of the Heisenberg calculus in Lemma \ref{lemkernel}.
\end{proof}

Recall the definition of the Fock bundle $\mathcal{F}\to S(\mathfrak{g}_{2\alpha})$ and its Sobolev versions $\mathcal{F}^s$ from Subsection \ref{subsecclosed}. The next result is a direct consequence of Theorem \ref{conja}, Lemma \ref{alknalkdn}, Lemma \ref{nkjbnadap} and Proposition \ref{nljnkjlnadkbiybg}.

\begin{cor}
\label{blablaks}
The principal symbol of $S_{\mathcal{V}}^\dagger S_{\mathcal{V}}\in \Psi^{-2s_\mathcal{V}}_H(P\backslash G; \mathbb{H}_{\nu_{\mathcal{V}}},\mathbb{H}_{-\nu_{\mathcal{V}}})$ is determined by $G$-equivariance via Proposition \ref{gequicad} and its value in the fiber over $Pe\in P\backslash G$ where it is given by 
$$\sigma_H^{-2s_\mathcal{V}}(S_{\mathcal{V}}^\dagger S_{\mathcal{V}})(Pe,\cdot)=c_{\mathcal{V}}\mathfrak{a}_{\sigma,\nu_{\mathcal{V}}}+P^{-2s_\mathcal{V}}V\in \Sigma_H^{-2s_{\mathcal{V}}}V,$$
where $\mathfrak{a}_{\sigma,\nu}$ is given in Proposition \ref{lknlajndljadn} and $c_{\mathcal{V}}$ is from Lemma \ref{nkjbnadap}. In particular, the principal symbol of the source projection $P_{\mathcal{V}}\in \Psi^0_H(X;\mathbb{H}_{\nu_{\mathcal{V}}})$ of $S_{\mathcal{V}}$ is determined by $G$-equivariance and its value $p_{\mathcal{V}}$ in the fiber over $Pe\in P\backslash G$ viewed as the morphism 
$$p_{\mathcal{V}}:\mathcal{F}\otimes \mathcal{H}\to \mathcal{F}\otimes \mathcal{H},$$
of bundle of Hilbert spaces over $S(\mathfrak{g}_{2\alpha})$, which is given by the source projection of the morphism
$$\mathfrak{a}_{\sigma,\nu_{\mathcal{V}}}:\mathcal{F}\otimes \mathcal{H}\to \mathcal{F}^{2s_\mathcal{V}}\otimes \mathcal{H},$$
where $\mathfrak{a}_{\sigma,\nu}$ is given in Proposition \ref{lknlajndljadn}.
\end{cor}

\section{Commutator properties of the Szegö map}
\label{commsec}

We now turn to proving Theorem \ref{conjb}. We use the Heisenberg calculus of van Erp-Yuncken \cite{vanerpyunck} reviewed above in Section \ref{subsecparatangb}, more specifically using the extended parabolic tangent groupoid $\hat{\T}_H(P\backslash G)$ from Subsection \ref{subsecparatang} combined with Ewert's approach to the Heisenberg calculus recalled in Subsection \ref{lknlknaddowow}. This is an appropriate approach only in real rank one, so we continue working under the assumption that $G$ has real rank one throughout this section. 

\begin{prop}
\label{ajnfkajn}
Write $\nu_{\mathcal{V}}=s_{\mathcal{V}}\alpha$. Theorem \ref{conjb} follows from the following two statements:
\begin{enumerate}
\item For any $b\in C^\infty(P\backslash G)=C^\infty(K)^M$, the operator 
$$\mathfrak{C}^b:=[S_{\mathcal{V}},b]^\dagger[S_{\mathcal{V}},b]:W^{-s_\mathcal{V}}_{H}(P\backslash G, \mathbb{H}_{\nu_{\mathcal{V}}})\to W^{s_\mathcal{V}}_{H}(P\backslash G, \mathbb{H}_{-\nu_{\mathcal{V}}}),$$
is a Heisenberg pseudodifferential operator of order $-2s_\mathcal{V}$ with $\sigma_H^{-2s_\mathcal{V}}(\mathfrak{C}^b)=0$.
\item For any $b\in C^\infty_0(A^+K)^M$, the operator 
$$\hat{\mathfrak{C}}^b:=S_{\mathcal{V}}^\dagger b S_{\mathcal{V}}:W^{-s_\mathcal{V}}_{H}(P\backslash G, \mathbb{H}_{\nu_{\mathcal{V}}})\to W^{s_\mathcal{V}}_{H}(P\backslash G, \mathbb{H}_{-\nu_{\mathcal{V}}}),$$
is a Heisenberg pseudodifferential operator of order $-2s_\mathcal{V}$ with $\sigma_H^{-2s_\mathcal{V}}(\hat{\mathfrak{C}}^b)=0$.
\end{enumerate}
\end{prop}

\begin{proof}
Take a $b\in C^\infty(Z)$ as in Theorem \ref{conjb}. Using Proposition \ref{dedcom}, we decompose 
$$b(ak)=b_0(k)+b_1(ak)+\chi(ak),$$
where $b_0\in C^\infty(K)^M$, $b_1\in C^\infty_0(A^+K)^M$ and $\chi\in L^\infty_c(K\backslash G)$. We have that
$$[S_{\mathcal{V}},b]=[S_{\mathcal{V}},b_0]+[S_{\mathcal{V}},b_1]+[S_{\mathcal{V}},\chi]=[S_{\mathcal{V}},b_0]-b_1S_{\mathcal{V}}-\chi S_{\mathcal{V}}.$$

For $\tilde{\chi}\in C^\infty_c(K\backslash G)$ such that $\tilde{\chi}\chi=\chi$, the operator $\chi S_{\mathcal{V}}$ is compact as soon as $\tilde{\chi}S_{\mathcal{V}}$ is. We can therefore assume that $\chi\in C^\infty_c(K\backslash G)$. Since the range of $S_{\mathcal{V}}$ is the kernel of an elliptic operator, elliptic estimates ensures that $\mathrm{Ran}(\chi S_{\mathcal{V}})\subseteq C^\infty_c(K\backslash G,\mathbb{V})$ so $\chi S_{\mathcal{V}}$ is compact. As such, $[S_{\mathcal{V}},b]$ is compact if $[S_{\mathcal{V}},b_0]$ and $b_1S_{\mathcal{V}}$ are compact as operators $W^{-s_\mathcal{V}}_{H}(P\backslash G, \mathbb{H}_{\nu_{\mathcal{V}}})\to L^2(K \backslash G,\mathbb{V})$. Item 1) and  Theorem \ref{bddsds} ensures that $[S_{\mathcal{V}},b_0]$ is compact for $b_0\in C^\infty(K)^M$. Item 2) and Theorem \ref{bddsds} ensures that $b_1S_{\mathcal{V}}$ is compact for $b_1\in C^\infty_0(A^+K)^M$.
\end{proof}

We now proceed to prove that item 1) and 2) of Proposition \ref{ajnfkajn} hold, thereby proving Theorem \ref{conjb}. 

\begin{lemma}
Item 1) of Proposition \ref{ajnfkajn} holds true.
\end{lemma}

\begin{proof}
For any $b\in C^\infty(K)^M$, we can write 
$$\mathfrak{C}^b:=[S_{\mathcal{V}},b]^\dagger[S_{\mathcal{V}},b],$$
as the evaluation at $\nu=\nu_{\mathcal{V}}$ of
\begin{align*}
\mathfrak{C}^b_\nu f(k):=\int_{A^+}\int_{K}\int_K a^{-2\nu}  \mathcal{G}_V^*(l_1k^{-1},a)&\mathcal{G}_V(l_2l_1^{-1},a)\times\\
&(b^*(k)-b^*(l_1))(b(l_2)-b(l_1))f(l_2)h(a)\rd l_1\rd l_2\rd a.
\end{align*}
We introduce the notation 
$$\pmb{G}_V^b(k,l,t)=\mathcal{G}_V(lk^{-1},t)(b(l)-b(k))=\pmb{G}_V(k,l,t)(b(l)-b(k)),$$
that by Lemma \ref{lknlnadarrr} defines a smooth function on $\hat{\T}_H(P\backslash G)$. Arguing as in Subsection \ref{boundviaext}, we take a $\chi\in C^\infty_c(\overline{A})$ being identically $1$ on $A^+$, and form the element
\begin{equation}
\label{lklnadkjbihbad}
\pmb{\mathcal{G}}_V^b:=\chi\tilde{h}(\pmb{G}_V^b)^**\pmb{G}_V^b \in C^\infty_c(\hat{\T}_H(P\backslash G)),
\end{equation}
where the convolution is in the extended parabolic tangent groupoid $\hat{\T}_H(P\backslash G)$. We can therefore write the operator $\mathfrak{C}^b_\nu$ as acting on $f\in C^\infty(P\backslash G,\mathbb{H}_\nu)=C^\infty(K,H)^M$ via
$$\mathfrak{C}_\nu^bf(k)=\mathrm{ev}_{t=1}\int_{A^+K} a^{-2\nu} [a_*\pmb{\mathcal{G}}^b_V](k,l,t)f(l)\rd l \rd a.$$
In particular, up to smoothing errors we have that 
$$\mathfrak{C}^b_\nu=\mathrm{ev}_{t=1}\hat{\pmb{k}}_{\pmb{\mathcal{G}}_V^b,\sigma}(\nu),$$
where we arguing as in Remark \ref{nochi} can assume that $\chi$ in Corollary \ref{lknlknad} has $\chi=1$ at $t=1$. We can therefore directly from Corollary \ref{lknlknad} conclude that $\mathfrak{C}_\nu^b\in \Psi^{-2s}_H(P\backslash G;\mathbb{H}_\nu,\mathbb{H}_{-\nu})$ for $\nu=s\alpha$.

To show that $\sigma_H^{-2s_\mathcal{V}}(\mathfrak{C}^b)=0$, it will by Remark \ref{ljknkjlnadbd} suffice to show that $\pmb{\mathcal{G}}_V^b(x,v,0)=0$. Evaluating at $t=0$ is induced from a groupoid homomorphism, so by \eqref{lklnadkjbihbad} it suffices to show that $\pmb{G}_V^b(x,v,0)=0$. Following the notation in the proof of Lemma \ref{lknlnadarrr}, we have that 
$$(\hat{\psi}^{-1})^*\pmb{G}_V^b(x,v,t)=\mathfrak{G}_V(\tilde{\kappa}(v,t))(b(l)-b(l\kappa(tvt^{-1}))).$$
Since $b\in C^\infty(K)^M$ we have $b(l\kappa(tvt^{-1}))\to b(l)$ as $t\to 0$. In particular, Lemma \ref{lknlnadarrr} implies that 
$$\hat{\pmb{G}}_V^b(x,v,0)=\mathfrak{G}_V(\kappa(v))(b(l)-b(l))=0.$$
We conclude that $\sigma_H^{-2s_\mathcal{V}}(\hat{\mathfrak{C}}^b)=0$.
\end{proof}

\begin{lemma}
Item 2) of Proposition \ref{ajnfkajn} holds true.
\end{lemma}

\begin{proof}
For any $b\in C^\infty_0(A^+K)^M$, we can write 
$$\hat{\mathfrak{C}}^b:=S_{\mathcal{V}}^\dagger b S_{\mathcal{V}},$$
as the evaluation at $\nu=\nu_{\mathcal{V}}$ of
$$\mathfrak{C}^b_\nu f(k):=\int_{A^+}\int_{K}\int_K a^{-2\nu}  \mathcal{G}_V^*(l_1k^{-1},a)\mathcal{G}_V(l_2l_1^{-1},a)b(al_1)f(l_2)h(a)\rd l_1\rd l_2\rd a.$$
We introduce the notation 
$$\hat{\pmb{G}}_V^b(k,l,t)=\mathcal{G}_V(lk^{-1},t)b(tk)=\pmb{G}_V(k,l,t)b(tk),$$
that by Lemma \ref{lknlnadarrr} defines a smooth function on $\hat{\T}_H(P\backslash G)$. Arguing as above, from a $\chi\in C^\infty_c(\overline{A})$ being identically $1$ on $A^+$, we form the element
\begin{equation}
\label{lklnadkjbihbad2}
\hat{\pmb{\mathcal{G}}}_V^b:=\chi\tilde{h}\pmb{G}_V^**\hat{\pmb{G}}_V^b \in C^\infty_c(\hat{\T}_H(P\backslash G)),
\end{equation}
where the convolution is in the extended parabolic tangent groupoid $\hat{\T}_H(P\backslash G)$. We can therefore write the operator $\hat{\mathfrak{C}}^b$ as acting on $f\in C^\infty(P\backslash G,\mathbb{H}_\nu)=C^\infty(K,H)^M$ via
$$\hat{\mathfrak{C}}^b_\nu f(k)=\mathrm{ev}_{t=1}\int_{A^+K} a^{-2\nu} [a_*\hat{\pmb{\mathcal{G}}}^b_V](k,l,t)f(l)\rd l \rd a.$$
In particular, up to smoothing errors we have that 
$$\hat{\mathfrak{C}}_\nu^b=\mathrm{ev}_{t=1}\hat{\pmb{k}}_{\hat{\pmb{\mathcal{G}},\sigma}_V^b}(\nu).$$
We can therefore directly from Corollary \ref{lknlknad} conclude that $\hat{\mathfrak{C}}^b_\nu\in \Psi^{-2s}_H(P\backslash G;\mathbb{H}_\nu,\mathbb{H}_{-\nu})$ for $\nu=s\alpha$. 

To show that $\sigma_H^{-2s_\mathcal{V}}(\hat{\mathfrak{C}}^b)=0$, it will by Remark \ref{ljknkjlnadbd} suffice to show that $\hat{\pmb{\mathcal{G}}}_V^b(x,v,0)=0$. Evaluating at $t=0$ is induced from a groupoid homomorphism, so by \eqref{lklnadkjbihbad2} it suffices to show that $\hat{\pmb{G}}_V^b(x,v,0)=0$. Following the notation in the proof of Lemma \ref{lknlnadarrr}, we have that 
$$(\hat{\psi}^{-1})^*\hat{\pmb{G}}_V^b(x,v,t)=\mathfrak{G}_V(\tilde{\kappa}(v,t))b(tx).$$
In particular, Lemma \ref{lknlnadarrr} implies that $\hat{\pmb{G}}_V^b(x,v,0)=\mathfrak{G}_V(\kappa(v))b(0x)$. Since $b\in C^\infty_0(A^+K)^M$ we have $b(0x)=0$ so $\hat{\pmb{G}}_V^b(x,v,0)=0$ and $\sigma_H^{-2s_\mathcal{V}}(\hat{\mathfrak{C}}^b)=0$.
\end{proof}

\section{Examples}

Let us consider two examples to illustrate Theorem \ref{conja} and Theorem \ref{conjb}, the holomorphic discrete series of $SU(n,1)$ and the quaternionic discrete series of $Sp(n,1)$. In both cases, we use the ball model for $K\backslash G$ in which $P\backslash G$ identifies with its boundary sphere. 

\subsection{The holomorphic discrete series of $SU(n,1)$}

We consider $G=SU(n,1)$. In this case, $K=U(n)$ is realized via $U(n)\cong S(U(n)\times U(1))$ as $U\oplus u\in U(n)\times U(1)$ such that $\det(U)=u^{-1}$. We write $B\subseteq \C^n$ for the complex unit ball, on which $G$ acts via Möbius transforms 
$$\begin{pmatrix} a& b\\c&d\end{pmatrix}.z:=\frac{az+b}{cz+d}.$$
The $G$-equivariant map $G\to B$, $g\mapsto g.0$ identifies $G/K$ with the open unit ball $B\subseteq \C^n$. We write 
$$J(g,z)=cz+d, \quad\mbox{for}\quad g=\begin{pmatrix} a& b\\c&d\end{pmatrix}.$$ 
For $l\in \N$, we consider the probability measure on $B$ given by 
$$\rd\mu_l=c_l(1-|z|^2)^{l-n-1} \rd V,$$
where $c_l$ is a normalizing constant. The holomorphic discrete series of $SU(n,1)$ depends on a parameter $l\in \N_{> n}$ and is realized on the weighted Bergman space 
$$A^2_l(B):=\{f\in L^2(B,\mu_l): \; f\; \mbox{holomorphic}\}.$$
The group $G$ acts on $A^2_l(B)$ via 
$$\pi_l(g)f(z)=J(g^{-1},z)^{-l}f(g^{-1}z).$$
The fact that $\pi_l$ is a representation follows from the cocycle identity 
$$J(g_1g_2,z)=J(g_1,z)J(g_2,g_1z),\quad g_1,g_2\in G, \; z\in B.$$
The representation $(A^2_l(B),\pi_l)$ is a discrete series representation whose lowest $K$-type is the constant function $1\in A^2_l(B)$ and on which $K$ acts as $\pi_l(U,u)1=u^l1$. By taking Taylor series, we can decompose $A^2_l(B)$ into $K$-types as
$$A^2_l(B)=\oplus_{k=0}^\infty \mathcal{P}_k,$$
where $P_k$ is spanned by all homogeneous polynomials of degree $k$. In particular, $P_k$ has an orthogonal basis given by the monomials $(z^\alpha)_{\alpha\in \N^n, |\alpha|=k}$. It is a well known computation that
\begin{equation}
\label{a2norm}
\|z^\alpha\|_{A^2_l(B)}\sim \sqrt{\frac{\alpha!}{(|\alpha|+l-1)!}},\quad \mbox{as $|\alpha|\to \infty$},
\end{equation}
and 
\begin{equation}
\label{l2normsp}
\|z^\alpha\|_{L^2(S^{2n-1})}\sim \sqrt{\frac{\alpha!}{(|\alpha|+n-1)!}},\quad \mbox{as $|\alpha|\to \infty$}.
\end{equation}
Therefore, we see from a short computation that for $s\in \R$,
\begin{equation}
\label{wshnorm}
\|z^\alpha\|_{W^s_H(S^{2n-1})}\sim |\alpha|^s\sqrt{\frac{\alpha!}{(|\alpha|+n-1)!}},\quad \mbox{as $|\alpha|\to \infty$}.
\end{equation}

Consider the Szegö maps $S_l:C^\infty(S^{2n-1})\to C^\infty(\overline{B})\cap \mathcal{O}(B)$ given by 
\begin{equation}
\label{szegosun}
S_lf(z):=\frac{1}{(2\pi i)^n}\int_{S^{2n-1}} \frac{f(\zeta)\rd S(\zeta)}{(1-\bar{\zeta}\cdot z)^l}.
\end{equation}
We shall see below that $S_l$ is Knapp-Wallach's Szegö map onto the discrete series representation $(A^2_l(B),\pi_l)$ for an appropriate $G$-action on $C^\infty(S^{2n-1})$. 

The Szegö map $S_l$ is in fact well defined for $l\in \N_{\geq n}$, where $l=n$ is rather a limit of discrete series representations. We write $S:=S_n$ for the classical Szegö map. Writing $P_S:L^2(S^{2n-1})\to H^2(S^{2n-1})$ for the Szegö projection onto the Hardy space 
$$H^2(S^{2n-1}):=\{f\in L^2(S^{2n-1}): \; f\; \mbox{extends holomorphically to $B$ }\},$$ 
one has $P_Sf(z)=\lim_{r\to 1-} Sf(rz)$. A short computation in coordinates shows the well known fact that $P_S\in \Psi^0_H(S^{2n-1})$ for the standard contact structure on $S^{2n-1}$. 

Now we show the equivariance of $S_l$. Using the elementary identity 
$$1-g.z\cdot \overline{g.\zeta}=\overline{J(g,\zeta)}^{-1}(1-z\cdot \bar\zeta)J(g,z)^{-1},$$
we conclude that 
\begin{align*}
S_lf(g.z)=&\frac{1}{(2\pi i)^n}\int_{S^{2n-1}} \frac{f(\zeta)\rd S(\zeta)}{(1-\bar{\zeta}\cdot g.z)^l}=\frac{1}{(2\pi i)^n}\int_{S^{2n-1}} \frac{f(g.\zeta)\rd S(g.\zeta)}{(1-\overline{g.\zeta}\cdot g.z)^l}=\\
=&\frac{J(g,z)^l}{(2\pi i)^n}\int_{S^{2n-1}} \frac{1}{(1-\overline{\zeta}\cdot z)^l}\overline{J(g,\zeta)}^{l}f(g.\zeta)\rd S(g.\zeta)=\\
=&\frac{J(g,z)^l}{(2\pi i)^n}\int_{S^{2n-1}} \frac{1}{(1-\overline{\zeta}\cdot z)^l}|J(g,\zeta)|^{-2n}\overline{J(g,\zeta)}^{l}f(g.\zeta)\rd S(\zeta).
\end{align*}
In particular, if we define the representation $\tilde{\pi}_l$ on $C^\infty(S^{2n-1})$ as
$$\tilde{\pi}_l(g)f(z):=|J(g^{-1},z)|^{-2n}\overline{J(g^{-1},z)}^{l}f(g^{-1}.z),$$
then $S_l:(C^\infty(S^{2n-1}),\tilde{\pi}_l)\to (L^2(B,\mu_l),\pi_l)$ is $G$-equivariant. From Taylor expansion and the multinomial theorem, we see that for $|z|<1$ and $|\zeta|=1$ it holds that 
\begin{equation}
\label{kernelexap}
\frac{1}{(1-\bar{\zeta}\cdot z)^l}=\sum_{k=0}^\infty \frac{l(l+1)\cdots (l+k-1)}{k!} (\bar{\zeta}\cdot z)^k=\sum_{\alpha\in \N^n} \frac{l(l+1)\cdots (l+|\alpha|-1)}{\alpha!}z^\alpha\bar{\zeta}^\alpha. 
\end{equation}
Therefore, the range of $S_l$ is contained in $(A^2_l(B),\pi_l)$, and we arrive at a $G$-equivariant map 
$$S_l:(C^\infty(S^{2n-1}),\tilde{\pi}_l)\to (A^2_l(B),\pi_l).$$
Since $(C^\infty(S^{2n-1}),\tilde{\pi}_l)$ is a principal series representation, it follows from Schur's lemma and irreducibility that $S_l$ is Knapp-Wallach's Szegö map.

We can in fact directly prove Theorem \ref{conja} for holomorphic discrete series representations of $SU(n,1)$ using the explicit form of $S_l$. Note that Equation \eqref{kernelexap} implies that, for any $s\in \R$, the range projection of $S_l$ in $W^s_H(S^{2n-1})$ is $P_S$ because $(z^\alpha)_{\alpha\in \N^n}$ is a complete orthogonal set in $P_SW^s_H(S^{2n-1})$. We compute from Equations \eqref{l2normsp} and \eqref{kernelexap} that 
$$[S_l\zeta^\alpha](z)=c(\alpha)z^\alpha,$$
where
$$c(\alpha)=\frac{l(l+1)\cdots (l+|\alpha|-1)}{\alpha!}\|\zeta^\alpha\|^2_{L^2(S^{2n-1})}\sim \frac{l(l+1)\cdots (l+|\alpha|-1)}{(n+|\alpha|-1)!}\sim |\alpha|^{l-n}.$$
We see from Equations \eqref{a2norm} and \eqref{wshnorm} that 
\begin{align*}
\|S_l\zeta^\alpha\|_{A^2_l(B)}\sim& |\alpha|^{l-n}\sqrt{\frac{\alpha!}{(|\alpha|+l-1)!}} \sim\\
\sim & |\alpha|^{l-n-s}\sqrt{\frac{(|\alpha|+n-1)!}{(|\alpha|+l-1)!}} \|\zeta^\alpha\|_{W^s_H(S^{2n-1})}\sim |\alpha|^{\frac{l-n}{2}-s} \|\zeta^\alpha\|_{W^s_H(S^{2n-1})}.
\end{align*}
Since $(z^\alpha)_{\alpha\in \N^n}$ is a complete orthogonal set in $A^2_l(B)$, we conclude that for $f\in C^\infty(S^{2n-1})$,
\begin{equation}
\label{lakndaldn}
\|S_lf\|_{A^2_l(B)}\sim \|P_Sf\|_{W^{s_l}_H(S^{2n-1})}, \quad\mbox{for}\quad s_l=\frac{l-n}{2}=\frac{l}{2}-\frac{p+2q}{4}.
\end{equation}
This corresponds to \eqref{nuvforea}. In particular, we conclude the following paraphrasing of Theorem \ref{conja}.

\begin{thm}
Let $l\in \N_{> n}$ and $s_l:=\frac{l-n}{2}$. The Szegö map \eqref{szegosun} extends to an $SU(n,1)$-equivariant bounded operator with closed range
$$S_l:W^{s_l}_H(S^{2n-1})\to L^2(B,\mu_l).$$
The range of $S_l$ is $A^2_l(B)$, and its source space is 
\begin{align*}
W^{s_l}_{H,{\rm hol}}(S^{2n-1}):=&P_SW^{s_l}_{H}(S^{2n-1})=\\
=&\{f\in W^{s_l}_H(S^{2n-1}):\; f\; \mbox{extends holomorphically to $B$}\}.
\end{align*}
\end{thm}

\subsection{The quaternionic discrete series of $Sp(n,1)$}

We now consider the group $G=Sp(n,1)$, consisting of $(n+1)\times (n+1)$-matrices over the quaternions $\mathbb{H}$ that preserve the sesquilinear form of signature $(n,1)$. In this case $K=Sp(n)\times Sp(1)$. We write $B\subseteq \mathbb{H}^n$ for the quaternionic unit ball, on which $G$ acts via Möbius transforms 
$$\begin{pmatrix} a& b\\c&d\end{pmatrix}.q:=\frac{aq+b}{cq+d}.$$
The $G$-equivariant map $G\to B$, $g\mapsto g.0$ identifies $G/K$ with the open unit ball $B\subseteq \mathbb{H}^n$. We write 
$$J(g,q)=cq+d, \quad\mbox{for}\quad g=\begin{pmatrix} a& b\\c&d\end{pmatrix}.$$ 
For $l\in \N$, we write $(\tau_l,s^l)$ for the $l+1$-dimensional $K$-representation factor over the projection map $K\to Sp(1)=SU(2)$ and consider standard unitary representation of the latter group on the space of homogeneous polynomials in two variables of degree $l$. We sometimes identify $\tau_l$ with a representation of $Sp(1)$ and extend $\tau_l$ in the obvious way $\tau_l(q)=|q|^l\tau_l(q|q|^{-1})$ to $\mathbb{H}^\times =\R^+ Sp(1)$. 

For $l\in \N$, we consider the probability measure on $B$ given by 
$$\rd\mu_l=c_l(1-|q|^2)^{l-2n} \rd V,$$
where $c_l$ is a normalizing constant. There is a unitary representation $\pi_l$ of $G$ on $L^2(B,\mu_l,s^l)$ given by  
$$\pi_l(g)f(z)=|J(g^{-1},z)|^{-2}\tau_l(J(g^{-1},z)^{-1})f(g^{-1}z).$$
We here follow Liu-Zhang \cite{liuzhang} and their explicit realization of the quaternionic discrete series representation. The quaternionic discrete series of semisimple Lie groups were introduced by Gross-Wallach \cite{grosswallach}, and  for $Sp(n,1)$ the quaternionic discrete series $(\mathcal{H}_l,\pi_l)$  depends on a parameter $l\in \N_{\geq 2n}$. It is realized as a particular subspace
$$\mathcal{H}_l\subseteq L^2(B,\mu_l,s^l)$$
of functions that are regular in the sense of \cite[Section 2.3]{liuzhang}. The lowest $K$-type in $\mathcal{H}_l$ is $s^l$. We know from \cite{grosswallach,liuzhang} that $\mathcal{H}_l$ decomposes into $K$-types with simple multiplicities as 
$$\mathcal{H}_l=\bigoplus_{m\geq 0} H^{(m,0)}\otimes s^{l+m},$$
for a specific $Sp(n)$-representation $H^{(m,0)}$ described in \cite[Proposition 4.3]{liuzhang}. For the purposes of Theorem \ref{conja}, we do not need a more detailed description of $\mathcal{H}_l$. 

We define the Szegö maps $S_l:C^\infty(S^{4n-1},s^l)\to  L^2(B,\mu_l,s^l)$ by 
$$S_lf(p):=\int_{S^{4n-1}} |1-q^*\cdot p|^{-2}\tau_l(1-q^*\cdot p)^{-1} f(q)\rd S(q).$$
An elementary computation shows that 
$$1-(g.q)^*\cdot (g.p)=(J(g,q)^*)^{-1}(1-q^*\cdot p)J(g,p)^{-1}.$$
We conclude that 
\begin{align*}
S_lf(g.z)=&\int_{S^{4n-1}} |1-q^*\cdot(g.p)|^{-2}\tau_l(1-q^*\cdot (g.p))^{-1} f(q)\rd S(q)=\\
=&\int_{S^{4n-1}} |1-(g.q)^*\cdot(g.p)|^{-2}\tau_l(1-(g.q)^*\cdot (g.p))^{-1} f(g.q)\rd S(g.q)=\\
=&|J(g,p)|^{2}\tau_l(J(g,p))\int_{S^{4n-1}}  |1-q^*\cdot p|^{-2}\tau_l(1-q^*\cdot p)^{-1} |J(g,q)|^{2}\tau_l(J(g,q))^*f(g.q)\rd S(g.q)=\\
=&|J(g,p)|^{2}\tau_l(J(g,p))\int_{S^{4n-1}}  |1-q^*\cdot p|^{-2}\tau_l(1-q^*\cdot p)^{-1} |J(g,q)|^{2-4n}\tau_l(J(g,q))^*f(g.q)\rd S(q).
\end{align*}
In particular, if we define the representation $\tilde{\pi}_l$ on $C^\infty(S^{4n-1},s^l)$ as
$$\tilde{\pi}_l(g)f(z):=|J(g^{-1},q)|^{2-4n}\tau_l(J(g^{-1},q))^*f(g^{-1}.q),$$
then $S_l:(C^\infty(S^{4n-1},s^l),\tilde{\pi}_l)\to (L^2(B,\mu_l,s^l),\pi_l)$ is $G$-equivariant.

We now show that $S_l$ is Knapp-Wallach's Szegö map. By a Schur lemma argument, and the fact that $\tilde{\pi}_l$ is a principal series representation, it suffices to show that $S_l$ maps into $\mathcal{H}_l$. Since $\mathcal{H}_l$ is spanned by the regular functions into $s^l$ -- regular in the sense of \cite[Section 2.3]{liuzhang} -- it suffices to show that for a fixed $q\in S^{4n-1}$ the function $p\mapsto   |1-q^*\cdot p|^{-2}\tau_l(1-q^*\cdot p)^{-1}v$ is regular for $v\in s^l$. By covariance we can in fact assume that $q=(1,0\ldots, 0)\in B\subseteq \mathbb{H}^n$. The fact that $S_l$ maps into $\mathcal{H}_l$ then reduces to showing that $p=(p_1,\ldots, p_n)\mapsto |1-p_1|^{2}\tau_l(1-p_1)^{-1}v$ is regular for $v\in s^l$. This follows from \cite[Theorem 3.11]{liuzhang}, showing that $|1-p_1|^{2}\tau_l(1-p_1)^{-1}$ is a distributional limit of a reproducing kernel on a Hilbert space of regular functions. We conclude that 
$$S_l:(C^\infty(S^{4n-1},s^l),\tilde{\pi}_l)\to (\mathcal{H}_l,\pi_l),$$ defines a $G$-equivariant map and coincides with the Knapp-Wallach Szegö transform.

Write $P_l:L^2(S^{4n-1},s^l)\to H^2(S^{4n-1},s^l)$ for the Szegö projection onto the Hardy space $H^2(S^{4n-1},s^l)$ defined as the closure of the algebraic direct sum $\bigoplus_{m\geq 0}^{\rm alg} H^{(m,0)}\otimes s^{l+m}\subseteq L^2(S^{4n-1},s^l)$. We conclude the following corollary of Theorem \ref{conja}.

\begin{thm}
Let $l\in \N_{\geq 2n}$ and $s_l:=\frac{l-2n+2}{2}$. The Szegö map $S_l:C^\infty(S^{4n-1},s^l)\to C^\infty(B,s^l)$ extends to an $Sp(n,1)$-equivariant bounded operator with closed range
$$S_l:W^{s_l}_H(S^{4n-1},s^l)\to L^2(B,\mu_l,s^l).$$
The range of $S_l$ is $\mathcal{H}_l$, and its source space is 
$$W^{s_l}_{H,{\rm hol}}(S^{4n-1},s^l):= P_lW^{s_l}_H(S^{4n-1},s^l).$$ 
In fact, the source projector $P_l$ belongs to $\Psi^0_H(S^{4n-1},s^l)$.
\end{thm}

We remark that $s_l=\frac{l+2}{2}+\frac{p+2q}{4}$, analogous to Equation \eqref{lakndaldn} up to the different enumerations of the discrete series in the complex and quaternionic cases, cf. \cite[Footnote on page 2986]{liuzhang}.

\section{The Szegö map as a boundary class in $K$-homology}
\label{szkhom}

Let us now turn to the source projection $P_{\mathcal{V}}$ of the Szegö map $S_{\mathcal{V}}$. The source projection $P_{\mathcal{V}}$ can be written as $P_{\mathcal{V}}=\mathcal{V}_V^*\mathcal{V}_V$ where $\mathcal{V}_V$ comes from a polar decomposition
$$S_{\mathcal{V}}=\sqrt{S_{\mathcal{V}}S_{\mathcal{V}}^*}\mathcal{V}_V,$$
where the adjoint is formed with respect to the Hilbert space inner products on $L^2(P\backslash G, \mathbb{H})$ and $L^2(K\backslash G,\mathbb{V})$, respectively. Here we view $S_{\mathcal{V}}$ as a compact operator $L^2(P\backslash G, \mathbb{H})\to L^2(K\backslash G,\mathbb{V})$ factoring over the dense inclusion $L^2(P\backslash G, \mathbb{H})\hookrightarrow W^{-s_\mathcal{V}}_H(P\backslash G, \mathbb{H})$. The operator $\mathcal{V}_V:L^2(P\backslash G, \mathbb{H})\to L^2(K\backslash G,\mathbb{V})$ is a partial isometry with range being the discrete series representation $\overline{\mathrm{Ran}(S_{\mathcal{V}})}$ and source being $\overline{\mathrm{Ran}(S_{\mathcal{V}}^*)\cap L^2(P\backslash G, \mathbb{H})}$. We have 
$$P_{\mathcal{V}}=\mathcal{V}_V^*\mathcal{V}_V.$$

\begin{prop}
\label{knadkna}
The source projection $P_{\mathcal{V}}\in \Psi^0_{H}(P\backslash G,\mathbb{H})$ has a $G$-invariant principal symbol. In particular, $P_{\mathcal{V}}$ defines an odd $G$-equivariant $K$-homology cycle 
$$(L^2(P\backslash G,\mathbb{H}_0),2P_{\mathcal{V}}-1),$$ 
for $C(P\backslash G)$ and the unitary $G$-action on $L^2(P\backslash G,\mathbb{H}_0)$.
\end{prop}

\begin{proof}
It follows directly from Corollary \ref{blablaks} and the equivariance of the Knapp-Stein intertwiner that $\sigma^0_H(P_{\mathcal{V}})$ is $G$-invariant. The proposition follows.
\end{proof}

\begin{remark}
\label{laknlakndexp}
We note that in Proposition \ref{knadkna}, we are using the unitary action on $L^2(P\backslash G,\mathbb{H})$, i.e. the $G$-action on $L^2(P\backslash G,\mathbb{H}_\nu)$ made unitary by means of correcting with the Radon-Nikodym derivative of the action. We can equally well consider  $(L^2(P\backslash G,\mathbb{H}_\nu),2P_{\mathcal{V}}-1)$ with its non-unitary action as an odd non-unitarily $G$-equivariant $K$-homology cycle. For an appropriate length function $\ell$ on $G$, determined by $\ell (ank)=a^{\nu_0}$ for suitable $\nu_0$, the pair $(L^2(P\backslash G,\mathbb{H}_\nu),2P_{\mathcal{V}}-1)$ forms a cycle for the exponentially bounded $K$-homology group \cite{goffbgg}
$$K_1^{G,\ell}(P\backslash G):=KK_1^{G,\ell}(C(P\backslash G),\C).$$ 
Clearly, the class of $(L^2(P\backslash G,\mathbb{H}_\nu),2P_{\mathcal{V}}-1)$ is independent of $\nu$ within a strip in $\mathfrak{a}_\C^*$.
\end{remark}

We can also consider the range projection $\tilde{P}_V$ of $S_{\mathcal{V}}$. We clearly have that 
$$\tilde{P}_V=\mathcal{V}_V\mathcal{V}_V^*.$$
Write $L^\infty_0(K\backslash G)\subseteq L^\infty(K\backslash G)$ for the norm closure of the compactly supported elements. 

\begin{prop}
\label{lknlmndirac}
The operator $[\tilde{P}_V,b]$ is compact on $L^2(K\backslash G,\mathbb{V})$ for any 
$$b\in L^\infty_0(K\backslash G)+C^\infty(Z).$$ 
\end{prop}

\begin{proof}
Arguing as in the proof of Proposition \ref{ajnfkajn}, $\chi \tilde{P}_V$ and $\tilde{P}_V\chi$ are compact whenever $\chi\in L^\infty_c(K\backslash G)$. By a limiting argument, $\chi \tilde{P}_V$ and $\tilde{P}_V\chi$ are compact whenever $\chi\in L^\infty_0(K\backslash G)$. So we need only consider $b\in  C^\infty(Z)$.

We have that $\tilde{P}_V$ is the ON-projection on $L^2(K\backslash G,\mathbb{V})$ onto a discrete series representation. Therefore, after possible going to a double cover, there exists \cite{atiischmid,laffaicm,partha} a homogeneous spinor bundle $\mathbb{S}\to K\backslash G$, containing $\mathbb{V}$ as an equivariant sub-bundle, a $G$-equivariant Dirac operator $D_V$ on $L^2(K\backslash G,\mathbb{S})$ (viewed as having the domain being the graph closure of $C^\infty_c(K\backslash G,\mathbb{S})$) such that for any $b\in C^\infty(Z)$,
\begin{itemize}
\item $\mathrm{Ran}(\tilde{P}_V)=\ker(D_V)$;
\item $D_V^*=\overline{D_V^\dagger}$ has closed range;
\item $a\Dom(D_V)\subseteq \Dom(D_V)$ for $a\in C^\infty(Z)$;
\item $[D_V,b]:\Dom(D_V)\to L^2(K\backslash G,\mathbb{S})$ is compact and admits a bounded extension to $L^2(K\backslash G,\mathbb{S})$. 
\end{itemize}

Since $\mathrm{Ran}(\tilde{P}_V=\ker(D_V)$, $\tilde{P}_V$ is the orthogonal projection onto $\ker(D_V)$. Compactness of $[\tilde{P}_V,b]$ on $L^2(K\backslash G,\mathbb{V})$ for $b\in C^\infty(Z)$ then follows from \cite[Subsection 3.4]{fries} and the fact that $Z$ is a quotient of the Higson compactification of $K\backslash G$, so $D_V$ defines a relative spectral triple (as defined in \cite[Definition 2.8]{fries}) for the ideal inclusion $C^\infty_c(K\backslash G)\lhd C^\infty(Z)$. 
\end{proof}

We now consider the Dirac operator $D_V$ appearing in the proof of Proposition \ref{lknlmndirac}. The Dirac operator defines a $G$-equivariant $K$-homology class $[D_V]\in K_0^G(K\backslash G)$. We note here that $K\backslash G$ is of even dimension since $G$ admits discrete series representations. We consider the short $G$-equivariant exact sequence of $C^*$-algebras
$$0\to C_0(K\backslash G)\to C(Z)\to C(P\backslash G)\to 0,$$
which is equivariantly semi-split with splitting defined from the positive $G$-equivariant Poisson transform $P$. This semi-split sequence defines an odd equivariant $KK$-element $\partial\in KK_1^G(C(P\backslash G),C_0(K\backslash G))$. Taking Kasparov products defines a boundary mapping 
$$\partial\otimes_{C_0(K\backslash G)}-:K_0^G(K\backslash G)\to K_1^G(P\backslash G).$$

\begin{thm}
\label{lknlnadkhom}
We have that
$$\partial[D_V]=[(L^2(P\backslash G,\mathbb{H}_0),2P_{\mathcal{V}}-1)],$$
in $K_1^{G}(P\backslash G)$.
\end{thm}

\begin{proof}
It follows from \cite[Subsection 3.4]{fries} and the discussion in the proof of Proposition \ref{lknlmndirac} that $\partial[D_V]$ is represented by the equivariant Toeplitz operators defined from $\tilde{P}_V$. Since $\tilde{P}_V=\mathcal{V}_V\mathcal{V}_V^*$ and $P_{\mathcal{V}}=\mathcal{V}_V^*\mathcal{V}_V$, the two projections define the same class in $K_1^{G}(P\backslash G)$.
\end{proof}

\begin{remark}
A consequence of Theorem \ref{lknlnadkhom} is that $[(L^2(P\backslash G,\mathbb{H}_0),2P_{\mathcal{V}}-1)]\neq 0$. Indeed, we can apply the forgetful map $\mathrm{Res}^G_e$ forgetting the group action. Under this map, the construction of the discrete series in \cite{laffaicm} produces a non-zero integer $k_V$ with $\mathrm{Res}^G_e[D_V]=k_V[\slashed{D}]$ where $\slashed{D}$ is any Dirac operator on $K\backslash G$ defined from an irreducible Clifford module, so $[\slashed{D}]$ is a generator of $K_0(K\backslash G)\cong \Z$. By standard $K$-theoretical considerations, $\partial[\slashed{D}]\in K_1(P\backslash G)\cong \Z$ is also a generator so by naturality 
$$\mathrm{Res}^G_e[(L^2(P\backslash G,\mathbb{H}_0),2P_{\mathcal{V}}-1)]=k_V\partial [\slashed{D}]\in K_1(P\backslash G),$$ 
is non-zero. However, the equivariant class $[(L^2(P\backslash G,\mathbb{H}_0),2P_{\mathcal{V}}-1)]\in K_1^{G}(P\backslash G)$ remains elusive as Bott's index theorem on homogeneous spaces \cite{botthomo} obstructs non-trivial pairings with equivariant $K$-theory while the ranks of $K$ and $M$ differ.
\end{remark}

\end{document}